\documentclass[microtype]{gtpart}
\usepackage{amsmath,amssymb}
\usepackage{tikz}
\usepackage{pinlabel}
\usetikzlibrary{arrows.meta, angles, quotes, calc}
\usetikzlibrary{positioning}
\usepackage{amsthm}
\usepackage{comment}
\usepackage{tikz-3dplot}

\theoremstyle{definition}
\newtheorem{theorem}{Theorem}
\newtheorem*{theorem*}{Theorem}

\newtheorem{definition}[theorem]{Definition}
\newtheorem*{definition*}{Definition}
\renewcommand\proofname{Proof.}

\newtheorem{prop}[theorem]{Proposition}
\newtheorem{lemma}[theorem]{Lemma}
\newtheorem{cor}[theorem]{Corollary}
\newtheorem{example}[theorem]{Example}
\newtheorem{rem}[theorem]{Remark}
\newtheorem{notation}[theorem]{Notation}

\numberwithin{theorem}{section}
\numberwithin{equation}{section}

\newcommand{\dashededgea}[0]{
\begin{tikzpicture}[x=0.75pt,y=0.75pt,yscale=0.3,xscale=0.3, baseline=-3pt] 

\draw [dash pattern={on 4pt off 3pt}]  (10,0)--(90,0);

\draw  (0,0) circle (10);
\draw  (100,0) circle (10);
\end{tikzpicture}}

\newcommand{\dashededgeb}[0]{
\begin{tikzpicture}[x=0.75pt,y=0.75pt,yscale=0.3,xscale=0.3, baseline=-3pt] 

\draw [dash pattern={on 4pt off 3pt}]   (0,0)--(90,0);

\draw  [fill={rgb, 255:red, 0; green, 0; blue, 0 }  ,fill opacity=1 ] (0,0) circle (10);
\draw  [fill={rgb, 255:red, 0; green, 0; blue, 0 }  ,fill opacity=1 ] (100,0) circle (10);
\end{tikzpicture}}

\newcommand{\dashededgec}[0]{
\begin{tikzpicture}[x=0.75pt,y=0.75pt,yscale=0.3,xscale=0.3, baseline=-3pt] 

\draw [dash pattern={on 4pt off 3pt}]   (10,0)--(90,0);

\draw  (0,0) circle (10);
\draw  [fill={rgb, 255:red, 0; green, 0; blue, 0 }  ,fill opacity=1 ] (100,0) circle (10);
\end{tikzpicture}}

\newcommand{\solidedge}[0]{
\begin{tikzpicture}[x=0.75pt,y=0.75pt,yscale=0.3,xscale=0.3, baseline=-3pt] 

\draw   (0,0)--(90,0);

\draw  [fill={rgb, 255:red, 0; green, 0; blue, 0 }  ,fill opacity=1 ]  (0,0) circle (10);
\draw  [fill={rgb, 255:red, 0; green, 0; blue, 0 }  ,fill opacity=1 ] (100,0) circle (10);
\end{tikzpicture}}

\newcommand{\graphd}[0]{
\begin{tikzpicture}[x=1pt,y=1pt,yscale=-0.3,xscale=0.4,baseline=10pt, line width = 1pt]

\draw  [fill= {rgb, 255:red, 0; green, 0; blue, 0 }  ,fill opacity=1 ] (-50, -150) circle (5);
\node at (-65,-150) {$6$};
\draw  [fill= {rgb, 255:red, 0; green, 0; blue, 0 }  ,fill opacity=1 ] (-50, -50) circle (5);
\node at (-65,-50) {$5$};
\draw  [fill= {rgb, 255:red, 0; green, 0; blue, 0 }  ,fill opacity=1 ] (-50, 50) circle (5);
\node at (-65,50) {$4$};
\draw  [fill= {rgb, 255:red, 0; green, 0; blue, 0 }  ,fill opacity=1 ] (50, -150) circle (5);
\node at (65,-150) {$3$};
\draw [fill= {rgb, 255:red, 0; green, 0; blue, 0 }  ,fill opacity=1 ] (50, -50) circle (5);
\node at  (65,-50) {$2$};
\draw [fill= {rgb, 255:red, 0; green, 0; blue, 0 }  ,fill opacity=1 ] (50, 50) circle (5);
\node at  (65,50) {$1$};

\draw [-Stealth] (-50,50) -- (-50,-50);
\draw [-Stealth] (-50,-50) -- (-50,-150);

\draw [-Stealth] (50,50) -- (50,-50);
\draw [-Stealth] (50,-50) -- (50,-150);

\draw [-Stealth]  [dash pattern={on 4pt off 3pt}]  (50,50) -- (-50,50);
\draw [-Stealth] [dash pattern={on 4pt off 3pt}]  (50,-50) -- (-50,-50);
\draw [-Stealth]  [dash pattern={on 4pt off 3pt}]  (50,-150) -- (-50,-150);

\end{tikzpicture}}

\newcommand{\graphe}[0]{
\begin{tikzpicture}[x=1pt,y=1pt,yscale=-0.3,xscale=0.4,baseline=10pt, line width = 1pt]

\draw  [fill= {rgb, 255:red, 0; green, 0; blue, 0 }  ,fill opacity=1 ] (-50, -150) circle (5);
\node at (-65,-150) {$4$};
\draw  [fill= {rgb, 255:red, 0; green, 0; blue, 0 }  ,fill opacity=1 ] (-50, -50) circle (5);
\node at (-65,-50) {$5$};
\draw  [fill= {rgb, 255:red, 0; green, 0; blue, 0 }  ,fill opacity=1 ] (-50, 50) circle (5);
\node at (-65,50) {$6$};
\draw  [fill= {rgb, 255:red, 0; green, 0; blue, 0 }  ,fill opacity=1 ] (50, -150) circle (5);
\node at (65,-150) {$3$};
\draw [fill= {rgb, 255:red, 0; green, 0; blue, 0 }  ,fill opacity=1 ] (50, -50) circle (5);
\node at  (65,-50) {$2$};
\draw [fill= {rgb, 255:red, 0; green, 0; blue, 0 }  ,fill opacity=1 ] (50, 50) circle (5);
\node at  (65,50) {$1$};

\draw [Stealth-] (-50,50) -- (-50,-50);
\draw [Stealth-] (-50,-50) -- (-50,-150);

\draw [-Stealth] (50,50) -- (50,-50);
\draw [-Stealth] (50,-50) -- (50,-150);

\draw [-Stealth]  [dash pattern={on 4pt off 3pt}]  (50,50) -- (-50,-50);
\draw [-Stealth] [dash pattern={on 4pt off 3pt}]  (50,-50) -- (-50,50);
\draw [-Stealth]  [dash pattern={on 4pt off 3pt}]  (50,-150) -- (-50,-150);

\end{tikzpicture}}

\newcommand{\graphg}[0]{
\begin{tikzpicture}[x=1pt,y=1pt,yscale=-0.2,xscale=0.4,baseline=10pt, line width = 0.7pt]

\draw  [fill= {rgb, 255:red, 0; green, 0; blue, 0 }  ,fill opacity=1 ] (0, -170) circle (5);
\draw  [fill= {rgb, 255:red, 0; green, 0; blue, 0 }  ,fill opacity=1 ] (0, 75) circle (5);

\draw  (-50, -50) circle (5);
\draw  (0, -100) circle (5);
\draw  (50, -50) circle (5);
\draw  (0, 0) circle (5);

\draw  [dash pattern={on 4pt off 3pt}]  (0,75) -- (0,5);
\draw [dash pattern={on 4pt off 3pt}]  (45,-50) -- (-45,-50);
\draw  [dash pattern={on 4pt off 3pt}]  (0,-170) -- (0,-105);

\draw   [dash pattern={on 4pt off 3pt}]  (-45,-50) -- (-5,-95);
\draw [dash pattern={on 4pt off 3pt}]  (5,-95) -- (45,-50);
\draw  [dash pattern={on 4pt off 3pt}]  (45,-45) -- (5,0);
\draw  [dash pattern={on 4pt off 3pt}]  (-5,0) -- (-45,-45);

\end{tikzpicture}}

\newcommand{\chorddiagrama}[0]{
\begin{tikzpicture}[x=1pt,y=1pt,yscale=0.2,xscale=0.3,baseline=20pt, line width = 1pt]
\draw [color={rgb, 255:red, 0; green, 0; blue, 255 }, line width=1pt]  [-Stealth] (0,-100)--(0, 400);

\draw [color={rgb, 255:red, 0; green, 0; blue,0 }, line width=1pt]   [Stealth-] (0,300)..controls (100,200)..(0,100) ;
\node at (120,200) {$(2)$};
\draw [color={rgb, 255:red, 0; green, 0; blue,0 }, line width=1pt]   [-Stealth] (0,0)..controls (100,100)..(0,200) ;
\node at (120,100) {$(1)$};

\draw  [fill={rgb, 255:red, 0; green, 0; blue, 0 }  ,fill opacity=1 ] (0,0) circle (5) ;
\node at (-20,0) {$O$};
\draw  [fill={rgb, 255:red, 0; green, 0; blue, 0 }  ,fill opacity=1 ] (0,100) circle (5) ;
\node at (-20,100) {};
\draw [fill={rgb, 255:red, 0; green, 0; blue, 0 }  ,fill opacity=1 ] (0,200) circle (5) ;
\node at (-20,200) {};
\draw [fill={rgb, 255:red, 0; green, 0; blue, 0 }  ,fill opacity=1 ] (0,300) circle (5);
\node at (-20,300) {};
\end{tikzpicture}}

\newcommand{\chorddiagramb}[0]{
\begin{tikzpicture}[x=1pt,y=1pt,yscale=0.2,xscale=0.3,baseline=20pt, line width = 1pt]
\draw [color={rgb, 255:red, 0; green, ; blue, 255 }, line width=1pt]  [-Stealth] (0,-100)--(0, 400);
\draw [color={rgb, 255:red, 0; green, 0; blue, 255 }, line width=1pt]  [-Stealth] (200,-100)--(200, 400);
\draw [color={rgb, 255:red, 0; green, 0; blue,0 }, line width=1pt]   [Stealth-] (0,300)--(200,0);
\node at (160,300) {$(1)$};
\draw [color={rgb, 255:red, 0; green, 0; blue,0 }, line width=1pt]   [-Stealth] (0,0)--(200,300) ;
\node at (40,300) {$(3)$};
\draw [color={rgb, 255:red, 0; green, 0; blue,0 }, line width=1pt]   [-Stealth] (0,150)--(200,150) ;
\node at (40,180) {$(2)$};

\draw  [fill={rgb, 255:red, 0; green, 0; blue, 0 }  ,fill opacity=1 ] (0,0) circle (5) ;
\node at (-20,0) {$O$};
\draw  [fill={rgb, 255:red, 0; green, 0; blue, 0 }  ,fill opacity=1 ] (0,150) circle (5) ;
\draw [fill={rgb, 255:red, 0; green, 0; blue, 0 }  ,fill opacity=1 ] (0,300) circle (5);

\draw  [fill={rgb, 255:red, 0; green, 0; blue, 0 }  ,fill opacity=1 ] (200,0) circle (5) ;
\node at (220,0) {$O$};
\draw  [fill={rgb, 255:red, 0; green, 0; blue, 0 }  ,fill opacity=1 ] (200,150) circle (5) ;
\draw [fill={rgb, 255:red, 0; green, 0; blue, 0 }  ,fill opacity=1 ] (200,300) circle (5);

\end{tikzpicture}}

\begin{document} 
\title{Cocycles of the space of long embeddings and BCR graphs with more than one loop}
\author{Leo Yoshioka}
\address{Graduate School of Mathematical Sciences, The University of Tokyo}
\email{yoshioka@ms.u-tokyo.ac.jp}

\begin{abstract}
The purpose of this paper is to construct non-trivial cocycles of the space $\text{Emb} (\mathbb{R}^j, \mathbb{R}^{n})$ of long embeddings. We construct the cocycles through configuration space integrals, associated with Bott-Cattaneo-Rossi graphs with more than one loop. As an application, we give a non-trivial cycle of long embeddings in the path component of the trivial long embedding, for odd $n,j$ with $n-j \geq 2$ and $j \geq 3$. This cycle is constructed from a chord diagram on directed lines. The non-triviality is shown by cocycle-cycle pairing, described by pairing between graphs and chord diagrams.
\end{abstract}

\maketitle


\tableofcontents

\part*{Introduction}
\addcontentsline{toc}{part}{Introduction}

A long embedding is an embedding of $\mathbb{R}^j$ into $\mathbb{R}^n$, which is standard outside a disk in $\mathbb{R}^j$. A  long $n$-knot is a  long embedding $\mathbb{R}^{n} \rightarrow \mathbb{R}^{n+2}$. In this paper, we produce some geometric non-trivial cocycles of the space $\text{Emb}(\mathbb{R}^j, \mathbb{R}^{n})$ of long embeddings through integrals associated with graphs. This kind of approach, which we call configuration space integrals, has its quantum field theoretical origin in E. Witten's work \cite{Wit} and its mathematical formulation has been developed by M. Kontsevich \cite{Kon}, D. Bar-Natan \cite{Bar}, R. Bott \cite{Bot}, C. Taubes \cite{BT}, G. Kuperberg, D. Thurston \cite{KT}, A. S. Cattaneo, C. A. Rossi \cite{CR} and some others from the 1990s to 2000s. Recently, this approach attracted much attention thanks to T. Watanabe \cite{Wat 5}, who in 2018 disproved the 4-dimensional Smale conjecture based on a series of his works \cite{Wat 2, Wat 3, Wat 4}.

Bott \cite{Bot}, Cattaneo and Rossi \cite{CR} defined 1-loop graphs (graphs with $b_1 = 1$) with two types of edges and vertices. We call these graphs BCR graphs. Using specific linear combinations of BCR graphs called graph cocycles, they defined configuration space integral invariants for long $n$-knots. K. Sakai and Watanabe \cite{Sak 2, SW} generalized their construction and systematically gave non-trivial higher cocycles of $\text{Emb}(\mathbb{R}^j, \mathbb{R}^{n})$ for $n-j\geq 3$ from BCR graph cocycles.

The graphs used in this paper are similar to theirs, but our graphs have more than one loop ($b_1 \geq 2$). Using such higher-loop graphs plays an essential role in getting cocycles of higher degree when the codimension $n-j$ is exactly two. The framework of applying higher-loop  BCR graphs was already considered by Sakai \cite{Sak 2}. However, no non-trivial cocycle corresponding to higher-loop graphs has been given so far. 

In this paper, we show the following. (See Section \ref{Main Result}.)
\begin{theorem}
Let $n$, $j$ be odd and satisfy $n-j\geq 2$, $ j\geq 3$. Then the linear combination $H$ of $2$-loop BCR graphs gives an explicit non-trivial $3n-2j-7$ (co)cycle of $\overline{\text{Emb}}(\mathbb{R}^j, \mathbb{R}^{n})$, the space of long embeddings modulo immersions. 
\end{theorem}

This theorem is shown by cocycle-cycle pairing. We systematically construct cycles of the space of long embeddings, which we call generalized ribbon cycles, from chord diagrams on directed lines. This diagrammatic construction makes it possible to compute cocycle-cycle pairing through pairing between graphs and diagrams. This diagrammatic computation is analogous to the diagrammatic description of cohomology-homology pairing of configuration spaces that  D.P. Sinha introduced in \cite{Sin}.

As a corollary, we give a non-trivial $S^{2}$-family of trivial long $3$-knots. Budney, Gabai \cite{BG} and Watanabe \cite{Wat 6} already showed that the $(n-1)$-th homotopy group of the path component of the trivial long $n$-knots, $\pi_{n-1}(\text{Emb}(\mathbb{R}^n, \mathbb{R}^{n+2})_{\iota})$\, $(n\geq 2)$, has an infinite-rank subgroup, by different approaches from ours. However, their families of long $n$-knots are constructed through $\text{Diff}_{\partial}(D^{n+1}\times S^1)$ or $\text{Emb}_{\partial}(D^1, D^{n+1}\times S^1)$. It would be interesting to compare our construction of families with theirs. Note that our approach has the possibility to produce an infinite-rank subgroup of $\pi_{n-1}(\text{Emb}(\mathbb{R}^n, \mathbb{R}^{n+2})_{\iota})$ using more general 2-loop graphs. 
 
Moreover, our graph cocycles with more than one loop further support the interesting similarity between the BCR graph complex and the graph complex introduced by G. Arone and V. Turchin \cite{AT1, AT2}. Their complex arises from a homotopy theoretical approach to embedding spaces: embedding calculus developed by Goodwillie, Klein and Weiss \cite{GW, GKW, Wei}. The significant point of their complex is that its graph homology has the full information about the dimension of the rational homology of $\text{Emb}(\mathbb{R}^j, \mathbb{R}^n)$ for $n\geq 2j+2$ \footnote{In \cite{FTW}, Fresse, Turchin and Willwacher generalized the results of \cite{AT2} to $n- j \geq 3$. Note that their approach consists of (i) Goodwillie--Weiss embedding calculus tower (ii) its mapping space model.  (ii) is applicable to any codimension, though (i) works for $n-j \geq 3$.}.  They mentioned in \cite{AT2} that some of their graphs are very similar to BCR graphs with no more than one loop. The point there was that Arone and Turchin's graph complex allows graphs with any number of loops. The $2$-loop graph cocycle we give in this paper has the $2$-loop graph \graphg{} similar to a graph which appears as a non-trivial element in \cite{AT2}. Since BCR graph cohomology can give information for any codimension more than one, we would be able to say that the stability of high codimensions that is described by this graph also survives in the range $n-j =2$.

Although we can apply our framework to the case when $n$ and $j$ are even, we have got no non-trivial cocycles of the space of long embeddings. The main cause is that we have not got a non-trivial graph cocycle for $n,j$ even. But one will be able to guess non-trivial cocycles from the computation of the 2-loop part of Arone and Turchin's graph homology \cite{CCTW}. Note that the graph homology is infinite dimensional.

The paper is organized as follows. First, in Section \ref{Background}, we review the framework of configuration space integrals associated with BCR graphs.  After that, some definitions and notation are given in Section \ref{Definition}. Main Result is stated in Section \ref{Main Result}. Then in Section \ref{Construction of the cocycles}, we construct the well-defined cocycles of $\text{Emb}(\mathbb{R}^n, \mathbb{R}^{n+2})$. After reviewing Sakai and Watanabe's construction of cycles in Section~\ref{Review of Sakai and Watanabe's construction of the wheel-like cycles}, we construct cycles from chord diagrams on directed lines in Section~\ref{Construction of the cycles}. In Section \ref{Some lemmas for computing cocycle-cycle pairing}, we show preliminary lemmas used in the proof of Main Result.  Finally, in Section \ref{Proof of Main Result}, we give the proof of Main Result by pairing the cocycles and the cycles of Section \ref{Construction of the cocycles} and \ref{Construction of the cycles}, described by graph-diagram pairing. 

\section*{Acknowledgement}
 The author is deeply grateful to the author's supervisor, Mikio Furuta, for his continuous support and encouragement for years. The author would like to thank Tadayuki Watanabe for teaching the author a lot about his work.  Discussion with Watanabe highly motivated the author during the preparation of this paper. The author thanks Keiichi Sakai for giving the author a great opportunity to talk and discuss at Shinshu Topology Seminar. Sakai also taught the author a lot about embedding spaces. The author also thanks Masahito Yamazaki and Tatsuro Shimizu for helpful suggestions and stimulating conversations. The author thanks Victor Turchin for informing the author about the papers \cite{FTW, CCTW} and for giving the author a lot of advice for better exposition. 
 
This research was supported by Forefront Physics and Mathematics Program to Drive Transformation (FoPM), a World-leading Innovative Graduate Study (WINGS) Program, the University of Tokyo.

\section{Background}
\label{Background}

\subsection{The BCR graph complex and the de Rham complex of $\text{Emb}(\mathbb{R}^j, \mathbb{R}^n)$}
Sakai in \cite{Sak 2} defined the BCR graph complex $(\mathcal{D}^{k,l}_g, \delta)$ and constructed a linear map from this space to the de Rham complex of $\text{Emb}(\mathbb{R}^j, \mathbb{R}^n)$ through configuration space integrals.
This correspondence gives a higher dimensional analog (in the sense $j \geq 2$) of the correspondence developed by Cattaneo, Cotta-Ramusino and Longoni \cite{CCL}. Refer to subsections \ref{BCR graphs} and \ref{The BCR graph complex} for the definition of the BCR graph complex (and the definitions of $k$, $g$ and $l$).
\begin{theorem}\cite[Theorem 1.2] {Sak 2}
\label{Sakai's result}
Let $k\geq1$, $g\geq0$. Configuration space integrals give a linear map 
\[
I: \mathcal{D}^{k,l}_g \rightarrow \Omega^{k(n-j-2) + (g-1)(j-1) +l}_{dR}(\text{Emb}(\mathbb{R}^j, \mathbb{R}^n)).
\]
Moreover, if one of 
\begin{itemize}
\item $n-j$ is even, $n-j\geq2$, $j\geq 2$ and $g=0$ 
\item $n$ and $j$ are odd, $n>j\geq 3$ and $g=1$
\end{itemize}
is satisfied, the map $I$ gives a cochain map.
\end{theorem}

Based on the above theorem, it is expected that we can obtain non-trivial cohomology classes of $\text{Emb}(\mathbb{R}^j, \mathbb{R}^n)$ if the graph cohomology is non-trivial and the map on the cohomology level induced from $I$ satisfies injectivity.

From this perspective, Sakai and Watanabe conducted further research on $g=1$(and $l=0$) in \cite{SW}, where configuration space integral invariants are put together into the universal one. Their work is based on the theory of finite type invariants for ribbon $n$-knots developed by Habiro, Kanenobu and Shima \cite{HS, HKS}.
\begin{theorem}\cite[Theorem 1.1] {SW}
Let $(n, j, k)$ satisfy some conditions. Let $\mathcal{A}^k_1$ be the quotient of the space of $1$-loop graphs by some relations (such as the STU relations of Jacobi diagrams). Then there exists a linear and injective map
\[
\alpha : \mathcal{A}^k_1 \rightarrow H_{(n-j-2)k}(\text{Emb}(\mathbb{R}^j, \mathbb{R}^n),\mathbb{R}). 
\]
In fact, the universal configuration space integral invariant $z_k$ gives a closed $(n-j-2)k$ form, and detects the non-triviality of the above cycles. We review the construction of the cycles in Section \ref{Review of Sakai and Watanabe's construction of the wheel-like cycles}.
\end{theorem}

\begin{rem}
In Theorem \ref{Sakai's result}, we can observe that when $n-j=2$ we have to use graphs with $g\geq2$ or $l\geq1$ to obtain cocycles of higher degree. 
\end{rem}

\begin{rem}
The case $l\geq1$ is out of the scope of this paper. For the case $j=1$, some progresses are made by Longoni \cite{Lon}, Sakai \cite{Sak 1, Sak 3}, Pelatt and Sinha \cite{PS}. In these works, new types of cycles are constructed generalizing the case $l=0$. 
\end{rem}

\section{Definition} 
\label{Definition}
\subsection{The space of long embeddings}
\label{The space of long embeddings}

First, We define the space of long embeddings and some related spaces.

\begin{definition}
A long embedding is an embedding $\mathbb{ R}^j \rightarrow \mathbb{R}^n$ which coincides with the standard linear embedding $\iota: \mathbb{R}^j \subset \mathbb{R}^n$ outside a disk in $\mathbb{R}^j$. A long embedding $\mathbb{R}^n \rightarrow \mathbb{ R}^{n+2}$ is called a long $n$-knot. We equip the space $\text{Emb}(\mathbb{R}^j, \mathbb{R}^n)$ of long embeddings with the induced topology from the weak $C^{\infty}$ topology. We define the space of long immersions $\text{Imm}(\mathbb{R}^j, \mathbb{R}^n)$ similarly. 
\end{definition}

\begin{rem}
Budney showed in \cite{Bud} that $\text{Emb}(S^j, S^n) \simeq SO_{n+1} \times_{SO_{n-j}} \text{Emb}(\mathbb{R}^j, \mathbb{R}^n)$. That is, the difference between the two spaces  $\text{Emb}(S^j, S^n)$ and $\text{Emb}(\mathbb{R}^j, \mathbb{R}^n)$ is given by the Stiefel manifold $SO_{n+1}/SO_{n-j}$, whose homotopy groups are well-studied.
\end{rem}

\begin{rem}
We often consider $\text{Emb}_{\partial}(D^j, D^n)$, the space of embeddings $D^j\rightarrow D^n$ which are standard near the boundary, instead of $\text{Emb}(\mathbb{R}^j, \mathbb{R}^n)$. The homotopy types of these two spaces are equivalent.
\end{rem}

For technical reasons, sometimes we think of the space of embeddings modulo immersions.
\begin{definition}
\label{embeddings modulo immersions}
We write $\overline{\text{Emb}}(\mathbb{R}^j, \mathbb{R}^n)$ for the homotopy fiber of the inclusion
\[
\text{Emb}(\mathbb{R}^j, \mathbb{R}^n) \hookrightarrow \text{Imm}(\mathbb{R}^j, \mathbb{R}^n)
\]
at the standard linear embedding $\iota: \mathbb{R}^j \subset \mathbb{R}^n$. That is, an element of $\overline{\text{Emb}}(\mathbb{R}^j, \mathbb{R}^n)$  is a one-parameter family $\{\overline{K}_t\}_{t \in [0,1]}$ of long immersions which satisfies $\overline{K}_0 =~\iota$, $\overline{K}_1\in \text{Emb}(\mathbb{R}^j, \mathbb{R}^n$). Let 
\[
r:  \overline{\text{Emb}}(\mathbb{R}^j, \mathbb{R}^n)  \rightarrow  \text{Emb}(\mathbb{R}^j, \mathbb{R}^n)
\]
be the natural projection.

\begin{notation}
Write $\iota: \mathbb{R}^j \rightarrow \mathbb{R}^n$ for the standard linear embedding. The path component of $\text{Emb}(\mathbb{R}^j, \mathbb{R}^n)$ which $\iota$ belongs to is denoted by $\text{Emb}(\mathbb{R}^j, \mathbb{R}^n)_{\iota}$. We often call this component the unknot component. Write $\overline{\text{Emb}}(\mathbb{R}^j, \mathbb{R}^n)_{\iota}$ for the component of the trivial family of the standard linear immersion. 
\end{notation}

\end{definition}

\subsection{BCR graphs}
\label{BCR graphs}
We move on to the definition of BCR graphs and their complex. See Section~\ref{Main Result} for examples of BCR graphs. Although we basically follow the definitions of \cite{Bot, CR, Sak 2, SW, Wat 1}, we restate them since there are some differences in conventions among these papers. 

\begin{definition}
A BCR graph is a graph which satisfies the following conditions.
\begin{itemize}
\item [(1)] There are two types of vertices, white and black. 
\item [(2)] There are two types of edges, dashed and solid. 
\item [(3)] Each white vertex has only dashed edges, and its valency is at least three. Each black vertex has at least one dashed edge and has an arbitrary number of solid edges.
\item [(4)] White vertices have no small loop
\begin{tikzpicture}[x=1pt,y=1pt,yscale=1,xscale=1, baseline=1pt]
\draw [dash pattern = on 3pt off 2pt] (0,0) .. controls (-20,20) and (20,20)  .. (0,0);
\draw (0,-1) circle (1.5);
\end{tikzpicture}. 
Black vertices have no small loop consisting of a solid edge only
\begin{tikzpicture}[x=1pt,y=1pt,yscale=1,xscale=1, baseline=1pt]
\draw  (0,0) .. controls (-20,20) and (20,20)  .. (0,0);
\draw  [fill={rgb, 255:red, 0; green, 0; blue, 0 }  ,fill opacity=1 ] (0,0) circle (1);
\end{tikzpicture}.
(See Remark \ref{rmk about loops} below.)
\end{itemize} 
\end{definition}

\begin{example}
Here are typical examples of vertices of BCR graphs.

\begin{tikzpicture}[x=0.75pt,y=0.75pt,yscale=0.4,xscale=0.4, baseline=20pt, line width = 1pt] 
\draw [dash pattern={on 5pt off 4pt}]  (0,0)--(95,95);
\draw [dash pattern={on 5pt off 4pt}]  (200,0)--(105,95);
\draw [dash pattern={on 5pt off 4pt}]  (100,200)--(100,105);
\draw [color={rgb, 255:red, 0; green, 0; blue, 0 }  ,draw opacity=1 ] (100, 100) circle (7);

\begin{scope}[xshift=300]
\draw [dash pattern={on 3pt off 2pt}]  (0,100)--(200,100);
\draw [fill={rgb, 255:red, 0; green, 0; blue, 0 }  ,fill opacity=1 ] (0,100) circle(7); 
\end{scope}

\begin{scope}[xshift=500]
\draw [dash pattern={on 3pt off 2pt}]  (0,0)--(95,95);
\draw [-] (105,95)--(200,0);
\draw [fill={rgb, 255:red, 0; green, 0; blue, 0 }  ,fill opacity=1 ] (100,100) circle(7); 
\end{scope}

\begin{scope}[xshift=700]
\draw [-]  (0,0)--(95,95);
\draw [dash pattern={on 3pt off 2pt}] (100,200)--(100,100);
\draw []  (105,95)--(200,0);
\draw [fill={rgb, 255:red, 0; green, 0; blue, 0 }  ,fill opacity=1 ] (100,100) circle(7); 
\end{scope}
\end{tikzpicture}
\end{example}

\begin{rem}
\label{rmk about loops}
Black vertices are allowed to have small loops
\begin{tikzpicture}[x=1pt,y=1pt,yscale=1,xscale=1, baseline=1pt]
\draw  [dash pattern = on 3pt off 2pt] (0,0) .. controls (-20,20) and (20,20)  .. (0,0);
\draw  [fill={rgb, 255:red, 0; green, 0; blue, 0 }  ,fill opacity=1 ] (0,0) circle (1);
\end{tikzpicture}
and double loops
\begin{tikzpicture}[x=1pt,y=1pt,yscale=1,xscale=1, baseline=1pt]
\draw [dash pattern = on 3pt off 2pt] (0,0) .. controls (-20,20) and (20,20)  .. (0,0);
\draw  (0,0) .. controls (-10,10) and (10,10)  .. (0,0);
\draw  [fill={rgb, 255:red, 0; green, 0; blue, 0 }  ,fill opacity=1 ] (0,0) circle (1);
\end{tikzpicture}
. We write the set of small loops of $\Gamma$ as $L_s(\Gamma)$ and the set of double loops as $L_d(\Gamma)$. Set $L(\Gamma) = L_s(\Gamma) \sqcup L_d(\Gamma)$. In our case ($n,j$ are odd), double loops vanish by the third relation in Definition \ref{admissible graphs}.
\end{rem}

\begin{notation}
Let $\Gamma$ be a BCR graph. We denote the set of solid edges of $\Gamma$ by $E_\eta(\Gamma)$, and the set of dashed edges by $E_\theta(\Gamma)$. Edges of small loops and double loops are included in these sets. The set of edges $E(\Gamma)$ is decomposed as 
\[
E(\Gamma) = E_{\eta}(\Gamma) \sqcup E_{\theta}(\Gamma).
\]
We denote the set of black vertices of $\Gamma$ by $B(\Gamma)$, and the set of white vertices by $W(\Gamma)$. The set of vertices $V(\Gamma)$ is decomposed as 
\[
V(\Gamma) = B(\Gamma) \sqcup W(\Gamma).
\]
\end{notation}

\begin{definition}[Defect of a graph]
Let $v$ be a vertex of a BCR graph $\Gamma$. We define the defect $l(v)$ of $v$ as
\begin{equation*}
l(v) = \begin{cases}
\# E_{\theta}(v) -3 & (v\in W(\Gamma)) \\
\# E_{\theta}(v)-1 & (v\in B(\Gamma)).\\
\end{cases}
\end{equation*}
Here $E_{\theta}(v)$ is the number of dashed edges $v$ has.
The defect $l(\Gamma)$ of $\Gamma$ is defined as
\[
l(\Gamma) = \sum_{v\in V(\Gamma)}l(v),
\]
and is equal to
\[
2 \#E_{\theta}(\Gamma) - 3\#W(\Gamma) - \#B(\Gamma).
\]
 \end{definition}
 
\begin{rem}
The defect of a graph measures how the graph is different, in terms of edge contractions, from graphs in which each black vertex has exactly one dashed edge and each white vertex has exactly three dashed edges.
\end{rem}

\begin{definition} [Order of a graph]
\label{Order of a graph}
The order $k(\Gamma)$ of a BCR graph $\Gamma$ is given by
\[
k(\Gamma) =\#E_{\theta}(\Gamma) - \#W(\Gamma).
\]
\end{definition}

\begin{rem}
When the defect is $0$, we have $k = \frac{\#V(\Gamma)}{2}$ . 
\end{rem}

\begin{definition} [$g$-loop graphs]
\label{g-loop graphs}
A BCR graph is a $g$-loop graph if the first Betti number of any component is $g$. Here, we consider small and double loops as points, and we do not count them as loops.
\end{definition}

\begin{definition} [Labels of a graph]
We give a label of $\Gamma$ by an ordering of $V(\Gamma)$, an ordering of $E(\Gamma)$, an orientation of each edge $e\in E(\Gamma)$ and an ordering of $L_s(\Gamma)$.
Here, an ordering of a set $X$ with $m$ elements is a bijective map from $\{1,2,\dots, m\}$ to $X$. We choose an ordering so that black vertices and solid edges come first.
\end{definition}

\begin{definition} [Orientations of a graph]
An orientation $s$ of a graph $\Gamma$ is a choice
\begin{equation*}
s \in \text{det} \left(\mathbb{R}V(\Gamma)\bigoplus  \oplus_{e\in E(\Gamma)} \mathbb{R}H(e) \right).
\end{equation*}
Here $H(e)$ is the set of two half-edges of $e$. A labeled graph determines an orientation of the underlying graph. 
\end{definition}

\begin{rem}
In this paper, we only handle the case $n,j$ are odd. The above definitions of labels and orientations are only for this case. We can define orientation of graphs so that the definition depends only on the parities of $n,j$.
\end{rem}

\begin{rem}
An ordering of $L_s(\Gamma)$ is used to define the configuration space integral associated with $\Gamma$. However, since we define configuration space integrals for graphs of defect $0$, we do not use this data in the rest of this paper.
\end{rem}

\subsection{The BCR graph complex}
\label{The BCR graph complex}

\begin{definition}
\label{admissible graphs}
We define relations of the vector space spanned by labeled BCR graphs of order $k$ and defect $l$. 
We denote the quotient space by $\mathcal{D}^{k,l}$.
\begin{itemize}
\item [(1)] Two labels with the same orientation are identified, and ones with opposite orientations are given opposite signs. 
\item [(2)] $\Gamma = 0$ for any graph $\Gamma$ with a double edge.

\begin{center}
\begin{tikzpicture}[x=1pt,y=1pt,yscale=1.5,xscale=1.5]
\draw   (-1,0) circle (1) ;
\draw   (50,0) circle (1) ;
\draw  [dash pattern = on 3pt off 2pt] (0,0) .. controls (20,5) and (30,5)  .. (50,0);
\draw  [dash pattern = on 3pt off 2pt](0,0) .. controls (20,-5) and (30,-5)  .. (50,0);
\begin{scope}[xshift= 100]
\draw  [fill={rgb, 255:red, 0; green, 0; blue, 0 }  ,fill opacity=1 ] (-1,0) circle (1) ;
\draw  [fill={rgb, 255:red, 0; green, 0; blue, 0 }  ,fill opacity=1 ] (50,0) circle (1) ;
\draw  (0,0) .. controls (20,5) and (30,5)  .. (50,0);
\draw  (0,0) .. controls (20,-5) and (30,-5)  .. (50,0);
\end{scope}
\end{tikzpicture}
\end{center}

\item [(3)] $\Gamma = 0$ for any graph $\Gamma$ with a double loop. (On the other hand,  a graph with small loops may survive.)

\item [(4)] When $l=0$, we set $\Gamma = 0$ for any ``non-admissible'' graph $\Gamma$. A graph $\Gamma$ is non-admissible (i) if it has a black vertex whose eta-valency is at least three such as 
\begin{tikzpicture}[x=0.75pt,y=0.75pt,yscale=0.4,xscale=0.4, baseline=20pt, line width = 1pt] 
\draw [dash pattern={on 3pt off 2pt}] (100,150)--(100,100);
\draw []  (0,0)--(95,95);
\draw [-]  (100,100)--(100,0);
\draw []  (105,95)--(200,0);
\draw [fill={rgb, 255:red, 0; green, 0; blue, 0 }  ,fill opacity=1 ] (100,100) circle(7); 
\end{tikzpicture},
 or (ii) if it has a multiple edge
 \begin{tikzpicture}[x=1pt,y=1pt,yscale=1.5,xscale=1.5, baseline=-1pt]
\draw  [fill={rgb, 255:red, 0; green, 0; blue, 0 }  ,fill opacity=1 ] (-1,0) circle (1) ;
\draw  [fill={rgb, 255:red, 0; green, 0; blue, 0 }  ,fill opacity=1 ] (50,0) circle (1) ;
\draw  [dash pattern = on 3pt off 2pt] (0,0) .. controls (20,5) and (30,5)  .. (50,0);
\draw (0,0) .. controls (20,-5) and (30,-5)  .. (50,0);
\end{tikzpicture}.

\end{itemize}

\end{definition}

\begin{notation}
We write the subspace of $\mathcal{D}^{k,l} $ generated by graphs with $g$ loops as $\mathcal{D}^{k,l}_g$. (Here, we count neither small nor double loops as loops.)
\end{notation}

\begin{rem}
The fourth condition is just to make the computation of the top cohomology easier, and it follows the conventions of \cite{SW}. So far, the author has yet to find precise reasons to exclude the possibility that we can construct a graph cocycle (of defect $0$) essentially using these non-admissible graphs. But one possible reason is that we can change non-admissible graphs to admissible graphs by Arnol'd relation (see subsection \ref{Comparing with the Arone-Turchin graph complex}) with respect to solid edges. 
\end{rem}

\begin{rem}
One possible variant of $\mathcal{D}$ is the graph complex $\mathcal{\overline{D}}$ that is the same as $\mathcal{D}$ for $l\geq0$ but has defect $(-1)$ part $\overline{\mathcal{D}}^{-1}$. Here a graph with defect $(-1)$ is a graph that has vertices of defect $0$ and exactly one black vertex with three solid edges and no dashed edge. Contractions of the three solid edges adjacent to the black vertex yield Arnol'd relation. 
\end{rem}

Next, we define the coboundary map  $\delta: \mathcal{D}^{k,l}_g\rightarrow \mathcal{D}^{k,l+1}_g$ of the graph complex.
\begin{definition}[Contraction of a BCR graph  \cite{Sak 2}]
\label{Contraction of a BCR graph}
We define the contraction $\Gamma/e$ of a BCR graph $\Gamma$ at an edge $e$ as follows.
\begin{itemize}
\item [(a)] When $e = (p,q)\in E_\eta(\Gamma)$ (except multiple edges).

Collapse $e$ to the black vertex at the ends of $e$ with a smaller label, and then reassign labels to vertices of the collapsed graph in a natural way. (That is, assign $\text{min}\{p, q\}$ to this collapsed vertex, and subtract $1$ from all vertex-labels greater than $\text{max}\{p,q\}$.

\begin{center}
\begin{tikzpicture}[x=1pt,y=1pt,yscale=1.5,xscale=1.5]
\draw  [fill={rgb, 255:red, 0; green, 0; blue, 0 }  ,fill opacity=1 ] (0,0) circle (1) node [anchor = north] {$p$};
\draw  [fill={rgb, 255:red, 0; green, 0; blue, 0 }  ,fill opacity=1 ] (50,0) circle (1) node [anchor = north] {$q$};
\draw  (0,0)--(50,0);
\draw (25, -10) node {$\Gamma$};

\begin{scope}[xshift=100]
\draw  [fill={rgb, 255:red, 0; green, 0; blue, 0 }  ,fill opacity=1 ] (0,0) circle (1) node [anchor = west] {$\text{min}\{p, q\}$};
\draw (0, -10) node {$\Gamma/e$};
\end{scope}
\end{tikzpicture}
\end{center}

\item [(b)] When  $e = (p,q)\in E_{\theta}(\Gamma)$ and at least one of $p$ and $q$ is white.

If exactly one is white, collapse $e$ to the black vertex. If both are white, collapse $e$ to the white vertex with a smaller label. The way to reassign labels is similar to (a).

\begin{center}
\begin{tikzpicture}[x=1pt,y=1pt,yscale=1.5,xscale=1.5]
\draw  (-1,0) circle (1) node [anchor = north] {$p$};
\draw  [fill={rgb, 255:red, 0; green, 0; blue, 0 }  ,fill opacity=1 ] (50,0) circle (1) node [anchor = north] {$q$};;
\draw  [dash pattern = on 3pt off 2pt] (0,0)--(50,0);
\draw (25, -10) node {$\Gamma$};

\begin{scope}[xshift=90]
\draw  [fill={rgb, 255:red, 0; green, 0; blue, 0 }  ,fill opacity=1 ] (0,0) circle (1) node [anchor = west] {$\text{min}\{p, q\}=q$};
\draw (0, -10) node {$\Gamma/e$};
\end{scope}

\end{tikzpicture}
\end{center}

\begin{center}
\begin{tikzpicture}[x=1pt,y=1pt,yscale=1.5,xscale=1.5]
\draw  (-1,0) circle (1) node [anchor = north] {$p$};
\draw  (50,0) circle (1) node [anchor = north] {$q$};;
\draw  [dash pattern = on 3pt off 2pt] (0,0)--(50,0);
\draw (25, -10) node {$\Gamma$};

\begin{scope}[xshift=100]
\draw (0,0) circle (1) node [anchor = west] {$\text{min}\{p, q\}$};
\draw (0, -10) node {$\Gamma/e$};
\end{scope}

\end{tikzpicture}
\end{center}

\item[(c)] When $e = (p,q)\in E_\theta(\Gamma)$ and both $p$ and $q$ are black vertices.

We construct a small loop \footnote{Though \cite{Sak 2} introduces sign of small loops, it seems unnecessary.}. If the number of small loops of $\Gamma$ is $a$, the label of this loop is $(a+1)$.

\begin{center}
\begin{tikzpicture}[x=1pt,y=1pt,yscale=1.5,xscale=1.5]
\draw  [fill={rgb, 255:red, 0; green, 0; blue, 0 }  ,fill opacity=1 ] (-1,0) circle (1) node [anchor = north] {$p$};
\draw  [fill={rgb, 255:red, 0; green, 0; blue, 0 }  ,fill opacity=1 ] (50,0) circle (1) node [anchor = north] {$q$};;
\draw  [dash pattern = on 3pt off 2pt] (0,0)--(50,0);
\draw (25, -10) node {$\Gamma$};

\begin{scope}[xshift=100]
\draw [dash pattern = on 3pt off 2pt] (0,0) .. controls (-20,20) and (20,20)  .. (0,0);
\draw (0, 20) node {$(a+1)$};
\draw  [fill={rgb, 255:red, 0; green, 0; blue, 0 }  ,fill opacity=1 ] (0,0) circle (1) node [anchor = west] {$\text{min}\{p, q\}$};
\draw (0, -10) node {$\Gamma/e$};
\end{scope}

\end{tikzpicture}
\end{center}

\item[(d)] When $e = (p,q)$ is the solid edge of a multiple edge.

We construct a double loop, which vanishes by the third relation in Definition \ref{admissible graphs}. We do not perform contractions of dashed edges of multiple edges. 

\end{itemize}

\end{definition}

\begin{definition}  [Coboundary map \cite{Sak 2}] 
\label{Coboundary operator}
We define the coboundary map $\delta : \mathcal{D}^{k, l}_g \rightarrow \mathcal{D}^{k, l+1}_g$ of $\mathcal{D}^{k, l}_g$ as follows:
\[
\delta(\Gamma) = \sum_{e\in E(\Gamma)\setminus\{\text{loops}\}} \sigma(e)\,\Gamma/e.
\]

Here, the sign $\sigma(e)$ of an edge $e$ is defined as follows. 

\begin{equation*}
\sigma(e=(p, q))= 
\begin{cases} 
(-1)^q  &  \text{if $p<q$ }\\
(-1)^{p+1} & \text{if $p>q$}.\\
\end{cases}
\end{equation*}
\end{definition}

\begin{prop}
$(\mathcal{D}^{k, \ast}_g , \delta)$ is a cochain complex.
\end{prop}
\begin{proof}
This follows from the fact that reversing the order of two contractions changes the sign. 
\end{proof}

\subsection{Comparing with the Arone-Turchin graph complex}
\label{Comparing with the Arone-Turchin graph complex}
Arone and Turchin  \cite{AT1, AT2} defined two graph complexes $\mathcal{E}^{j,n}$ and $HH^{j,n}$ ($j$ and $n$ come from $\text{Emb}(\mathbb{R}^j, \mathbb{R}^n)$). These complexes both arise from a homotopy theoretical approach: the description of the space of long embeddings as a space of derived maps between modules over an operad.

The complex $\mathcal{E}^{j,n}$ consists of so-called hairy graphs with two types of vertices and with one type of edges. Graphs of $HH^{j,n}$ have two types of edges as graphs of $\mathcal{D}$.
The differences between $H H^{j,n}$ and our complex $\mathcal{D} = \mathcal{D}^{j,n}$ are
\begin{itemize}
\item [1.] Their graphs have only black vertices.
\item [2.] Their graphs do not have cycles that consist of one type of edges. In particular, there are no small loops, double loops or double edges.
\item [3.] Their coboundary map is the sum of contractions of only solid edges.
\item [4.] Their complex has another relation called Arnol'd relation \cite{Arn} as follows. (The edges in the figure must be the same type, either dashed or solid.)

\begin{tikzpicture}[x=0.75pt,y=0.75pt,yscale=1,xscale=1, baseline=20pt, line width = 1pt] 
\draw [-Stealth]  (0,0)--(95,0) ;
\draw(50,0) node [anchor = north] {$(1)$} ;
\draw [-Stealth]  (100,0)--(50,48) ;
\draw (75,30) node [anchor = west]  {$(2)$} ;

 ;
\draw [fill={rgb, 255:red, 0; green, 0; blue, 0 }  ,fill opacity=1 ] (0,0) circle(4) node [anchor = south] {$i$}; 
\draw [fill={rgb, 255:red, 0; green, 0; blue, 0 }  ,fill opacity=1 ] (100,0) circle(4) node [anchor = south] {$j$}; ; 
\draw [fill={rgb, 255:red, 0; green, 0; blue, 0 }  ,fill opacity=1 ] (50,50) circle(4) node [anchor = south] {$k$} ; 
\end{tikzpicture}
+
\begin{tikzpicture}[x=0.75pt,y=0.75pt,yscale=1,xscale=1, baseline=20pt, line width = 1pt] 
\begin{scope}[xshift=100]
\draw [-Stealth]  (95,0)--(50,50);
\draw (75,30) node [anchor = west]  {$(1)$} ;
\draw [-Stealth]  (50,50)--(5,0);
\draw (5,30) node [anchor = west]  {$(2)$} ;

\draw [fill={rgb, 255:red, 0; green, 0; blue, 0 }  ,fill opacity=1 ] (0,0) circle(4) (0,0) circle(5) node [anchor = south] {$i$}; ; 
\draw [fill={rgb, 255:red, 0; green, 0; blue, 0 }  ,fill opacity=1 ] (100,0) circle(4) node [anchor = south] {$j$} ; 
\draw [fill={rgb, 255:red, 0; green, 0; blue, 0 }  ,fill opacity=1 ] (50,50) circle(4) node [anchor = south] {$k$} ; 
\end{scope}
\end{tikzpicture}
+
\begin{tikzpicture}[x=0.75pt,y=0.75pt,yscale=1,xscale=1, baseline=20pt, line width = 1pt] 
\begin{scope}[xshift=200]
\draw [-Stealth]  (0,0)--(95,0);
\draw(50,0) node [anchor = north] {$(2)$} ;
\draw [-Stealth]  (50,50)--(5,0);
\draw (5,30) node [anchor = west]  {$(1)$} ;
\draw [fill={rgb, 255:red, 0; green, 0; blue, 0 }  ,fill opacity=1 ] (0,0) circle(4) node [anchor = south] {$i$} ; 
\draw [fill={rgb, 255:red, 0; green, 0; blue, 0 }  ,fill opacity=1 ] (100,0) circle(4) node [anchor = south] {$j$} ; 
\draw [fill={rgb, 255:red, 0; green, 0; blue, 0 }  ,fill opacity=1 ] (50,50) circle(4) node [anchor = south] {$k$} ; 
\end{scope}
\end{tikzpicture}
= 0.
 
\end{itemize}
Note that our convention for solid and dashed edges are opposite from \cite{AT1} and \cite {AT2}: we follow the convention that solid edges correspond to $\mathbb{R}^j$ of the domain.

\subsection{Configuration space integrals}
\label{Configuration space integral}
From now on, for simplicity, we assume that graphs used to define configuration space integrals are connected and satisfy $l=0$.

\begin{notation}
We write $C_s(\mathbb{R}^n)$ for the configuration space of $s$ points in $\mathbb{R}^n$:
\[
C_s(\mathbb{R}^n) = (\mathbb{R}^n)^{s}\setminus \text{all diagonals}.
\]
\end{notation}

\begin{definition}[Configuration space]
Let $K:\mathbb{R}^j \rightarrow\mathbb{R}^{n}$ be a long embedding. Then we define the configuration space $C_{s,t}(K, \mathbb{R}^n)$ as follows.
\[
 C_{s,t}(K, \mathbb{R}^n)= \left\{(x_1, x_2, \dots, x_s, x_{s+1}, \dots, x_{s+t}) \in (\mathbb{R}^n)^{s+t}\setminus \text{all diagonals}\, \middle| 
\begin{array}{l}\forall i \in \{1, \dots, s\}, x_i \in K\end{array}\right\}. 
\]
Hence $ C_{s,t}(K, \mathbb{R}^n)$ is the pullback of the following diagram:
\begin{center}
\begin{tikzpicture}[auto]
\node (c) at (0, 2) {$C_{s,t}(K, \mathbb{R}^n)$}; \node (d) at (3, 2) {$C_{s+t}(\mathbb{R}^n)$}; 
\node (a) at (0, 0) {$C_{s}(\mathbb{R}^j)$}; \node (b) at (3, 0) {$C_{s}(\mathbb{R}^n)$};

\draw[->] (c)  to (d);
\draw[->] (a) to node {$K^{\times s}$} (b);
\draw [->] (c) to  (a);
\draw [->] (d) to  node {restriction} (b);
\end{tikzpicture}
\end{center}
\end{definition}

\begin{definition} [Direction map]
\label{Direction map}
Let $\Gamma$ be a labeled BCR graph and $s=\#B(\Gamma), t=\#W(\Gamma)$.
Let $e = (i, j)$ be an edge of $\Gamma$. Depending on $e\in E_{\theta}(\Gamma)$ or $e \in E_{\eta}(\Gamma)$, the direction map $\phi_e = \phi_{i j}$ from $C_{s,t}(K, \mathbb{R}^n)$ to a sphere is defined as follows:
\begin{itemize}
\item If $e = (i, j)\in E_{\theta}(\Gamma)$,
\[
\phi_{ij} : C_{s,t}(K, \mathbb{R}^n) \longrightarrow C_2(\mathbb{R}^n)=C_{0,2}(\emptyset, \mathbb{R}^n) \longrightarrow S^{n-1}.
\]
\item If $e =(i, j)\in E_{\eta}(\Gamma)$,

\[
\phi_{ij} : C_{s,t}(K, \mathbb{R}^n) \longrightarrow C_2(\mathbb{R}^j) =C_{2,0}(K, \mathbb{R}^n) \longrightarrow S^{j-1}.
\]
\end{itemize}
In both cases, the first map is the restriction
\[
(x_1, x_2, \dots, x_s, x_{s+1}, \dots, x_{s+t}) \mapsto (x_i, x_j),
\]
and the second is defined as 
\[
(x,y) \mapsto \frac{y-x}{||y-x||}.
\]

Intuitively, dashed edges measure the direction between two points in $\mathbb{R}^n$, while solid edges measure the direction in $\mathbb{R}^j$.


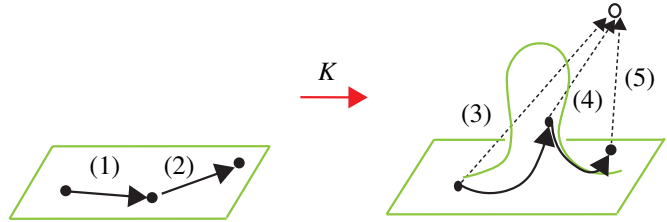
\begin{figure}[htpb]
\labellist
\small \hair 10pt
\pinlabel $(1)$ at 200 230
\pinlabel $(2)$ at 300 230
\pinlabel $(3)$ at 700 300
\pinlabel $(4)$ at 850 320
\pinlabel $(5)$ at 920 350
\pinlabel $K$ at 500 360
\endlabellist
\centering
\includegraphics[width = 10cm]{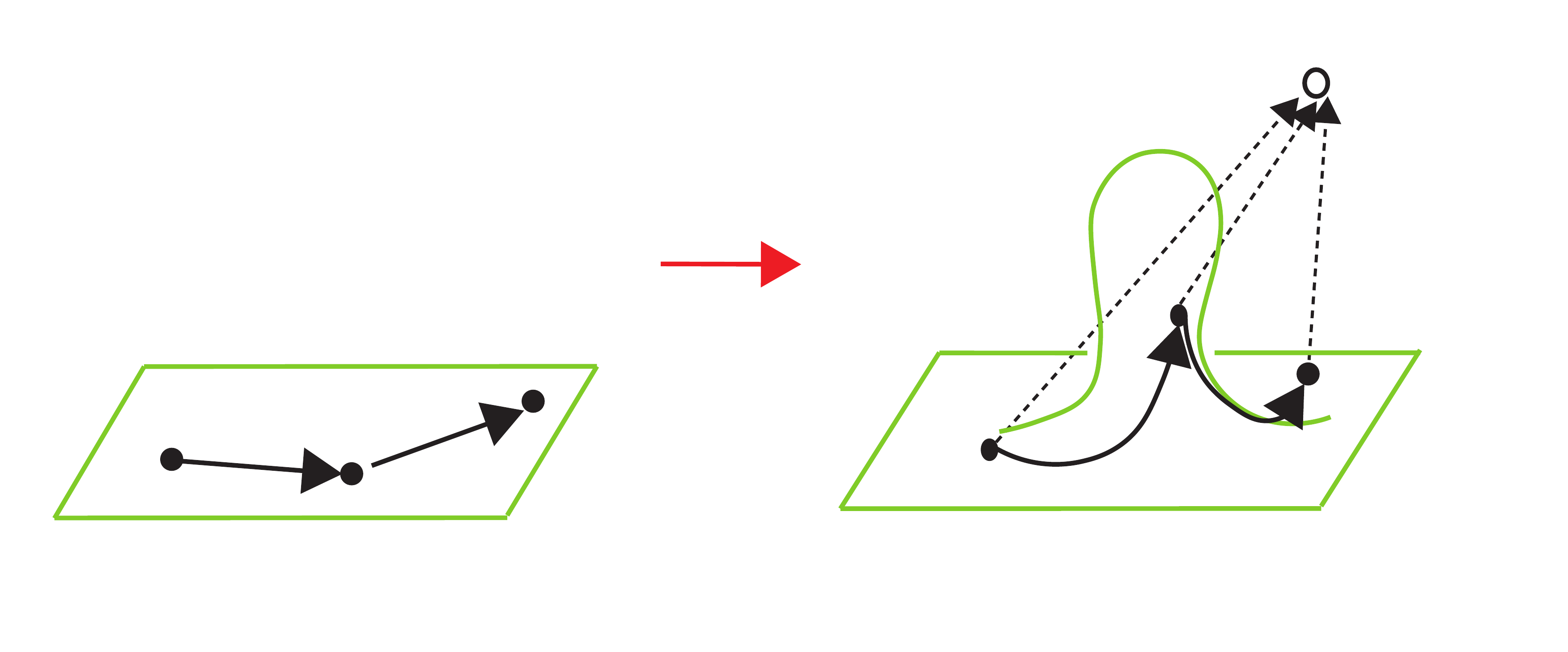}
\caption{Graph and its direction maps}
\end{figure}

\end{definition}

\begin{definition} [Propagator]
\label{Propagator}
We define the differential form $\omega (K, \Gamma)$ on $C_{s,t}(K, \mathbb{R}^n)$ of degree $(n-1) \# E_{\theta}(\Gamma) + (j-1) \# E_{\eta}(\Gamma)$ by
\[
\omega (K, \Gamma) = \bigwedge_{e\in E_{\eta}(\Gamma)} \phi_e^{\ast}\omega_{S^{j-1}} \bigwedge_{e\in E_{\theta}(\Gamma)} \phi_e^{\ast} \omega_{S^{n-1}}. 
\]
Here, the order of the wedge products is arbitrary, since both $\omega_{S^{j-1}}$ and $\omega_{S^{n-1}}$ are even forms when $n$ and $j$ are odd. We can use the order of edges, for example.
\end{definition}

\begin{rem}
The forms $\phi_e^{\ast}\omega_{S^{j-1}}$ and $\phi_e^{\ast} \omega_{S^{n-1}}$ are often called propagators. Leturcq \cite{Let} has given more flexible propagators to define invariants of embeddings into rational homology $n$-spheres.
\end{rem}

The next proposition ensures the convergence of integrals we later define.

\begin{prop} \cite{AS, Sin 2}, \cite[Proposition1.2]{BT}

\noindent There exists a canonical compactification $\overline{C}_{s,t}(K, \mathbb{R}^n)$ of $C_{s,t}(K, \mathbb{R}^n)$ such that $\phi_{ij}$ can be extended to the whole $\overline{C}_{s,t}(K, \mathbb{R}^n)$. The codimension 1 part of $\partial\overline{C}_{s, t}(K, \mathbb{R}^n)$ is 
\[
\bigsqcup_{A\subset \{1, \dots, s, s+1, \dots, s+t, \infty\}} \widetilde{C}_A (K, \mathbb{R}^n), 
\]
where $\widetilde{C}_A (K, \mathbb{R}^n)$ is ``the configuration space with $A$ being infinitely close''. (See subsection \ref{Construction of a correction term} for details.)
\end{prop}

\begin{definition}
We classify $\widetilde{C}_A (K, \mathbb{R}^n) $ ($A\subset \{1, \dots, s, s+1, \dots, s+t, \infty\}$ and $\# A\geq 2$) into four types as follows ($\Gamma_A$ is the subgraph of $\Gamma$ generated by vertices of $A$).
\begin{itemize}
\item   Infinite faces: $\infty \in A$.
\item   Principal faces: $\# A = 2$.
\item   Hidden faces: $\# A \geq 3$ and $\Gamma_A$ is not the whole $\Gamma$; $\Gamma_A \subsetneq \Gamma$.
\item   Anomalous  faces: $\# A \geq 3$ and $\Gamma_A$ is the whole $\Gamma$; $ \Gamma_A = \Gamma$.
\end{itemize}
\end{definition}

\begin{notation}
We write $E_{s, t}(\mathbb{R}^j , \mathbb{R}^n)$ for the bundle over $\text{Emb}(\mathbb{R}^j, \mathbb{R}^n)$ whose fiber at a long embedding $K$ is the configuration space $C_{s,t}(K, \mathbb{ R}^n)$. More precisely, $E_{s, t}(\mathbb{R}^j , \mathbb{R}^n)$ is the pullback of the following diagram:
\begin{center}
\begin{tikzpicture}[auto]
\node (c) at (0, 2) {$E_{s,t}(\mathbb{R}^j, \mathbb{R}^n)$}; \node (d) at (5, 2) {$C_{s+t}(\mathbb{R}^n)$}; 
\node (a) at (0, 0) {$C_{s}(\mathbb{R}^j) \times \text{Emb}(\mathbb{R}^j , \mathbb{R}^n$)}; \node (b) at (5, 0) {$C_{s}(\mathbb{R}^n)$};

\draw[->] (c)  to (d);
\draw[->] (a) to node {evaluation} (b);
\draw [->] (c) to  (a);
\draw [->] (d) to  node {restriction} (b);
\end{tikzpicture}
\end{center}

By abusing notation, we also write $E_{s, t}(\mathbb{R}^j , \mathbb{R}^n)$ for the bundle whose fiber is $\overline{C}_{s,t}(K, \mathbb{R}^n)$. We write $\widetilde{E}_A(\mathbb{R}^j , \mathbb{R}^n)$ for the bundle over $\text{Emb}(\mathbb{R}^j, \mathbb{R}^n)$ whose fiber at a long embedding $K$ is the configuration space $\widetilde{C}_A(K, \mathbb{R}^n)$. 
\end{notation}

As in Definition \ref{Direction map}, we can construct the direction map $\phi_e: E_{s, t}(\mathbb{R}^j , \mathbb{R}^n) \rightarrow S^{n-1}$ or $\phi_e: E_{s, t}(\mathbb{R}^j , \mathbb{R}^n) \rightarrow S^{j-1}$ depending on the type (solid or dashed) of $e$.

\begin{notation}
We write $\omega(\Gamma)$ for the from on $E_{s, t}(\mathbb{R}^j , \mathbb{R}^n)$ defined  by pulling back the standard volume forms $\omega_{S^{n-1}}$ and $\omega_{S^{j-1}}$ (in the same way as Definition \ref{Propagator}).
\end{notation}

\begin{definition} [Configuration space integrals]
\label{configurationspaceintegras}
Let $C_{s, t}$ (and $\overline{C}_{s, t}$)  be a typical fiber of the bundle $\pi: E_{s, t}(\mathbb{R}^j , \mathbb{R}^n) \rightarrow \text{Emb}(\mathbb{R}^j, \mathbb{R}^n)$. Let $\Gamma$ be a labeled BCR graph. We define the configuration space integral associated with $\Gamma$ by the fiber integral
\begin{equation*}
I(\Gamma)= \int_{\overline{C}_{s, t}} \omega(\Gamma)= \int_{C_{s, t}} \omega(\Gamma).
\end{equation*}
The form $I(\Gamma)$ on $\text{Emb}(\mathbb{R}^j, \mathbb{R}^n)$ is sometimes written as $\pi_{\ast} \omega(\Gamma)$.
 \end{definition}

By Stokes' theorem (see \cite[Appendix]{BT} and  \cite[Appendix A.2]{Wat 1}), we have
\begin{align*}
&dI(\Gamma) \\
=& \pi_{\ast} d\omega(\Gamma) + (-1)^r \pi^{\partial}_{\ast} \omega(\Gamma) \\
= & \int_{\overline{C}_{s, t}} d \omega(\Gamma) + (-1)^r \int_{\partial \overline{C}_{s, t}} \omega(\Gamma)\\
= &(-1)^r \int_{\partial \overline{C}_{s, t}} \omega(\Gamma)\\
= &(-1)^r  \sum_A \int_{\widetilde{C}_{A}} \omega(\Gamma)
\end{align*}
where $r$ is the degree of $\int_{\partial \overline{C}_{s, t}} \omega(\Gamma)$. Hence for the closedness of $I(\Gamma)$, we should check that the integral over each $\widetilde{C}_A$ vanishes or is canceled.

Contributions of infinite faces vanish for dimensional reasons. (Recall that all our embeddings are standard outside a disk.) 

\begin{theorem}\cite[Theorem 5.10, Lemma 5.18]{Sak 2}
If $\infty \in A$, the integral $\int_{\widetilde{C}_A} \omega(\Gamma)$ vanishes.
\end{theorem}

On the other hand, contributions of principal faces are canceled when we sum up the configuration space integrals for graphs in a graph cocycle (with their coefficients).

\begin{theorem}\cite[Theorem 5.11]{Sak 2}
Let $\# A = 2$. The integral $\int_{\widetilde{C}_{A}} \omega(\Gamma)$ survives only when $A = e$ for some edge $e$ of $\Gamma$. Moreover, we have
\[
\int_{\widetilde{C}_{e}} \omega(\Gamma) = - \int_{C_{V(\Gamma/e)}} \omega(\delta_e\Gamma). \footnote{We follow the inward normal convention for orientations of boundaries, which causes the sign of the right-hand side.}
\]
Here $C_{V(\Gamma/e)}$ is the configuration space of vertices of $\Gamma/e$.
\end{theorem}

\section{Main Result}
\label{Main Result}

Recall $r: \overline{\text{Emb}}(\mathbb{R}^j, \mathbb{R}^n) \rightarrow \text{Emb}(\mathbb{R}^j, \mathbb{R}^n)$ is the natural projection defined in Definition \ref{embeddings modulo immersions}.

\begin{figure}[htpb]
\begin{center}
\includegraphics[width = 13cm]{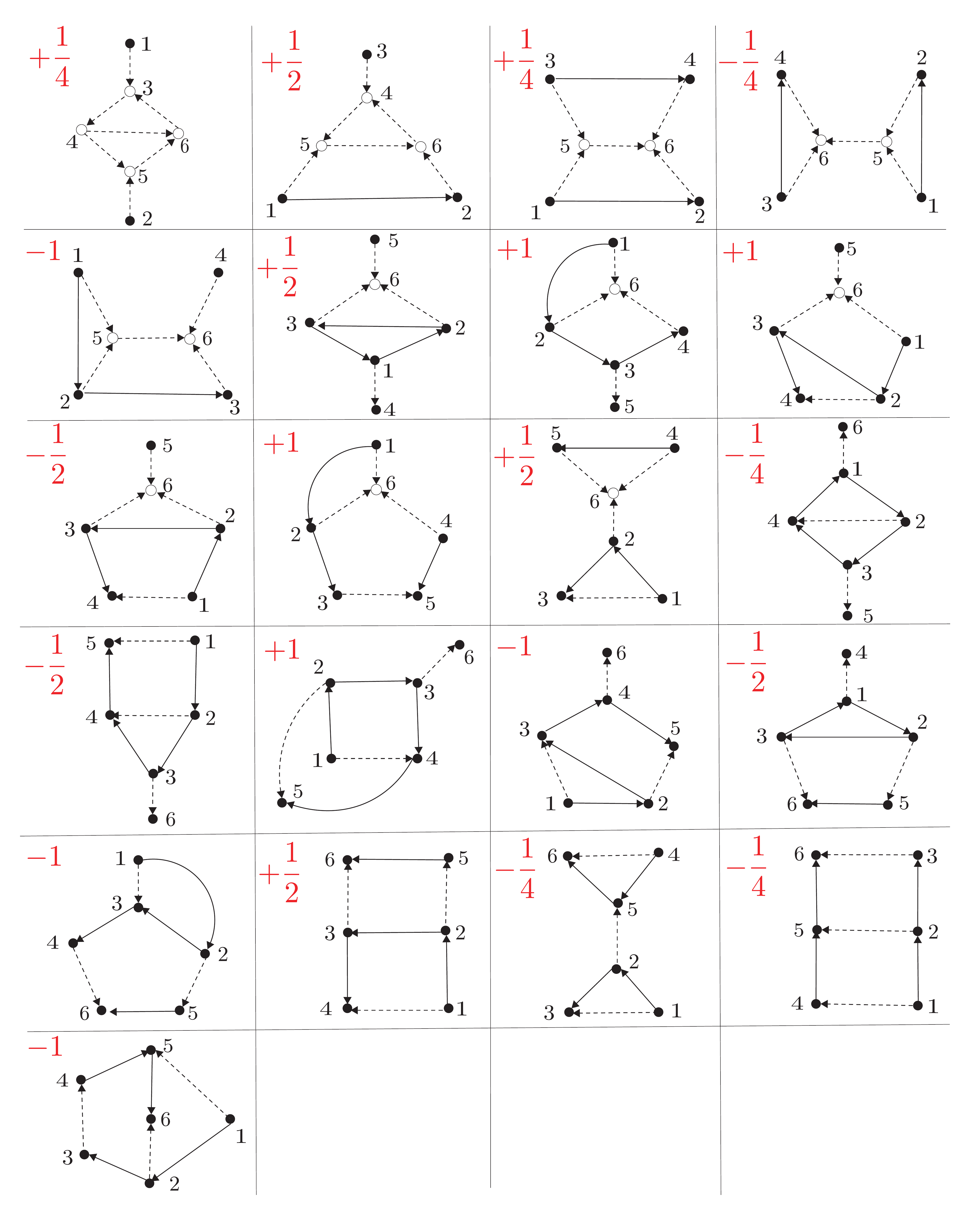}
\caption{Non-trivial graph cocycle $H$ (odd, odd)}
\label{2-loop graph cocycle of order 3}
\end{center}
\end{figure}

\begin{theorem} 
\label{main result 2}

Let $(n,j) = (\text{odd},  \text{odd})$ with $j \geq 3$. 
\begin{itemize}
\item[1] .
The linear combination of 2-loop BCR graphs of order 3
\[
 H = \sum_{i=1}^{21} w(\Gamma_i)\Gamma_i
\] 
is a graph cocycle. For the definition of  $\Gamma_i$ and its coefficient $w(\Gamma_i)$, see Figure \ref{2-loop graph cocycle of order 3}. 
\item [2]. 
Let $\overline{c}(H)$ be the correction term defined in Definition \ref{definition of the correction term}. Then
\[\overline{z}^3_2 = r^{\ast} I(H) + \overline{c}(H) \in \Omega^{3n-2j-7}_{dR} \overline{\text{Emb}}(\mathbb{ R}^j, \mathbb{R}^n)
\]
is closed and  $[\overline{z}^3_2] \in H^{3n-2j-7}_{dR} \overline{\text{Emb}}(\mathbb{R}^j, \mathbb{R}^n)$ is non-trivial. In fact we can explicitly construct  $\psi \in H_{3n-2j-7} \text{Emb}(\mathbb{R}^j, \mathbb{R}^n)$ and its lift $\overline{\psi}$, so that  the pairing $\overline{z}^3_2(\overline{\psi}) $ is non-trivial. See Figure \ref{Cycle psi} for $\psi$, which is constructed from the presentation in Figure \ref{Ribbon presentations P2}.
\item [3]. Moreover, if  $H^{3n-j-6}_{dR}(V_{n,j}(\mathbb{R})) = 0$, we can obtain a non-trivial element $z^3_2 \in H^{3n-2j-7}_{dR}(\text{Emb}(\mathbb{R}^j, \mathbb{R}^n))$. Here $V_{n,j} = V_{n,j}(\mathbb{R})$ is the real Stiefel manifold $\frac{SO(n)}{SO(n-j)}$. 
\end{itemize}
\end{theorem}

\begin{rem} (See \cite[Theorem 3.14] {MT})
When $(n, j)=(\text{odd}, \text{odd})$,
\[
H^{\ast}_{dR}(V_{n,j}) \cong \bigwedge(e_{n-j}, e_{2(n-j)+3}, e_{2(n-j)+7},\dots, e_{2(n-j)+2j-3}).
\]
In particular if $j=3$, $H^{3n-j-6}_{dR}(V_{n,j}(\mathbb{R})) = 0$ holds. 
\end{rem}

\begin{rem}
A similar graph to the first graph $\Gamma_1$ of  Figure \ref{2-loop graph cocycle of order 3} appears in \cite[Table C]{AT2} as one of the rational homology generators of  $\mathcal{E}^{m,n}$, the graph complex of hairy graphs. One can also check that $HH^{m,n}$ has the last graph as one of the rational generators.
\end{rem}

We can take the above non-trivial cycle $\psi$ from the unknot component, and also over the 2-sphere when the codimension is 2.

\begin{cor}
\label{main corollary}
\[
\pi_2(\text{Emb} (\mathbb{R}^3, \mathbb{R}^5)_{\iota}) \otimes \mathbb{Q} \neq 0.
\]
\end{cor}

\section{Construction of  the cocycles}
\label{Construction of the cocycles}

In this section, we show $r^{\ast} I(H) + \overline{c}(H)$ is closed. According to the paragraphs after Definition \ref{configurationspaceintegras}, what to show is the vanishing of contributions of principal, hidden and anomalous faces.

\subsection{H is a graph cocycle}
\label {H is a graph cocycle}

\qquad We can show that $H$ is a graph cocycle, by direct computation of the coboundary map $\delta$ (defined in Definition \ref{Coboundary operator}). In this computation, we need to put together isomorphic graphs with different labels, taking their orientations into account. This will be the most difficult part when computing the graph cohomology through computer calculous. 

Note that for the top degree, $\delta$ is described by a matrix (See Figure \ref{matrix}). 
We can easily see that the matrix is decomposed into blocks depending on the type of edges contracted:
(A) \dashededgea{} (B) \dashededgeb{} (C) \dashededgec{} (D) \solidedge{}.

\begin{figure}[htpb]

\labellist
\pinlabel $\#W=4$ at 320 1950
\pinlabel $\#W=0$ at 1080 1950
\endlabellist

\centering
\includegraphics[width=6cm]{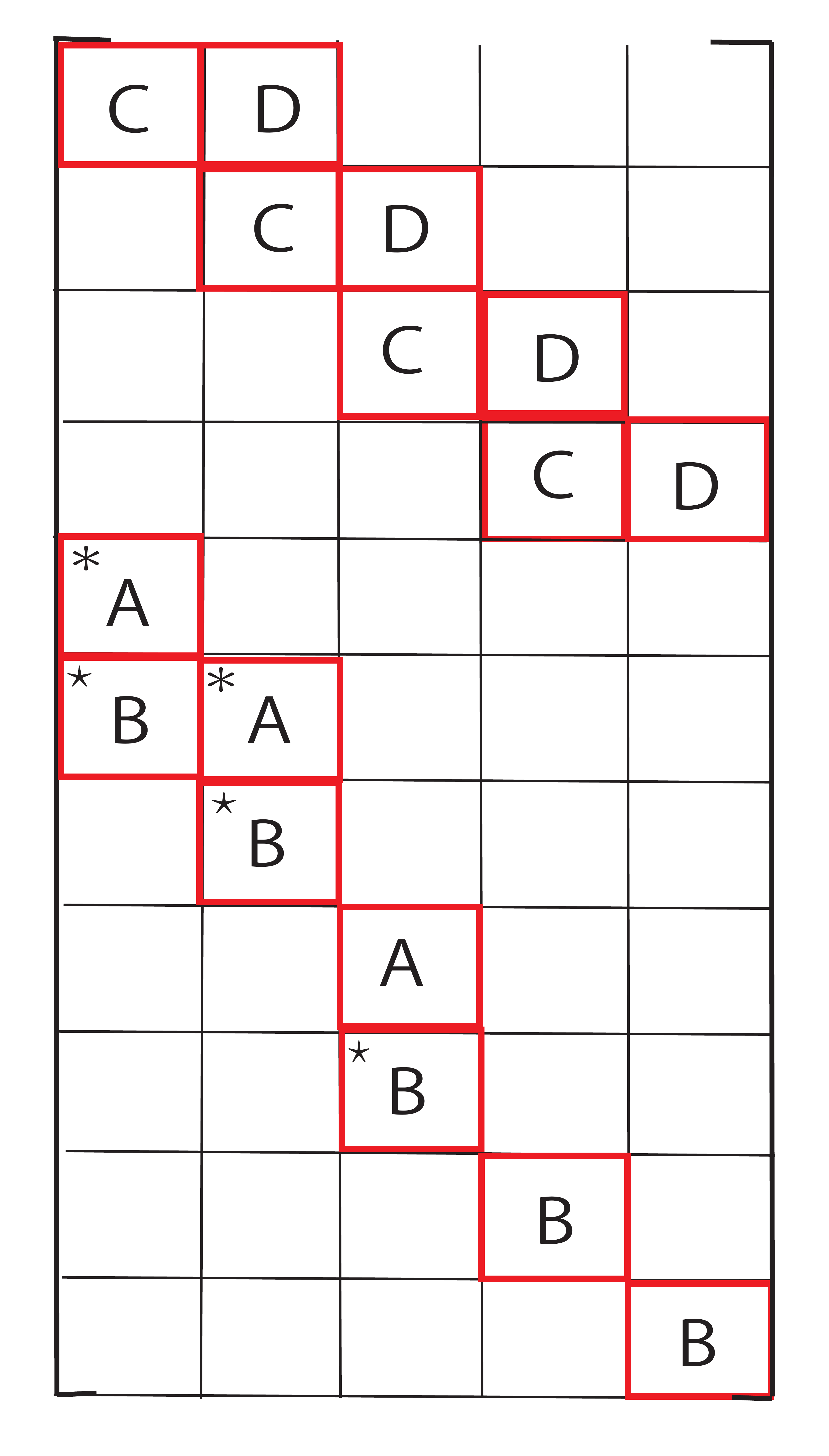}
\caption{An example of the matrix of  $\delta$. For our cocycle $H$, the blocks with $\ast$ are $0$. The blocks with $\star$ do not appear.}
\label{matrix}
\centering
\end{figure}


\subsection{Vanishing of contributions of hidden faces}
\label{Vanishing of contributions of hidden faces}

\qquad Here, we show that contributions of hidden faces vanish. The fundamental approach is the following proposition. The cancellation using the symmetry as below was first introduced by Kontsevich $\cite{Kon}$. Lescop gave a detailed explanation of it in  \cite{Les 1, Les 2}.

\begin{prop}
\label{basic case}
Let $s=\#B(\Gamma), t=\#W(\Gamma)$ and let $\widetilde{C}_A\subset\partial C_{s,t}(K, \mathbb{R}^n)$ be a hidden face. If the subgraph $\Gamma_A$ of $\Gamma$ satisfies one of the following conditions, the contribution of the integral over $\widetilde{C}_A$ vanishes. (Below $u$ and $b$ stand for univalent and bivalent, respectively).
\begin{itemize}
\item [(1)] $\Gamma_A$ is not connected.
\item [(2)] $\Gamma_A$ has a bivalent vertex of the following form. (Gray vertices can be both black and white.)
\tikzset{every picture/.style={line width=0.75pt}}  
\begin{center}
\begin{tikzpicture}[x=0.75pt,y=0.75pt,yscale=0.6,xscale=0.6]

\draw [color={rgb, 255:red, 0; green, 0; blue, 0 }  ,draw opacity=1 ]  [fill={rgb, 255:red, 0; green, 0; blue, 0 }  ,fill opacity=1 ] (0, 0) circle (5);
\draw (0,-40) node {$v_1$ (b)};

\draw  (-100,-100)--(-5,-5);
\draw [color={rgb, 255:red, 0; green, 0; blue, 0 }  ,draw opacity=1 ]  [fill={rgb, 255:red, 0; green, 0; blue, 0 }  ,fill opacity=1 ] (-100, -100) circle (5);
\draw (-100,-130) node {$v_2$};

\draw (5,-5)--(100,-100);
\draw [color={rgb, 255:red, 0; green, 0; blue, 0 }  ,draw opacity=1 ]  [fill={rgb, 255:red, 0; green, 0; blue, 0 }  ,fill opacity=1 ] (100, -100) circle (5);
\draw (100,-130) node {$v_3$};

\draw [color={rgb, 255:red, 0; green, 0; blue, 0 }  ,draw opacity=1 ] (250, 0) circle (5);
\draw (250,-40) node {$v_1 (b)$};

\draw [dash pattern={on 5pt off 4pt}]  (150,-100)--(245,-5);
\draw [color=gray ]  [fill=gray  ,fill opacity=1 ] (150, -100) circle (5);
\draw (150, -130) node  {$v_2 $};

\draw [dash pattern={on 5pt off 4pt}] (255,-5)--(350,-100);
\draw [color= gray ]  [fill=gray  ,fill opacity=1 ] (350, -100) circle (5);
\draw (350, -130) node  {$v_3$};
\end{tikzpicture}
\end{center}

\item [(3)] $\Gamma_A$ has a univalent vertex of the following form.
\tikzset{every picture/.style={line width=0.75pt}}  
\begin{center}
\begin{tikzpicture}[x=0.75pt,y=0.75pt,yscale=0.6,xscale=0.6]

\draw  (0,100)--(0,0);
\draw [color={rgb, 255:red, 0; green, 0; blue, 0 }  ,draw opacity=1 ] [fill={rgb, 255:red, 0; green, 0; blue, 0 }  ,fill opacity=1 ] (0, 100) circle (5);
\draw (40,100) node {$v_1 (u)$};
\draw (20,50) node {$e$};

\draw [dash pattern={on 5pt off 4pt}] (100,95)--(100,0);
\draw [color={rgb, 255:red, 0; green, 0; blue, 0 }  ,draw opacity=1 ] (100, 100) circle (5);
\draw (140,100) node {$v_1 (u)$};
\draw (120,50) node {$e$};

\end{tikzpicture}
\end{center}
\end{itemize}
\end{prop}

\begin{proof}
For (2), we use the involution of $v_1$ with respect to the middle point of $v_2$ and $v_3$:
\[
v_1 \mapsto v_2+v_3-v_1.
\]
When $(n, j) = $ (odd, odd), this involution reverses the orientation, while the integrand form does not change. 

For (3), the group $\mathbb{R}_+$ acts on $\widetilde{C}_A$ by rescaling the edge $e$. Since the form $\omega(\Gamma)$ can be constructed through the projection 
\[
\widetilde{C}_A \rightarrow \widetilde{C}_A/{\mathbb{R}_+},
\]
where the quotient space on the right-hand side has less dimension, we can see the integral over $\widetilde{C}_A$ is zero.

For (1), we can use the action of $\mathbb{R}_+$ that increases or decreases the distance between two components. 

\end{proof}

Unfortunately, our graph cocycle $H$ has hidden faces that satisfy none of the above conditions. However, since $(n, j) = (\text{odd}, \text{odd})$ we can apply the following proposition that was shown in \cite{Sak 2}. 

\begin{prop}
\begin{itemize}
\item
If  $n-j =$ even and $\Gamma_A$ is 0-loop (namely, $b_1(\Gamma_A) = 0$), the contribution of $\widetilde{C}_A$ vanishes. 
\item
If  $(n, j) = $ (odd, odd) and $\Gamma_A$ is 1-loop (namely, $b_1(\Gamma_A) = 1$),  the contribution of $\widetilde{C}_A$ vanishes. 
\end{itemize}
\end{prop}

\begin{rem}
The second vanishing argument is also useful to cancel the contribution of the anomalous face, when $(n,j)$ = (odd, odd) and the first Betti number of the graph is odd.
\end{rem}

\begin{proof}
The first statement is a consequence of similar involutions and rescaling to those of Proposition \ref{basic case}. 
We review the involution used to prove the second statement. This is an involution of all the vertices of $A$ with respect to one black vertex of $A$:
\begin{align*}
&(x_1, x_2, \dots, x_s, x_{s+1}, x_{s+2}, \dots, x_{s+t}) \\
\mapsto & (x_1, 2 x_1-x_2, \dots, 2x_1-x_s, 2\iota(x_1)-x_{s+1}, 2 \iota(x_1)-x_{s+2}, \dots, 2\iota(x_1)-x_{s+t}).
\end{align*}
Here, $\iota: \mathbb{R}^j \rightarrow \mathbb{R}^n$ is the linear injective map defined by the differential of the embedding. 
(If all the vertices of $A$ are white, take a white vertex as the central vertex~$x_1$.)
\end{proof}

If a subgraph $\Gamma_A$ of $\Gamma_i$ in the cocycle $H$ is 2-loop, it satisfies one of (1) (2) (3). Hence, it follows that all the contributions of hidden faces of $H$ do vanish.

\subsection{Construction of correction terms for anomalous faces}
\label{Construction of a correction term}

We now proceed to construct a correction term to cancel the contribution of the integral over the anomalous face $\widetilde{C}_{V(\Gamma)}$. If $H^{\ast}_{dR}(V_{n,j}(\mathbb{R}))$ vanishes on a certain degree, we can construct a correction term $c$ as a form on $\text{Emb}(\mathbb{R}^j, \mathbb{R}^n)$. Otherwise we construct one on $\overline{\text{Emb}}(\mathbb{R}^j, \mathbb{R}^n)$, the space of long embeddings modulo immersions. Sakai in $\cite{Sak 2}$ follows the first approach, while Sakai and Watanabe do the second one in $\cite{SW}$.

For the above purpose, we first describe the structure of each codimension one face $\widetilde{C}_A.$  Readers can also refer to \cite{Bot, CCL, Sak 2} for more detailed explanations. We only describe it when  the subset $A\subset V(\Gamma)=\{1, \dots, s, s+~1, \dots, s+~t\}$ has at least one element  of $ \{1, \dots, s\}$. The anomalous face is the case where $A$ is equal to the entire $V(\Gamma)$.

The space $\widetilde{E}_A(\mathbb{R}^j, \mathbb{R}^n)$ is defined as the pullback of the following diagram.  The configuration space $\widetilde{C}_A(K, \mathbb{R}^n)$ is the fiber  of the projection $\widetilde{E}_A(\mathbb{R}^j, \mathbb{R}^n) \rightarrow \text{Emb}(\mathbb{R}^j, \mathbb{R}^n)$ at an embedding $K$.
\begin{center}
\begin{tikzpicture}[auto]
\node (c) at (0, 2) {$\widetilde{E}_A(\mathbb{R}^j, \mathbb{R}^n)$}; \node (d) at (3, 2) {$B_A$}; 
\node (a) at (0, 0) {$E_{V/A}(\mathbb{R}^j, \mathbb{R}^n)$}; \node (b) at (3, 0) {$\text{Inj}(\mathbb{R}^j, \mathbb{R}^n)$};
\node (e) at (0, -1) {$\text{Emb}(\mathbb{R}^j, \mathbb{R}^n)$}; 

\draw[->] (c)  to (d);
\draw[->] (a) to node {$\scriptstyle D$} (b);
\draw [->] (c) to  (a);
\draw [->] (d) to  node {$b$} (b);
\draw [->] (a) to  node {$\pi$} (e); 
\end{tikzpicture}
\end{center}
Here, {$\text{Inj}(\mathbb{R}^j, \mathbb{R}^n)$} is the space of linear injective maps from $\mathbb{R}^j$ to $\mathbb{R}^n$. 
$E_{V/A}(\mathbb{R}^j, \mathbb{R}^n)$ is the bundle over $\text{Emb}(\mathbb{R}^j, \mathbb{R}^n)$ whose fiber is the configuration space $C_{V/A}$ of $V(\Gamma)/A$. $B_A$  is the bundle over $\text{Inj}(\mathbb{R}^j, \mathbb{R}^n)$ with the``normalized configuration space'' of $A$ as the fiber. See Notation \ref{normalized configuration space}. 
$D$  is a map from $E_{V/A}(\mathbb{R}^j, \mathbb{R}^n)$  to $\text{Inj} (\mathbb{R}^j ,\mathbb{R}^n) $ which assigns the differential at the collapsed point $A/A$.

\begin{notation}
\label{normalized configuration space}
We define the configuration spaces $C_{V/A}$ and $B_A$ as follows.
Write $V(\Gamma)/A $ and $A$ as
\begin{align*}
V(\Gamma)/A &= \{i_1, \dots, i_{s'}, i_{s'+1}, \dots, i_{s'+t'}\} ,\\
A &= \{j_1, \dots, j_{s''}, j_{s''+1}, \dots, j_{s''+t''}\}, 
\end{align*}
where $s' + s''=s+1$ and $t' + t''=t$. Then $C_{V/A}(K)$  for $K \in \text{Emb}(\mathbb{R}^j, \mathbb{R}^n)$  and $B_A(I)$ for $I \in \text{Inj}(\mathbb{R}^j, \mathbb{R}^n)$ are defined as 
\begin{align*}
C_{V/A} (K) &= C_{s', t'} (K)  \\
B_A(I) &= C_{s'', t''}(I)/ \sim.
\end{align*}
Here, $\sim$ is a relation generated by (a) rescalings (b) translations in $I$ direction:
\begin{itemize}
\item [(a)]$(x_{j_1}, \dots, x_{j_{s''}}, x_{j_{s''+1}}, \dots, x_{j_{s'' + t''}}) \sim r (x_{j_1}, \dots, x_{j_{s''}}, x_{j_{s''+1}}, \dots, x_{j_{s'' + t''}}) \quad (r\in~\mathbb{R}_{>0})$
\item [(b)]$(x_{j_1}, \dots, x_{j_{s''}}, x_{j_{s'' +1}}, \dots, x_{j_{s'' + t''}}) \sim v+ (x_{j_1}, \dots, x_{j_{s''}}, x_{j_{s''+1}}, \dots, x_{j_{s'' +t''}}) \quad (v /\!/ I).$
\end{itemize}
\end{notation}

Observe that if $A = V(\Gamma)$, (the restriction of) the form $\omega(\Gamma)$ on $\widetilde{E}_A(\mathbb{R}^j, \mathbb{R}^n)$ is described as the pullback of a form on $B_A$, for which we write $\omega(\Gamma)$, abusing notation.

Let $H = \sum_i \Gamma_i$ be a $g$-loop graph cocycle of order $k$. In the next proposition, we construct a correction term for anomalous faces when  $H_{dR}^{d} (\text{Inj}(\mathbb{R}^j, \mathbb{R}^n)) = H_{dR}^{d} (V_{n,j})=0$ for a certain $d$.

\begin{prop}
\label{correction term1}
Assume that contributions of hidden faces of $H$ are canceled.
For each $\Gamma$ of $H$, set $A=V(\Gamma)$ and consider the bundle $b: B_A \rightarrow \text{Inj}(\mathbb{R}^j, \mathbb{R}^n)$.
Then the fiber integral
\[
b_{\ast} \omega(H) = \sum_i b_{\ast} \omega(\Gamma_i)  \in \Omega^{d}_{dR}(\text{Inj}(\mathbb{R}^j, \mathbb{R}^n)), \quad d=k(n-j-2)+(g-1)(j-1)+(j+1)
\]
is closed. Moreover,  when $b_{\ast} \omega(H)$ is exact, we can construct a correction term of the anomalous face by using $\mu \in \Omega^{d-1}_{dR}(\text{Inj}(\mathbb{R}^j, \mathbb{R}^n))$ which satisfies $d\mu = (-1)^{r+1} b_{\ast}\omega(H)$ ($r = d-j$):

\[
c(H) = \pi_{\ast}D^{\ast}\mu.
\]
\end{prop}

\begin{proof}
Without loss of generality, assume $H = \Gamma_1+\Gamma_2$.  The first assertion follows because the integrals ${b_1}_{\ast}\omega(\Gamma_1)$ and  ${b_2}_{\ast}\omega(\Gamma_1)$ do not have anomalous faces, while contributions of hidden and principal faces of the integrals are canceled. 

Next, we show that $I(H) + c(H)$ is closed. 
Using Stokes' theorem (see the paragraph after Definition \ref{configurationspaceintegras}), we have
\begin{align*}
(-1)^r dI(H) &= \int_{\widetilde{C}_{V(\Gamma_1)}} \omega(\Gamma_1) + \int_{\widetilde{C}_{V(\Gamma_2)} }\omega(\Gamma_2)\\
&= \pi_{\ast} D^{\ast} {b_1}_{\ast} \omega(\Gamma_1) + \pi_{\ast} D^{\ast} {b_2}_{\ast} \omega(\Gamma_2)
\end{align*}
On the other hand, 
\begin{align*}
(-1) ^r dc(H) &= (-1)^r d\pi_{\ast}D^{\ast}\mu \\
& = (-1)^r \pi_{\ast}D^{\ast}d\mu +  \pi^{\partial}_{\ast}D^{\ast}\mu\\
& = -\pi_{\ast}D^{\ast}{b_1}_{\ast}\omega(\Gamma_1) - \pi_{\ast}D^{\ast}{b_2}_{\ast}\omega(\Gamma_2)
\end{align*}
So we have $d(I(H) + c(H)) = 0$.
\end{proof}

Next, we consider the case $H_{dR}^{d} (\text{Inj}(\mathbb{R}^j, \mathbb{R}^n)) =0$ is not satisfied. In this case, we construct a correction term of 
\[
r^{\ast} I(H)\in \Omega^{k(n-j-2) + (g-1)(j-1) } \overline{\text{Emb}}(\mathbb{R}^j, \mathbb{R}^n),
\]
where $r: \overline{\text{Emb}}(\mathbb{R}^j, \mathbb{R}^n) \rightarrow \text{Emb}(\mathbb{R}^j, \mathbb{R}^n)$ is the natural projection defined in Definition \ref{embeddings modulo immersions}.
Define
\[
\overline{D}: [0,1] \times C_1(\mathbb{R}^j) \times \overline{\text{Emb}}(\mathbb{R}^j, \mathbb{R}^n) \longrightarrow \text{Inj}(\mathbb{R}^j, \mathbb{R}^n)
\]
by $\overline{D}(t, x, \overline{K}) = D_x(\overline{K}_t)$, where $D$ is the differential.  Note that $\overline{D}(1) = D \circ (id \times r)$. Let 
\[
p_{23}: [0,1] \times C_1(\mathbb{R}^j) \times \overline{\text{Emb}}(\mathbb{R}^j, \mathbb{R}^n) \rightarrow C_1(\mathbb{R}^j) \times \overline{\text{Emb}}(\mathbb{R}^j, \mathbb{R}^n)
\] 
be the projection to the second and the third factors. 
\begin{definition}
\label{definition of the correction term}
We define the correction term for $\Gamma$ by 
\[
\overline{c}(\Gamma) = (-1)^r(-1)^{d+1} \pi_{\ast} {p_{23}}_{\ast}  \overline{D}^{\ast} b_{\ast}\omega(\Gamma).
\]
For a graph cocycle $H = \sum_i \Gamma_i$, set $\overline{c}(H) = \sum_i \overline{c}(\Gamma_i).$
\end{definition}

\begin{prop}
Let $H = \sum_i \Gamma_i$ be a graph cocycle such that contributions of hidden faces are canceled. Then $\overline{c}(H)$ is a correction term of $r^{\ast}I(H)$. 
\end{prop}
\begin{proof}
By using the closedness of $b_{\ast}\omega(H) = {b_1}_{\ast}\omega(\Gamma_1) + {b_2}_{\ast}\omega(\Gamma_2)$, we have
\begin{align*}
(-1)^r d \overline{c}(H) &=  (-1)^{d+1} \pi_{\ast} d ({p_{23}}_{\ast} \overline{D}^{\ast} {b_1}_{\ast}\omega(\Gamma_1)) + (-1)^{d+1}\pi_{\ast} d ({p_{23}}_{\ast} \overline{D}^{\ast} {b_2}_{\ast}\omega(\Gamma_2))\\
& = - \pi_{\ast} (id\times r)^{\ast} D^{\ast} {b_1}_{\ast} \omega(\Gamma_1) - \pi_{\ast} (id\times r)^{\ast} D^{\ast} {b_2}_{\ast} \omega(\Gamma_2)\\
& = - r^{\ast} \pi_{\ast} {D}^{\ast} {b_1}_{\ast}\omega(\Gamma_1) - r^{\ast} \pi_{\ast} {D}^{\ast} {b_2}_{\ast}\omega(\Gamma_2)\\
& = - (-1)^r  d (r^{\ast}I (H)).
\end{align*}

\begin{center}
\begin{tikzpicture}[auto]
\node (c) at (0, 2) {$\text{Emb}(\mathbb{R}^j, \mathbb{R}^n)$}; \node (d) at (5, 2) {$C_1(\mathbb{R}^j) \times \text{Emb}(\mathbb{R}^j, \mathbb{R}^n)$}; 
\node (a) at (0, 0) {$\overline{\text{Emb}}(\mathbb{R}^j, \mathbb{R}^n)$}; \node (b) at (5, 0) {$C_1(\mathbb{R}^j) \times \overline{\text{Emb}}(\mathbb{R}^j, \mathbb{R}^n)$}; 
\node (f) at (10, 0) {$\text{Inj}(\mathbb{R}^j, \mathbb{R}^n)$}; \node (g) at (12, 0) {$B_{V(\Gamma)}$}; 
\node (e) at (5, -2) {$[0,1] \times C_1(\mathbb{R}^j) \times \overline{\text{Emb}}(\mathbb{R}^j, \mathbb{R}^n)$};

\draw[->] (d)  to node {$\small \pi$} (c);
\draw[->] (b) to node {$\small \pi$} (a);
\draw [->] (a) to  node {$r$} (c); 
\draw [->] (b)  to node {$id \times r$} (d);
\draw [->] (e) to node {$p_{23}$} (b); 
\draw [->] (d) to node {$D$} (f);
\draw [->] (e) to node {$\overline{D}$} (f);
\draw [->] (g) to node {$b$} (f); 

\end{tikzpicture}
\end{center}
\end{proof}


\section{Review of Sakai and Watanabe's construction of the wheel-like cycles}
\label{Review of Sakai and Watanabe's construction of the wheel-like cycles}
Here, we review Sakai and Watanabe's construction \cite[Section 4] {SW} of the wheel-like cycles
\[
c_k : (S^{n-j-2})^k \longrightarrow \text{Emb}(\mathbb{R}^j, \mathbb{R}^n),\quad \mathbf{v} \mapsto \varphi^\mathbf{v}_k
\]
and 
\[
\widetilde{c_k} : (S^{n-j-2})^k \longrightarrow \overline{\text{Emb}}(\mathbb{R}^j, \mathbb{R}^n) 
\]
for $k \geq 2$. The long embeddings $\varphi^\mathbf{v}_k$ are constructed by ``perturbating" (\cite[4.1.1] {SW}) the original long embedding $\varphi_k$. The long embedding $\varphi_k$ is constructed from a specific oriented immersed 2-disk in $\mathbb{R}^{3}$, called the wheel-like ribbon presentation. This embedding $\varphi_k$ has specific $k$ parts called ``crossings'', where we can perform the perturbation to the direction $v_i \in S^{n-j-2}\, (i=1, \dots, k)$. This perturbation is analogous to the perturbation for constructing $k(n-3)$-cycles in $\text{Emb}(\mathbb{R}^1, \mathbb{R}^n)$. (See \cite{CCL}.) These crossings can be described as images of several standard disks and annuli in $\mathbb{R}^j$ (see \cite[4.2.2] {SW}), which look like ``planetary systems''. See Figure 4.6 of \cite{SW}.

Below, we often consider $\mathbb{R}^n$ as the product $\mathbb{R}^3 \times \mathbb{R}^{n-j-2} \times \mathbb{R}^{j-1}$. In this section, the standard embedding $\iota: \mathbb{R}^j \subset \mathbb{R}^n$ is considered as $x_1 \mapsto x_1$ and $x_i \mapsto x_{n-j+i}$ for $2\leq i \leq j$.

\subsection{The wheel-like ribbon presentations}
The original embedding $\varphi_k$ is constructed from the wheel-like ribbon presentation $P \subset \mathbb{R}^3$ (defined below) by (i) thickening $P$ and (ii) taking the boundary. The embedding $\varphi_k$ is realized in $\mathbb{R}^3 \times \mathbf{0} \times \mathbb{R}^{j-1}$.

\begin{definition}\cite[Definition 4.1] {SW}
The wheel-like ribbon presentation $P = \mathcal{D} \cup \mathcal{B}$ of order $k$ is a based oriented, immersed 2-disk in $\mathbb{R}^{3}$ such that
\begin{itemize}
\item $\mathcal{D} = D_0 \cup D_1 \cup \dots \cup D_k$ are disjoint $(k+1)$ disks,
\item $\mathcal{B}=  B_1 \cup B_2 \cup \dots \cup B_k$ are disjoint $k$ bands ($B_i \approx I \times I$),
\item $B_{i+1}$ connects $D_0$ with $D_i$ so that $B_{i+1} \cap D_0 = \{0\} \times I$ and $B_{i+1} \cap D_i = \{1\} \times I$ ($B_{1}$ is considered as $B_{k+1}$),
\item $B_{i}$ intersects transversally with the interior of $D_i$,
\item the base point $\ast$ is on the boundary of $D_0$ (but not on the boundaries of~$B_i$).
\end{itemize}
\end{definition}

Near the intersection of $B_i$ and $D_i$, they look, using a local coordinate, like
\begin{align*}
B&= \{(x_1, x_2, x_3) \in \mathbb{R}^3\,|\,|x_1| \leq \frac{1}{2}, |x_2| < 3, x_3=0\},\\
D&= \{(x_1, x_2, x_3) \in \mathbb{R}^3\,|\,|x_1|^2 + |x_3|^2 \leq 1, x_2=0\}.
\end{align*}
See the left picture in  Figure \ref{Perturbation of a ribbon presentation illust}.

\begin{definition}\cite[Definition 4.2] {SW}
We define the ribbon $(j+1)$-disk $V_P$ by 
\[
V_P = \left(\mathcal{B} \times [-\frac{1}{4}, \frac{1}{4}]^{j-1} \right) \bigcup  \left(\mathcal{D} \times [-\frac{1}{2}, \frac{1}{2}]^{j-1}\right) \subset \mathbb{R}^3 \times \mathbf{0} \times  \mathbb{R}^{j-1}.
\]
Note that the thicknesses of bands and that of disks are different.
\end{definition}

\begin{definition}
After smoothing corners of $V_P$, the embedding $\varphi_k$ is defined by the connected sum $\partial V_P \# \iota(\mathbb{R}^j)$ at the base point.
\end{definition}

\subsection{Perturbation of crossings}
\label{Perturbation of crossings}
Here, we define the perturbation $\varphi^{\mathbf{v}}_k$ of the original embedding $\varphi_k$. First, we define the perturbation $P_{\mathbf{v}}$ of the ribbon presentation $P$. The embedding $\varphi^{\mathbf{v}}_k$ is constructed from $P_{\mathbf{v}}$ in the same way as $\varphi_k$. Recall that the ribbon presentation $P$ is constructed in $\mathbb{R}^3 \times \mathbf{0} \times \mathbf{0}$, and the thickened presentation $V_p$ is constructed in $\mathbb{R}^3 \times \mathbf{0} \times \mathbb{R}^{j-1}$. The perturbation is performed using the coordinates $(x_4, \dots, x_{n-j+1})$ of $\mathbb{R}^{n-j-2}$ and the coordinate $x_3$ of $\mathbb{R}^3$. 

\begin{notation}
We consider the $(n-j-2)$-dimensional sphere $S^{n-j-2}$ as
\[
S =S^{n-j-2}=\left\{(x_3, \dots, x_{n-j+1}) \in \mathbb{R}^{n-j-1}\, | \, (x_3-1)^2+ x_4^2+\dots+x^2_{n-j+1}=1\right\}.
\]
Note that if $n-j-2=0$, $S$ is the set of two points $S^0 =\{x_3=0, x_3=2\}$.
\end{notation}

\begin{notation}
The band obtained by perturbing $B$ to $v\in S$ direction is written as $B(v)$. (See the right picture in Figure \ref{Perturbation of a ribbon presentation illust}):
\[
B(v) = 
\begin{cases}
\left\{(x_1, x_2, \gamma(x_2)v) \in \mathbb{R}^2 \times \mathbb{R}^{n-j-1} \, | \, |x_1| \leq 1/2, |x_2|<3 \right\}.\\
\end{cases}
\]
Here $\gamma(y)$ is a test function defined as 
\[
\gamma(y)=
\begin{cases}
 \text{exp}(-y^2/\sqrt{9-y^2}) & (|y| \leq 3) \\
 0 & (|y| \geq 3)
\end{cases}.
\]
\end{notation}


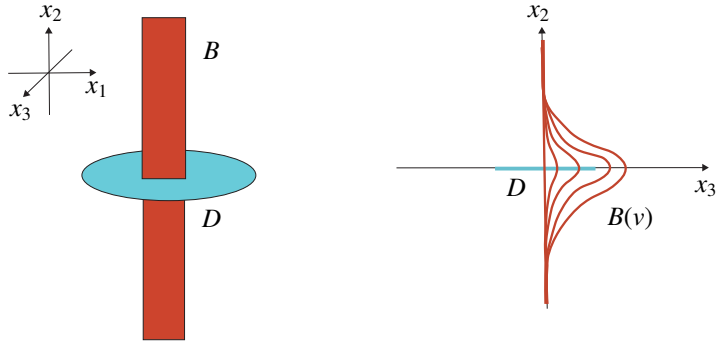
\begin{figure}[htpb]
\labellist
\small \hair 10pt
\pinlabel $B$ at 350 450
\pinlabel $D$ at 350 230
\pinlabel $B(v)$ at 900 230
\pinlabel $D$ at 750 270

\pinlabel $x_2$ at 780 500
\pinlabel $x_3$ at 1000 270

\pinlabel $x_3$ at 100 370
\pinlabel $x_2$ at 140 500
\pinlabel $x_1$ at 200 400

\endlabellist

\centering
\includegraphics[width = 12cm]{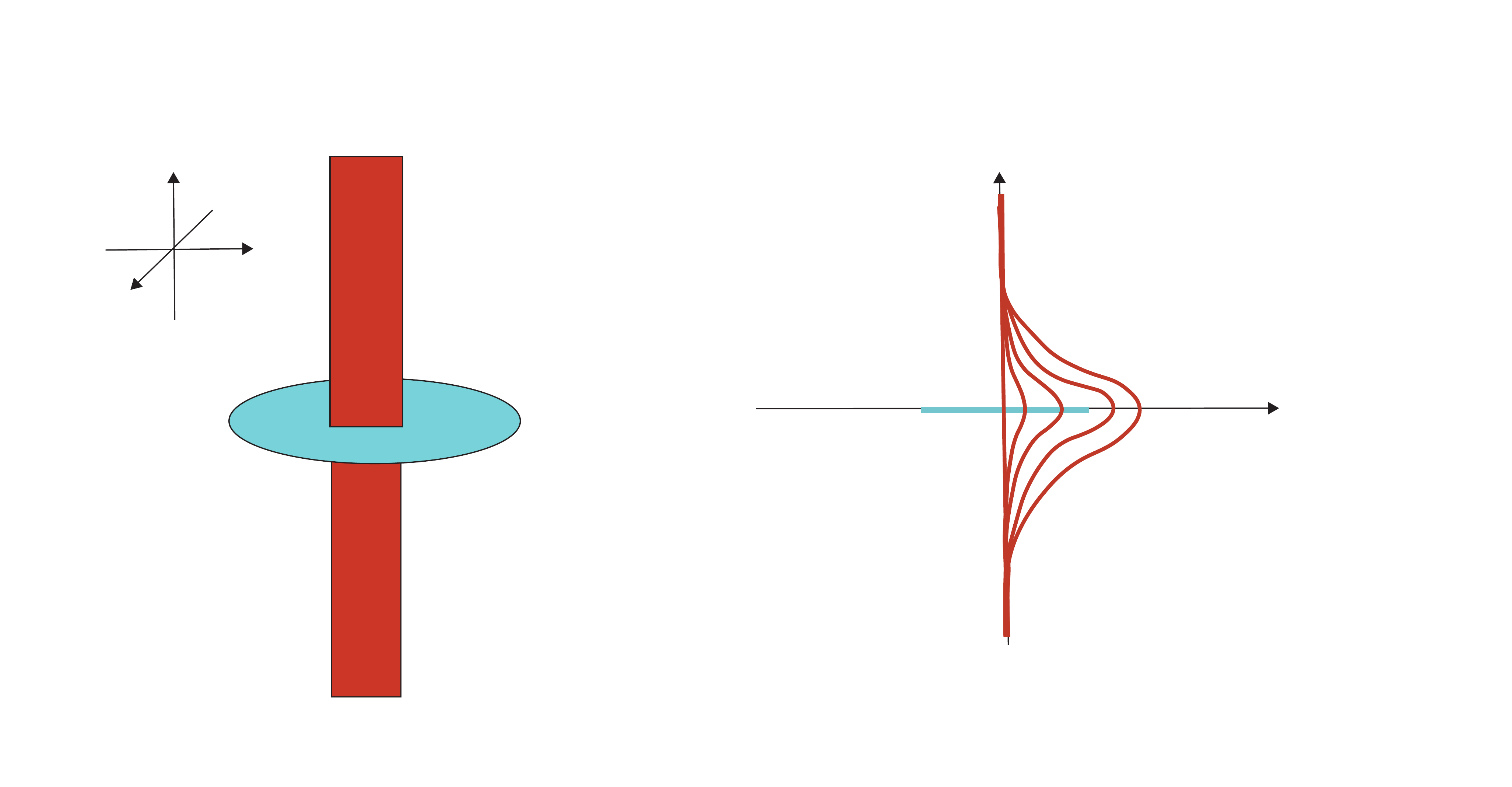}
\caption{Perturbation of a crossing}
\label{Perturbation of a ribbon presentation illust}
\centering
\end{figure}

\begin{definition}
We define $B_i(v_i) (v_i \in S) $ as the band in $\mathbb{R}^3 \times \mathbb{R}^{n-j-2}$ constructed from $B_i$ by replacing $B \times \mathbf{0}$ by $B(v_i)$.
More precisely, let $U_i \subset  \mathbb{R}^3 \times \mathbb{R}^{n-j-2}$ be a small neighborhood of the crossing $B_{i} \cap D_i$. and let 
\[
\xi_i : U_i \rightarrow [-3,3] ^{n-j+1}
\]
be a local coordinate that maps $B_i$ to $B \times \mathbf{0} $ and $D_i$ to $D \times \mathbf{0}$. Then $B_i(v_i)$ coincides with 
\[
(B_i\setminus U_i) \bigcup \xi_i^{-1}(B(v)).
\]
\end{definition}

\begin{notation}
Let $\mathbf{v} = (v_1, \dots, v_k) \in  {(S^{n-j-2})}^k$. We define $\mathcal{B}_{\mathbf{v}}$ by 
\[
\mathcal{B}_{\mathbf{v}} =  B_1 (v_1) \cup B_2 (v_2) \cup \dots \cup B_k (v_k).
\]
The perturbation $P_{\mathbf{v}} \subset \mathbb{R}^{n-j+1}$ of the wheel-like ribbon presentation $P$ is defined by $P_{\mathbf{v}} = \mathcal{D} \cup \mathcal{B}_{\mathbf{v}} $. Set
\[
V_{P_{\mathbf{v}}} = \left(\mathcal{B}_{\mathbf{v}} \times [-\frac{1}{4}, \frac{1}{4}]^{j-1} \right) \bigcup  \left(\mathcal{D} \times [-\frac{1}{2}, \frac{1}{2}]^{j-1}\right) \subset \mathbb{R}^n. 
\]
\end{notation}

\begin{definition}
As a submanifold in $\mathbb{R}^n$, the perturbation $\varphi^\mathbf{v}_k$ of $\varphi_k$ is defined by the connected sum $\partial V_{P_{\mathbf{v}}} \# \iota(\mathbb{R}^j)$. The parametrization of this submanifold is described in subsection \ref{The path to the trivial immersion}.
\end{definition}

\begin{definition}
We define the wheel-like cycle $c_k$ by
\[
c_k : (S^{n-j-2})^k \longrightarrow \text{Emb}(\mathbb{R}^j, \mathbb{R}^n),\quad \mathbf{v} \mapsto \varphi^\mathbf{v}_k.
\]
\end{definition}

\begin{notation}\cite [Definition 4.6] {SW}
Let $\hat{U}_i = U_i \times [-\frac{3}{4}, \frac{3}{4}]^{j-1}$. Inside the neighborhood $\hat{U}_i$, there are two components $\hat{D}_i$ and $\hat{B}_i(v_i)$ of $\varphi^{\mathbf{v}}_k$. $\hat{D}_i$ is homeomorphic to the punctured $j$-disk, while $\hat{B}_i(v_i)$ is homeomorhic to the annulus $I \times S^{j-1}$. $\hat{D}_i$ and $\hat{B}_i(v_i)$ correspond to $D_i$ and $B_i(v_i)$ respectively.
The pair $(\hat{U}_i , \hat{D}_i \cup \hat{B}_i(v_i))$ is called the $i$-th crossing of $\varphi^\mathbf{v}_k$.
\end{notation}

\subsection{The path to the trivial immersion}
\label{The path to the trivial immersion}
In the previous subsections, we have given a continuous  $(S^{n-j-2})^{\times k}$-family of embedded $\mathbb{R}^j$ in $\mathbb{R}^n$. To give a parametrization to this submanifold $\varphi^\mathbf{v}_k$ for each $\mathbf{v}$, we consider sequential resolutions of crossings. (See \cite[Remark 4,3]{SW}.) These resolutions are possible in the space of immersions. More precisely, we consider a path from $\varphi^\mathbf{v}_k$ to the trivially immersed $\mathbb{R}^j$, associated with the following moves.
\begin{itemize}
\item [($m1$)] Pull the disk $D_i$ of each crossing to the $x_1$-direction so that $D_i$ and $B_i(v_i)$ do not intersect.
\item [($m2$)] Pull back  $D_i$ and $B_i(v_i)$ to the based disk $D_0$.
\end{itemize}

Using this path, we can equip each point of the submanifold $\varphi^\mathbf{v}_k$ with a coordinate. This parametrization is continuous with respect to $\mathbf{v}$. Moreover, this path gives a lift $\widetilde{\varphi^\mathbf{v}_k}$ of $\varphi^\mathbf{v}_k$ to $\overline{\text{Emb}}(\mathbb{R}^j, \mathbb{R}^n)$. We define the cycle $\widetilde{c_k} : (S^{n-j-2})^k \longrightarrow \overline{\text{Emb}}(\mathbb{R}^j, \mathbb{R}^n)$ by $\mathbf{v} \mapsto \widetilde{\varphi^\mathbf{v}_k}$.

\subsection{Crossings as embeddings from standard disks and annuli}
In \cite{SW}, Sakai and Watanabe deformed $\varphi^\mathbf{v}_k$ by isotopy to give embeddings $\varphi^\mathbf{v, \varepsilon}_k$. In this deformation, we ``shrink''  $\hat{D}_i$ and $\hat{B}_i(v_i)$ of $\varphi^\mathbf{v}_k$. See \cite[4.2.1]{SW} for the precise definition. This shrinking aims to make pairing arguments with cocycles easier. We write $\hat{D}_i(\varepsilon)$ (resp. $\hat{B}_i(v_i, \varepsilon))$) for the image of $\hat{D}_i $ (resp. $\hat{B}_i(v_i))$ by this deformation.

The $i$-th crossing $(\hat{U}_i, \hat{D}_i(\varepsilon) \cup  \hat{B}_i(v_i, \varepsilon))$ of $\varphi^{\mathbf{v}, \varepsilon}_k$ is described  as an image of a standard disk and annulus $\mathbf{D}_i(\varepsilon) \cup \mathbf{B}_{i}(\varepsilon)$: 
\begin{align*}
\mathbf{D}_i(\varepsilon) &= \{ (x_1, \dots, x_j) \in \mathbb{R}^j \, |\, (x_1-p_i)^2 + x_2^2 + \dots + x_j^2 \leq (\varepsilon^2)^2\}\quad (1 \leq i \leq k),\\
\mathbf{B}_i(\varepsilon) &= \{ (x_1, \dots, x_j) \in \mathbb{R}^j \, |\, (3\varepsilon/4)^2 \leq (x_1-p_{i-1})^2 + x_2^2 + \dots + x_j^2 \leq \varepsilon^2\} \quad (1 \leq i \leq k+1),
\end{align*}
where $p_i = i/k\, (1\leq i \leq k-1)$ and $p_0 = p_k = 0$.  Recall that $\mathbf{B}_1 (\epsilon) = \mathbf{B}_{k+1} (\epsilon) $. 

The subset $\cup_i\mathbf{D}_i \bigcup \cup_i \mathbf{B}_i(\varepsilon)$ looks like ``planetary systems'' on $\mathbb{R}^j$. (See Figure \ref{figure of embeddings from standard disks and annuli}.)

\begin{figure}[htpb]
\begin{center}
\tikzset{every picture/.style={line width=1pt, xscale=0.5pt, yscale = 0.5pt}}

\begin{tikzpicture}

\draw (2,-1) rectangle (26, 7);

\draw (6,3) circle(2) [line width = 5pt] [color= {rgb, 255:red, 0; green, 0; blue, 0 }, fill opacity =0.5];
\draw [fill = black] (6,3) circle(0.3);
\draw (6,4) node  {$\mathbf{D}_3(\varepsilon)$};
\draw (6,0) node  {$\mathbf{B}_1(\varepsilon) = \mathbf{B}_4(\varepsilon)$};

\draw (14,3) circle(2) [line width = 5pt] [color= {rgb, 255:red, 0; green, 0; blue, 0}, fill opacity =0.5];
\draw [fill = black] (14,3) circle(0.3);
\draw (14,4) node  {$\mathbf{D}_1(\varepsilon)$};
\draw (14,0) node  {$\mathbf{B}_2(\varepsilon)$};

\draw (22,3) circle(2) [line width = 5pt] [color= {rgb, 255:red, 0; green, 0; blue, 0}, fill opacity =0.5];
\draw [fill = black] (22,3) circle(0.3);
\draw (22,4) node  {$\mathbf{D}_2(\varepsilon)$};
\draw (22,0) node  {$ \mathbf{B}_3(\varepsilon)$};

\end{tikzpicture}
\end{center}
\caption{$\varphi^\mathbf{v, \varepsilon}_3$ as embeddings from standard disks and annuli}
\label{figure of embeddings from standard disks and annuli}

\end{figure}
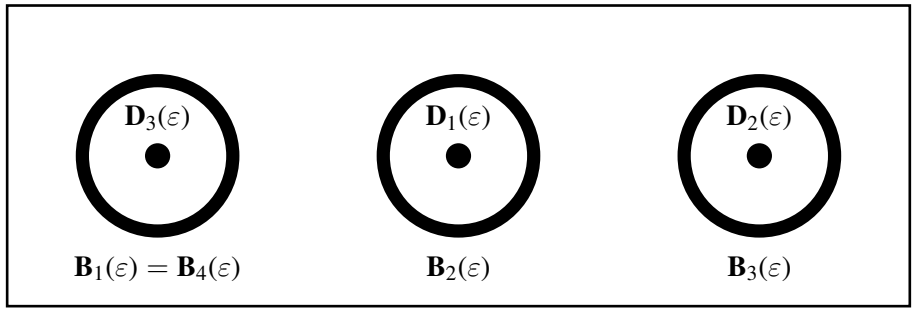

\begin{rem}
In Section \ref{Construction of the cycles}, the cycle $c_k: \mathbf{v} \mapsto \varphi^{\mathbf{v}, \varepsilon}_k$ is described as the chord diagram on $k$ directed lines as in Figure \ref{chorddiagramforSW}.
\end{rem}

\begin{figure}[htpb]

\begin{center}
\begin{tikzpicture}[x=1pt,y=1pt,yscale=0.2,xscale=0.3,baseline=20pt, line width = 1pt]
\draw [color={rgb, 255:red, 0; green, ; blue, 255 }, line width=1pt]  [-Stealth] (0,-100)--(0, 400);
\draw [color={rgb, 255:red, 0; green, 0; blue, 255 }, line width=1pt]  [-Stealth] (200,-100)--(200, 400);
\draw [color={rgb, 255:red, 0; green, 0; blue, 255 }, line width=1pt]  [-Stealth] (400,-100)--(400, 400);

\draw [color={rgb, 255:red, 0; green, 0; blue,0 }, line width=1pt]   [Stealth-] (0,300)--(200,0);
\draw (60, 280) node {$(1)$};
\draw [color={rgb, 255:red, 0; green, 0; blue,0 }, line width=1pt]   [Stealth-] (200,300)--(400,0) ;
\draw (260, 280) node {$(2)$};
\draw [color={rgb, 255:red, 0; green, 0; blue,0 }, line width=1pt]   [Stealth-] (400,300)--(0,0) ;
\draw (360, 230) node {$(3)$};

\draw  [fill={rgb, 255:red, 0; green, 0; blue, 0 }  ,fill opacity=1 ] (0,0) circle (5) ;
\node at (-20,0) {$O$};
\draw [fill={rgb, 255:red, 0; green, 0; blue, 0 }  ,fill opacity=1 ] (0,300) circle (5);

\draw  [fill={rgb, 255:red, 0; green, 0; blue, 0 }  ,fill opacity=1 ] (200,0) circle (5) ;
\node at (220,0) {$O$};
\draw [fill={rgb, 255:red, 0; green, 0; blue, 0 }  ,fill opacity=1 ] (200,300) circle (5);

\draw  [fill={rgb, 255:red, 0; green, 0; blue, 0 }  ,fill opacity=1 ] (400,0) circle (5) ;
\node at (420,0) {$O$};
\draw [fill={rgb, 255:red, 0; green, 0; blue, 0 }  ,fill opacity=1 ] (400,300) circle (5);

\end{tikzpicture}
\end{center}
\caption{the cycle $c_3$ as a chord diagram on three directed lines}
\label{chorddiagramforSW}
\end{figure}
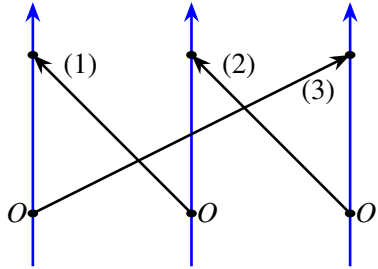

\section{Construction of the cycles}
\label{Construction of the cycles}
In this section, we construct a cycle of the space of long embeddings from what we call ``a set of planetary systems'' on $\mathbb{R}^j$.
Furthermore, we will see that such a set of planetary systems is described by a chord diagram on directed lines. The sets of planetary systems corresponding to Sakai and Watanabe's wheel-like cycles are the ones in which each planetary system has exactly one planetary orbit. (See Figure \ref{figure of embeddings from standard disks and annuli}.) In terms of ribbon presentations, we use ribbon presentations with more than one node, while wheel-like ribbon presentations have exactly one node, the based disk $D_0$.

\subsection{Planetary  systems and chord diagrams on directed lines}
Using the terminology ``planetary system" for configuration spaces appears in Sinha's work \cite{Sin}.
\begin{definition}[A set of planetary systems on $\mathbb{R}^j$]
A set of planetary systems $\mathcal{S}$ on $\mathbb{R}^j$ of order $k$ consists of the following.
\item [(1)] Points $a_i$ at $(i,0,\dots,0) \in \mathbb{R}^j \quad (i=1, \dots, s)$.
\item [(2)] Spheres $S(L^l_i, a_i)$\quad$(l=1, \dots, t_i)$ with radius $L^j_i$ centered at $a_i$ $(i = 1,\dots, s)$.

Here, $s$ and $t_i\ (i=1, \dots, s)$ are non-negative integers such that the total number of points and spheres is $2k$. \footnote{$s$ and $t_i$ here are not related to $s$ and $t$ of $C_{s,t}$. } We also assume $ L^l_i$ is sufficiently smaller than $1$, and assume $L^l_i$ is sufficiently smaller than $L^{l^{\prime}}_i$ if $l<l^{\prime}$. For example, set $\delta = \frac{1}{10^N}$ and $L^l_i = \delta^{(t_i +1- l)}.$

Below we call points of (1)  fixed stars. We call spheres of (2)  planetary orbits. Each fixed star forms one planetary system which consists of $t_i$ planetary orbits.

\end{definition}

\begin{definition}
An ordered pairing $\{p_i\}_{i =1, \dots, k}$ for $\mathcal{S}$ is a pairing among the set of fixed stars and planetary orbits.
Fixed stars must be paired with planetary orbits. Planetary orbits may be paired with other planetary orbits (orbit-orbit pairing). The order of two elements of each pair is taken so that the fixed star (if it exists) is the first. 
\end{definition}

A set of planetary systems and its ordered pairing can be described by a chord diagram on directed lines.
\begin{definition}[Chord diagram on directed lines]
A chord diagram on $s$ directed lines of order $k$ consists of the following data.
\begin{itemize}
\item  Integers $t_i \geq 0$ ($i = 1, \dots, s$) such that
\[
\sum_{i=1}^s (t_i +1) = 2k.
\]
\item  An ordered pairing $\{p_i\}_{i=1,\dots,k}$ among $2k$ points $V = \{(i,l)\in \mathbb{Z}^2 | 1\leq i \leq s, 0\leq l \leq t_i\}$.
\end{itemize}
Such a chord diagram is drawn as \chorddiagramb{}. The ordered pairing must be taken so that it satisfies
\begin{itemize}
\item At least one of the two points of each pair is not on the $x$-axis (that is, its second coordinate is not zero).
\item Any point $v = (i, 0)$ on the $x$-axis must be the first of some ordered pair.
\end{itemize}
\end{definition}

A set of planetary systems and its ordered pairing determines a chord diagram on directed lines in an expected way. Namely, points on the $x$-axis are fixed stars. Other points are planetary orbits. Conversely, a chord diagram on directed lines determines a set of planetary systems (up to rescalings) and its ordered pairing.

Let $C$ be a chord diagram on directed lines. Remember that the set of chords (pairings) $E(C)$ is ordered, and each chord is oriented.
\begin{definition}[Induced ordering of $V(C)$]
\label{Induced ordering}
We order the set of vertices $V(C)$ in a canonical way using the above ordering and orientations. That is, the initial vertex of the $i$th chord is labeled by $(2i-1)$. The end vertex of the $i$th chord is labeled by $2i$. 
\end{definition}
Using the ordering of $V(C)$ , let the image of $\tau: \{1, \dots, g-1\} \rightarrow \{1, \dots, 2k\} =~V(C)$ be the first vertices of orbit-orbit pairing; $g = k-s$. We assume that $\tau$ preserves the orders.

\subsection{Construction of the cycles}
\label{Construction of the cycles sub}

Once we fix a set of planetary systems and its ordered pairing, namely a chord diagram on directed lines, we can construct a ``generalized ribbon cycle''. These generalized cycles have more parameters than the wheel-like cycles in Section \ref{Review of Sakai and Watanabe's construction of the wheel-like cycles}.

\begin{definition}[Generalized ribbon cycles]
Let $\mathcal{S}$ be a set of planetary systems of order $k$. 
Let $\{p_i\}_{i =1, \dots, k}$ be an ordered pairing.
Let $(g-1)$ be the number of orbit-orbit pairings. 
From this data, we construct a $k(n-j-2)+(g-1)(j-1)$ cycle $\psi = \psi (\mathcal{S}, \{p_i\}_{i =1, \dots, k})$ in the following. We call this cycle the generalized ribbon cycle associated with $(\mathcal{S}, \{p_i\}_{i =1, \dots, k})$.
For simplicity, we only construct for the following case:
\begin{itemize}
\item There is only one fixed star $a_1$ (and hence one planetary system).
\item There are three planetary orbits $S(L_l, a_1) (l=1,2,3)$.
\item The pairing is taken as $(a_1, S(L_2))$, $(S(L_1), S(L_3))$.
\end{itemize}
The corresponding chord diagram is $C_1 = \chorddiagrama{}$. The cycles corresponding to other $(\mathcal{S}, \{p_i\}_{i =1, \dots, k})$ are constructed similarly.

Below we construct a cycle in $\text{Emb} (\mathbb{R}^j, \mathbb{R}^n)$
\[
\phi: S^{j-1} \times S^{n-j-2} \times S^{n-j-2}  \rightarrow \text{Emb} (\mathbb{R}^j, \mathbb{R}^n).
\]
Let $S(L)$ and $D(L)$ denote the $(j-1)$-sphere and $j$-ball respectively, of radius $L$ centered at O in $\mathbb{R}^{j}$. And let $L_0, L_1, L_2, L_3 \in \mathbb{R}$ satisfy
\[
0< L_0 < L_1< L_2 < L_3 < L_4=1, 
\]
where $L_i$ is sufficiently smaller than $L_{i+1}$. (For example, set $L_i = \delta^{4-i}$). For each $\theta\in S^{j-1}$, we write $D_{\theta}(\varepsilon)$ for the ball of radius $\varepsilon$ centered at $\theta \in S(L_1)$. We assume $\varepsilon$ is sufficiently smaller than $L_1$. (See Figure \ref{Preimage}.) Step~1 and Step~2 below are to construct a cycle $\phi_{\theta} : S^{n-j-2} \times S^{n-j-2}  \rightarrow \text{Emb}(\mathbb{R}^j, \mathbb{R}^n) $ for each parameter $\theta \in S^{j-1}$.  Let $0< h_4<h_3<h_2<h_1<h_0$ be coordinates of the $x_{j+1}$ axis.

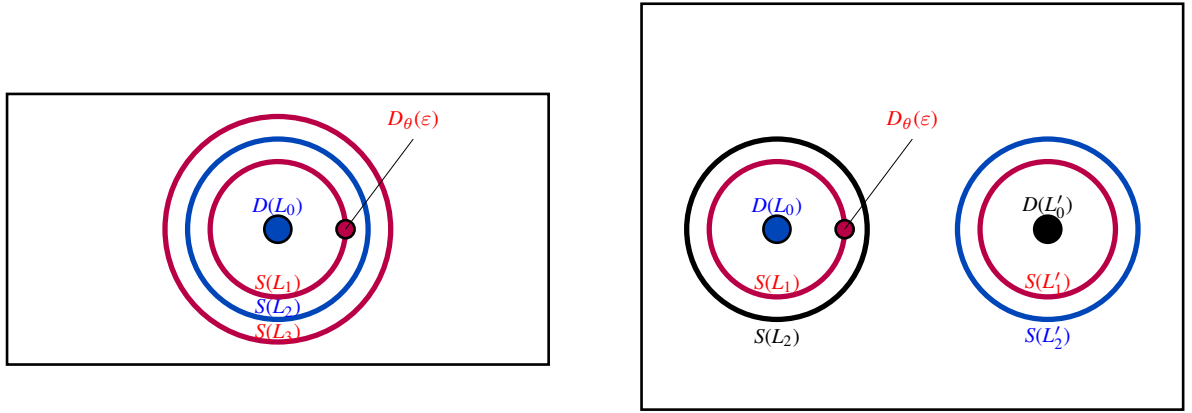
\begin{figure}
\begin{center}
\tikzset{every picture/.style={line width=1pt, xscale=0.6pt, yscale = 0.6pt}}  
\begin{tikzpicture}

\draw (0,0) rectangle (12, 6);

\draw (6,3) circle(2.5) [line width = 2pt] [color = {rgb, 255:red, 185; green, 0; blue, 70 }, fill opacity =1.0 ];
\node  [red] at (6, 0.7) {\scriptsize $S(L_3)$};

\draw (6,3) circle(2) [line width = 2pt] [color = {rgb, 255:red, 0; green, 70; blue, 185 }, fill opacity =1.0];
\node [blue] at (6, 1.25) {\scriptsize $S(L_2)$};

\draw (6,3) circle(1.5) [line width = 2pt] [color = {rgb, 255:red, 185; green, 0; blue, 70 }, fill opacity =1.0];
\node [red] at (6, 1.8) {\scriptsize $S(L_1)$};

\draw (7.5, 3) circle (0.2) [fill={rgb, 255:red, 185; green, 0; blue, 70}, fill opacity =1.0] ; 
\draw  [line width=0.3pt] (7.5,3)--(9,5) node  [anchor= south] [red]  {\scriptsize $D_{\theta}({\varepsilon})$};

\draw (6,3) circle(0.3) [fill={rgb, 255:red, 0; green, 70; blue, 185}, fill opacity =1.0] ; 
\node [blue] at (6, 3.5) {\scriptsize $D(L_0)$};

\begin{scope}[xshift =400]

\draw (0,-1) rectangle (12, 8);

\draw (9,3) circle(2) [line width = 2pt] [color = {rgb, 255:red, 0; green, 70; blue, 185 }, fill opacity =1.0];
\node  [blue] at (9, 0.6) {\scriptsize $S(L^{\prime}_2)$};

\draw (9,3) circle(1.5) [line width = 2pt] [color = {rgb, 255:red, 185; green, 0; blue, 70 }, fill opacity =1.0];
\node  [red] at (9, 1.8) {\scriptsize $S(L^{\prime}_1)$};

\draw (3,3) circle(2) [line width = 2pt] [color = {rgb, 255:red, 0; green, 0; blue, 0 }, fill opacity =1.0];
\node  at (3, 0.6) {\scriptsize $S(L_2)$};

\draw (3,3) circle(1.5) [line width = 2pt] [color = {rgb, 255:red, 185; green, 0; blue, 70 }, fill opacity =1.0];
\node [red] at (3, 1.8) {\scriptsize $S(L_1)$};

\draw (4.5, 3) circle (0.2) [fill={rgb, 255:red, 185; green, 0; blue, 70}, fill opacity =1.0] ; 
\draw  [line width=0.3pt] (4.5,3)--(6,5) node  [anchor= south] [red]  {\scriptsize $D_{\theta}({\varepsilon})$};

\draw (3,3) circle(0.3) [fill={rgb, 255:red, 0; green, 70; blue, 185}, fill opacity =1.0] ; 
\node [blue] at (3, 3.5) {\scriptsize $D(L_0)$};

\draw (9,3) circle(0.3) [fill={rgb, 255:red, 0; green, 0; blue, 0}, fill opacity =1.0]  ; 
\node at (9, 3.5) {\scriptsize $D(L^{\prime}_0)$};

\end{scope}

\end{tikzpicture}

\caption{Preimage  on $\mathbb{R}^j$ of $\phi$ (left) and $\psi$ (right)}
\label{Preimage}
\end{center}
\end{figure}

\begin{itemize}
\item[\underline{Step 1} :]
Deform $D(L_3)$ by isotopy in $\mathbb{R}^{j}\times \mathbb{R} \subset \mathbb{R}^{n}$ to the $x_{j+1}$ direction,  so that $\mathbb{R}^j$ is embedded in  $\mathbb{R}^{j+1}$ as 
\[
S(r)\times (h_4, h_0) 
\]
in the range $h_4<  x_{j+1}<h_0$. Here $S(L_3)$, $S(L_2)$ and $S(L_1)$ are mapped to $S(r) \times \{h_3\}$ , $S(r) \times \{h_2\}$ and $S(r) \times \{h_1\}$ respectively. We assume $r$ is sufficiently small compared to  $\text{min}\{|h_{i}-h_{i+1}|\}_i$.

\item[\underline{Step 2} :]
For $L = L_2, L_3$, let $S(L) \times I$ be the tubular neighborhood of $S(L)$ in $\mathbb{R}^{j}$ with width $2\varepsilon$. 
Construct a ribbon crossing of  $S(L_3) \times I $ and $D_{\theta}(\varepsilon)$ near  $S(r)\times h_3$, and a ribbon crossing of $S(L_2)\times I$ and $D(L_0)$ near $S(r)\times h_2$. Here, a ribbon crossing is the crossing given by the perturbation of a crossing of a ribbon presentation as in subsection \ref{Perturbation of crossings}. See Figure \ref{Image}.
Then we obtain a cycle
\[
\phi_{\theta}: S^{n-j-2} \times S^{n-j-2} \rightarrow \text{Emb}(\mathbb{R}^j, \mathbb{R}^n).
\]
We describe a ribbon presentation corresponding the cycle $\phi_{\theta}$ in subsection~\ref{Ribbon presentations with more than one node}. By choosing a ribbon presentation suitably, we can construct this cycle in the unknot component. 
\item[\underline{Step 3} :]
Move $D_\theta(\varepsilon)$ on $S(L_1)$. In $\mathbb{R}^{j+1}$, we change the behavior only near $h_1$. At the height $h_1$, we turn the branch tube around the stem tube. See Figure~\ref{Image}. Then we get the desired cycle $\phi$. The parametrization of $\phi$ is given in subsection \ref{Lifting the cycle}.
\end{itemize}

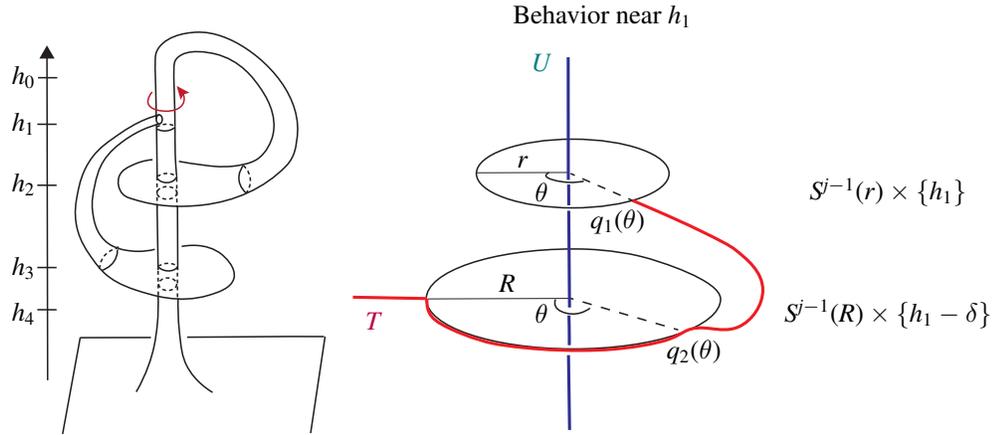
\begin{figure}[htpb]

\labellist
\small \hair 10pt
\pinlabel $h_0$ at 65 510
\pinlabel $h_1$ at 65 450
\pinlabel $h_2$ at 65 370
\pinlabel $h_3$ at 65 260
\pinlabel $h_4$ at 65 200

\pinlabel {Behavior near $h_1$} at 825 590

\pinlabel $R$ at 700 240
\pinlabel $r$ at 720 400
\pinlabel $\theta$ at 745 360
\pinlabel $\theta$ at 745 200
\pinlabel $q_1(\theta)$ at 845 320
\pinlabel $q_2(\theta)$ at 945 150

\pinlabel ${S^{j-1}(r) \times \{h_1\}}$ at 1200 360
\pinlabel ${S^{j-1}(R) \times \{h_1 - \delta\}}$ at 1200 200

\pinlabel \color{purple}{$T$} at 525 190
\pinlabel \color{teal}{$U$} at 745 530

\endlabellist

\begin{center}
\includegraphics[width = 12cm]{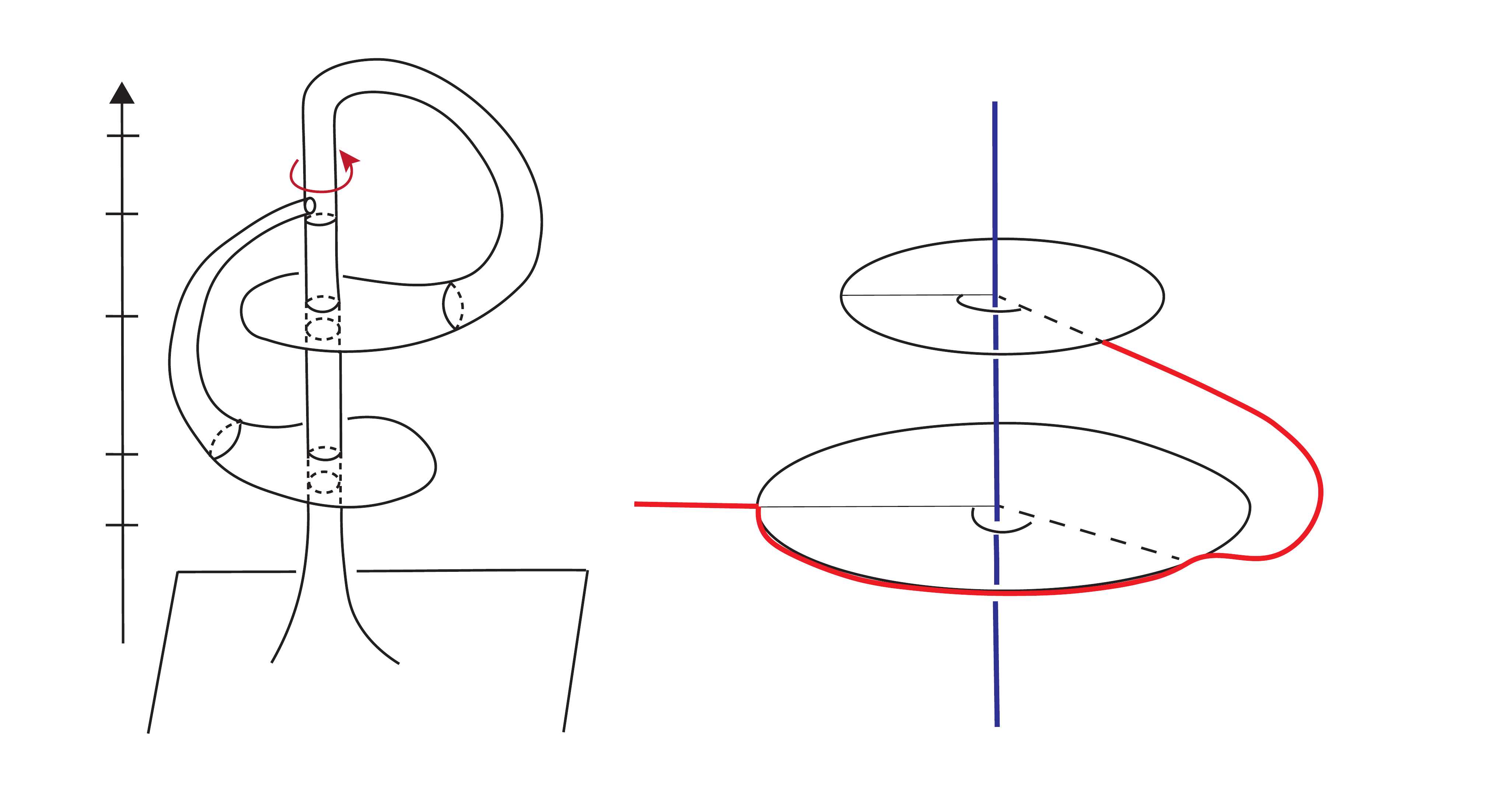}
\caption{Image on $\mathbb{R}^{j+1}$}
\label{Image}
\end{center}
\end{figure}

We describe Step 3 more precisely. We write $U$ for the stem tube $S^{j-1}(r)\times (h_4, h_0)$. In Figure \ref{Image}, we draw the core $c_U$ of $U$ in blue. For $\theta \in S^{j-1} $, let $\partial D_{\theta} (\varepsilon)$ be mapped to the $\theta$ side of $U$. 
For the initial point $ O= (1, 0,\dots,0) \in S^{j-1}$, we construct, near $h_1$, a branch tube $T = T(O)$ with small radius which starts from $U$ and goes to the $x_1$-direction. We draw the core $c = c(O)$ of $T$ in red.

For $\theta \in S^{j-1}$, we replace the core $c$ with $c(\theta)$ defined as follows. 
\begin{itemize}
\item $c(\theta)$ starts from the point $q_1(\theta)=(r\theta, h_1)\in S^{j-1}(r) \times \{h_1\}$ and goes straight to the point $(R\theta, h_1)\in S^{j-1}(R) \times \{h_1\} (R > r)$.
\item Next $c(\theta)$ runs down to $S^{j-1}(R)\times \{h_1-\delta\}$ preserving $\theta \in S^{j-1}(R)$ and reaches the point $q_2(\theta)=(R \theta, h_1 -\delta)$. 
\item Then $c(\theta)$ connects $q_2(\theta)$ and $q_2(O)$ through the shortest geodesic of $S^{j-1}(R)$. 
\end{itemize}

However, there is one problem. When $\theta$ is the antipode $A \in S^{j-1}$ of the base point $O$, there are infinitely many geodesics from $q_2(A)$ to $q_2(O)$. 
To avoid this, $c(\theta)$ takes a short-cut from $q_2(\theta)$ to $q_2(O)$ inside the ball $D^{j}(R)$, when $\theta$ is in a small neighborhood of $A$. At the same time, when $\theta$ is near $A$, we perturb $c(\theta)$ to  the $x_{j+2}$-direction. 


\begin{figure}[htpb]
\labellist
\small \hair 10pt
\pinlabel ${||\theta|| = 0}$ at 200 580
\pinlabel ${0<||\theta||\leq 0.95 \pi}$ at 570 580
\pinlabel ${0.95 \pi<||\theta||< \pi}$ at 970 580

\pinlabel ${||\theta|| = \pi}$ at 200  280
\pinlabel ${0.95 \pi<||\theta||< \pi}$ at 570 280
\pinlabel ${0 <||\theta|| \leq 0.95\pi}$ at 970 280

\pinlabel $\color{purple}T$  at 100  370
\pinlabel $\color{teal}U$  at 270  340
\pinlabel $S(r)$  at 310  480
\pinlabel $S(R)$  at 310  400

\endlabellist
\centering
\includegraphics[width = 10cm]{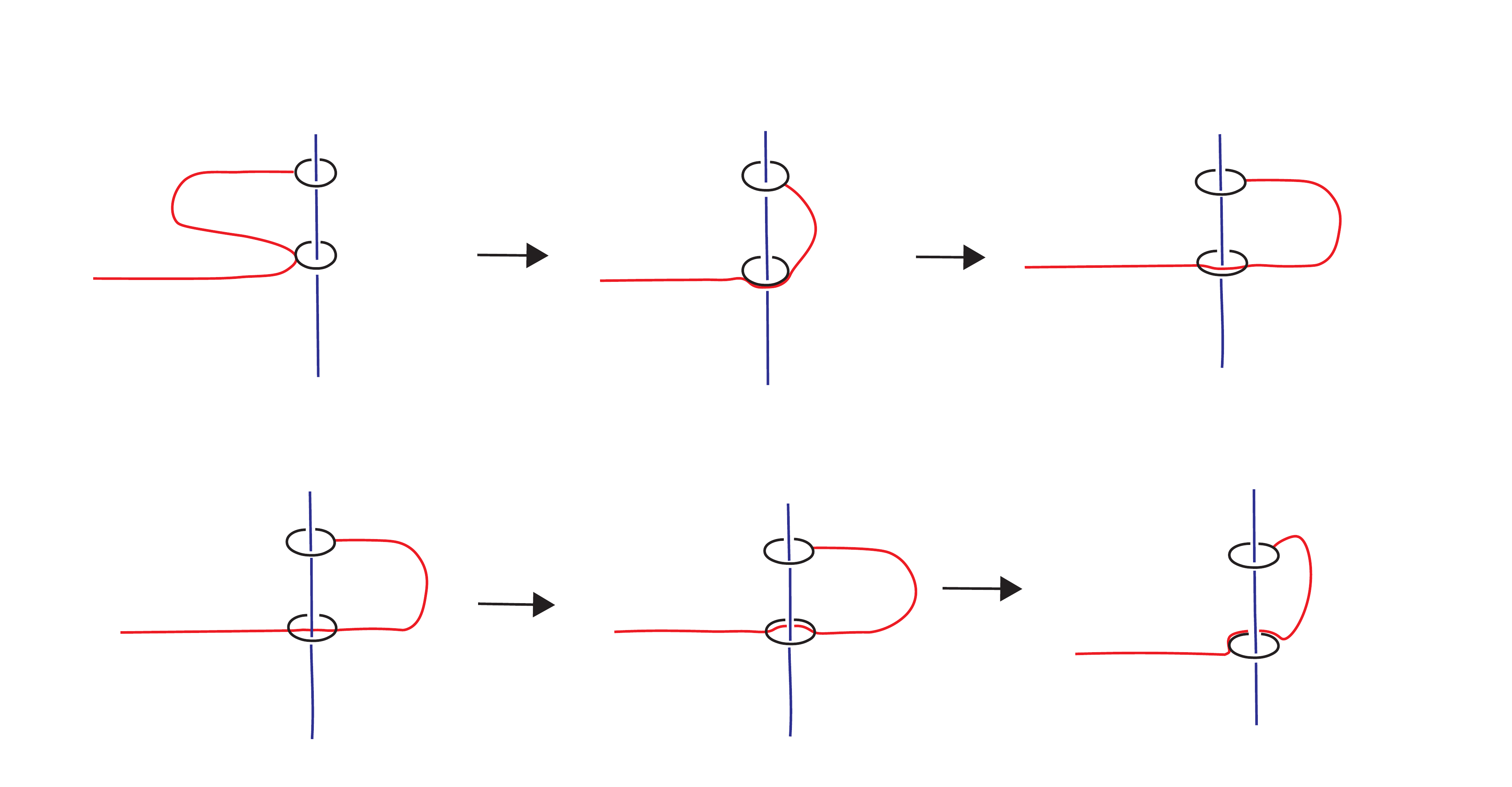}
\caption{The replaced core $c(\theta)$ near $h_1$ $(a = 0.05 \pi)$}
\label{Perturbing}
\centering
\end{figure}
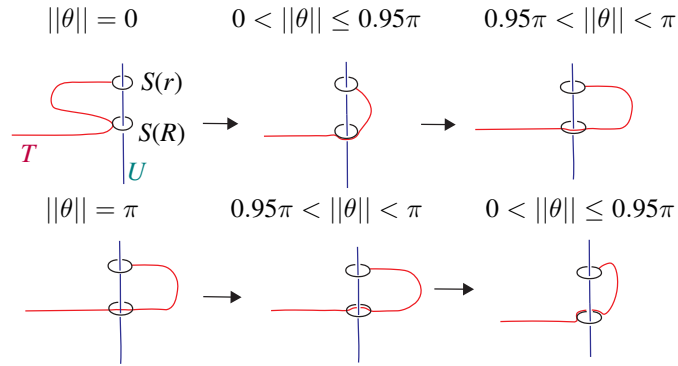

More precisely, let
\[
||\cdot|| : S^{j-1} \rightarrow [0, \pi ] 
\]
be the function that assigns the geodesic distance from the base point $O$. Then in the range $||\theta||> \pi - a$ for a sufficiently small $a$, we assume that the core $c(\theta)$ approaches the core $c_{U}$ of $U$, and assume that they intersect with each other in $\mathbb{R}^{j+1}$ when $||\theta||= \pi$. At this time, we perturb $c(\theta)$ in the positive direction of the $x_{j+2}$ axis. See Figure \ref{Perturbing}.

Then we take a tubular neighborhood $N(\theta) \in \mathbb{R}^{j+1}$, which is perturbed to the $x_{j+2}$ direction when $||\theta||> \pi - a$. The replaced tube $T(\theta)$ is defined as the boundary of $N(\theta)$. 

As a candidate of a dual cycle of $H$, we construct the generalized ribbon cycle in Figure \ref{Cycle psi} from the diagram \chorddiagramb{}. The ribbon presentation corresponding to the cycle $\psi$ is given in subsection 6.3. The cycle $\psi$ is also constructed in the unknot component. In general, cycles have parameters $\theta^{\prime}_i \in S^{j-1}$\ $(i = 1,\dots, g-1)$ and $y_j\in S^{n-j-2}$\ $(j=1, \dots, k)$. \footnote{We change the notation for the parameter moving on $S^{j-1}$, from ``\ $\theta$\ '' to ``\ $\theta^{\prime}$\ ''. In Section \ref{Some lemmas for computing cocycle-cycle pairing}, we use the letter ``\ $\theta$\ '' for a different parameter.}

\begin{figure}[htpb]
\centering
\includegraphics[width = 6cm]{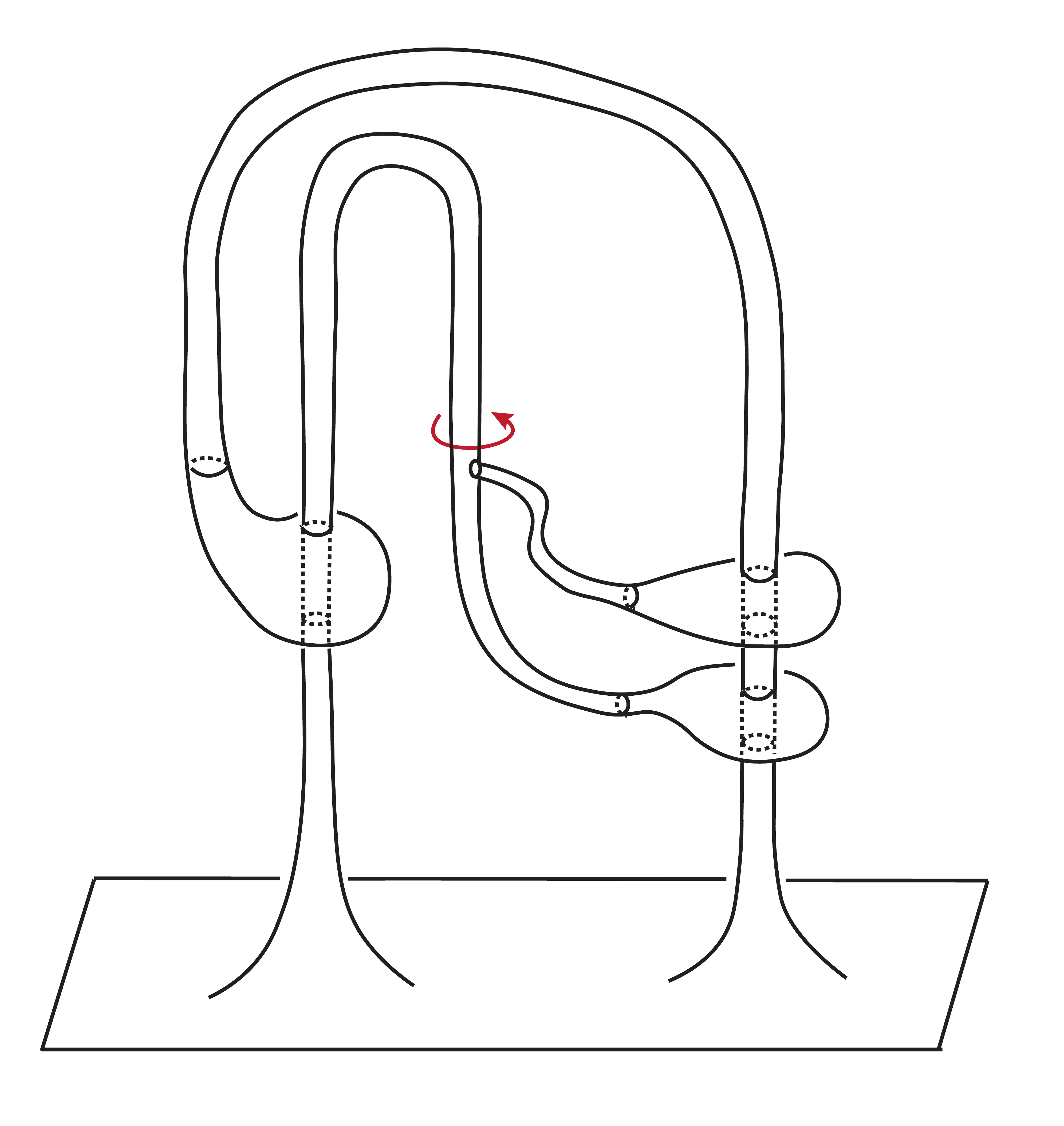}
\caption{Cycle $\psi$}
\label{Cycle psi}
\centering
\end{figure}

\end{definition}

\subsection{Ribbon presentations with more than one node}
\label{Ribbon presentations with more than one node}

\qquad Here, we give ribbon presentations corresponding to the (cycles of) long embeddings in subsection \ref{Construction of the cycles sub}. (We review notation of ribbon presentations at the beginning of this subsection.) These ribbon presentations have a node other than the based disk $D_0$, and this node is connected (by two bands) to two leaves. Using moves of ribbon presentations, we  show that  the cycles of subsection \ref{Construction of the cycles sub} are in the path component of the trivial long embedding. The key point is the way the two leaves intersect with bands.

\begin{definition}
A  ribbon presentation $P = \mathcal{D} \cup \mathcal{B}$ of order $k$ is a based oriented, immersed 2-disk in $\mathbb{R}^{3}$ such that
\begin{itemize}
\item $\mathcal{D} = D_0 \cup D_1 \cup \dots \cup D_k$ are disjoint $(k+1)$ disks,
\item $\mathcal{B}=  B_1 \cup B_2 \cup \dots \cup B_k$ are disjoint $k$ bands ($B_i \approx I \times I$),
\item each band connects two different disks,
\item $B_{i}$ can intersect transversally with the interior of disks,
\item the base point $\ast$ is on the boundary of $D_0$ (but not on the boundaries of $B_i$).
\end{itemize}
\end{definition}

\begin{notation}\cite[Definition 3.1]{HS}
A disk without intersecting bands is called a node.  A disk is called a leaf if it has at least one intersecting band and it is incident to exactly one band. We draw a grey disk for a node except for the based disk $D_0$. We call a ribbon presentation simple if each leaf intersects with exactly one band.
\end{notation}

In \cite{HS}, Habiro and Shima introduced moves of ribbon presentations that do not change the isotopy class of the ribbon $2$-knot represented by the presentation. The following are examples of the moves.
\begin{notation}
We define $S1$ and $S4$ moves as in Figure \ref{S1S4}.

\begin{figure}[htpb]

\labellist
\small \hair 10pt
\pinlabel $S1$ at 1350 1360
\pinlabel $S4$ at 1330 690
\endlabellist

\centering
\includegraphics[width = 10cm]{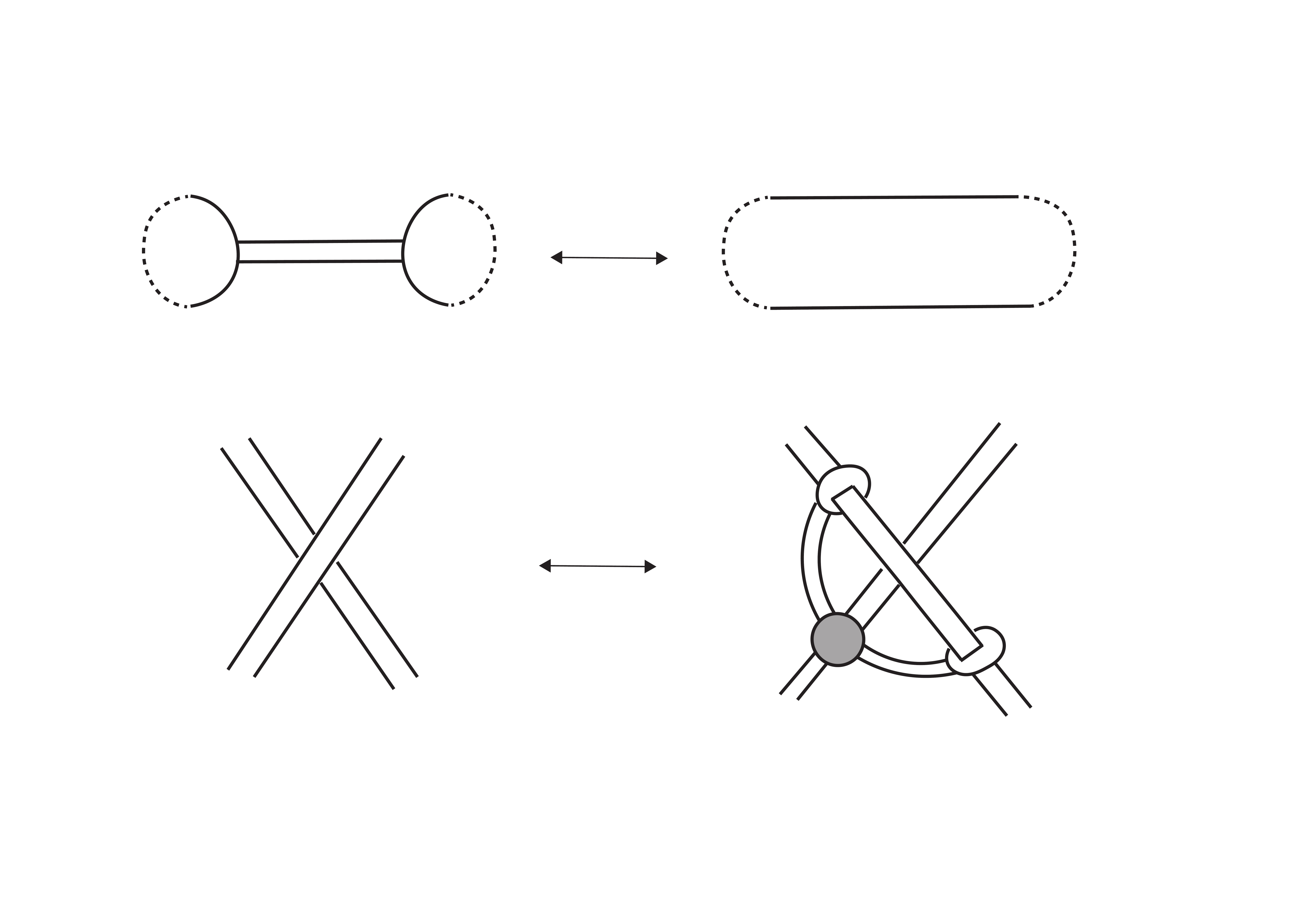}
\caption{$S1$ and $S4$ moves}
\label{S1S4}
\centering
\end{figure}

\end{notation}

To the chord diagrams $C_1 = \chorddiagrama{} $ and $C_2 = \chorddiagramb{} $, we assign the simple ribbon presentations $P_1$ in Figure \ref{Ribbon presentations P1} and $P_2$ in Figure   \ref{Ribbon presentations P2}.  Each directed line corresponds to a sequence of bands and nodes that connects the based disk $D_0$ and a leaf. The origin of each directed line corresponds to this leaf.  Each chord corresponds to a crossing. The initial point of each chord corresponds to the disk (leaf) of the crossing. If this initial point is not an origin, this leaf is connected to a node by a band. The target point of each chord corresponds to a small segment of the band of the crossing. 

\begin{figure}[htpb]

\labellist
\small \hair 10pt
\pinlabel $B_1$ at 1140 1050
\pinlabel $B_2$ at 1530 1550
\pinlabel $B_3$ at 530 1550

\pinlabel $D_0$ at 1650 750
\pinlabel $D_1$ at 850 1550
\pinlabel $D_2$ at 1250 1250
\pinlabel $D_3$ at 900 1950
\endlabellist

\centering
\includegraphics[width = 7cm]{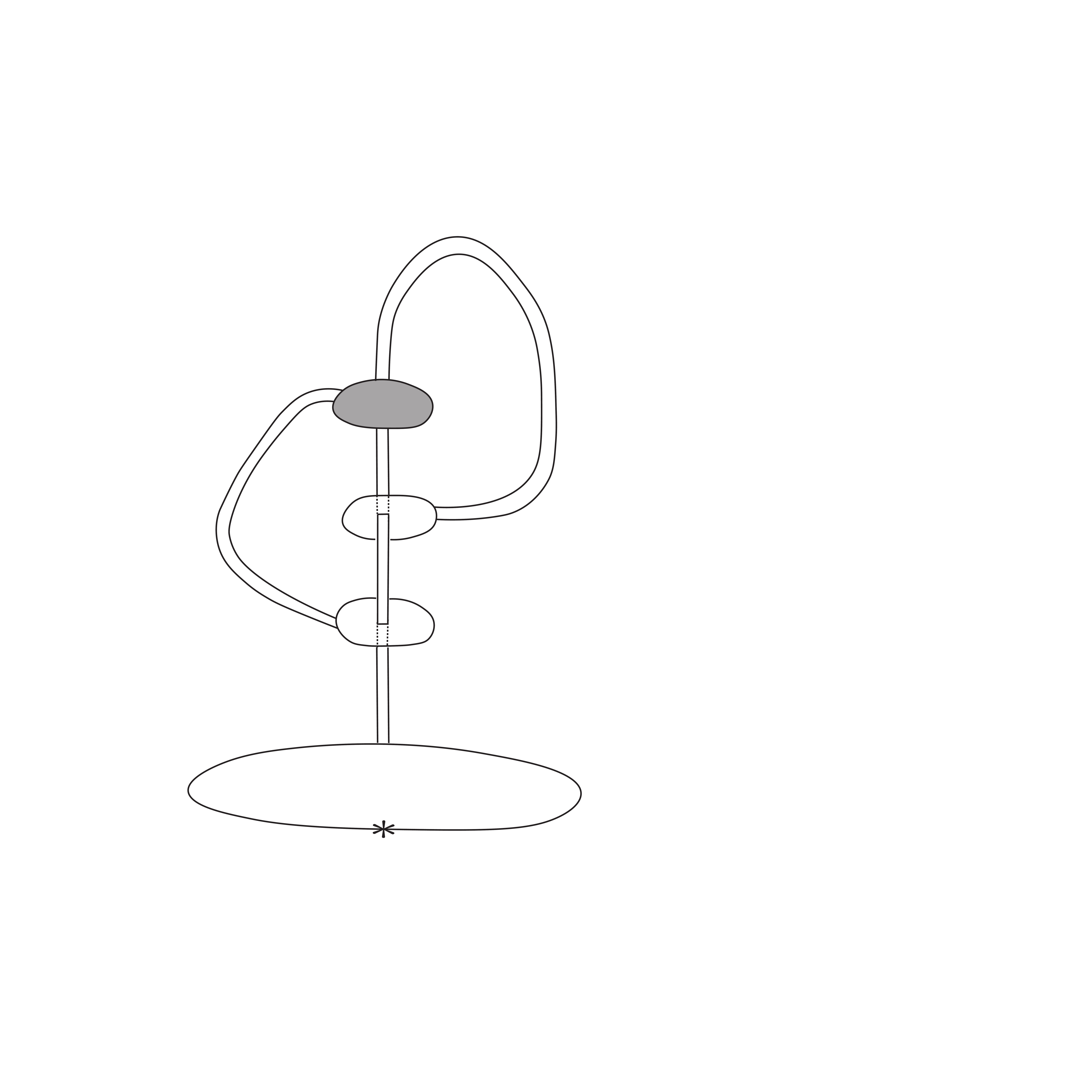}
\caption{Ribbon presentation $P_1$ corresponding to $C_1$}
\label{Ribbon presentations P1}
\centering
\end{figure}
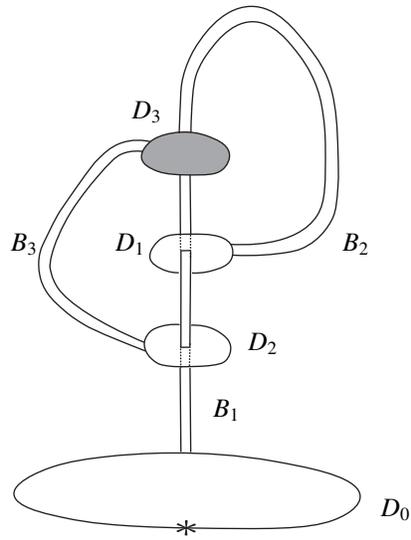

\begin{figure}[htpb]

\labellist
\small \hair 10pt
\pinlabel $B_1$ at 840 1050
\pinlabel $B_2$ at 1500 1050
\pinlabel $B_3$ at 1000 1650
\pinlabel $B_4$ at 1000 2000

\pinlabel $D_0$ at 1650 750
\pinlabel $D_1$ at 1600 1550
\pinlabel $D_2$ at 1600 1800
\pinlabel $D_3$ at 950 1200
\pinlabel $D_4$ at 850 1550
\endlabellist

\centering
\includegraphics[width = 6cm]{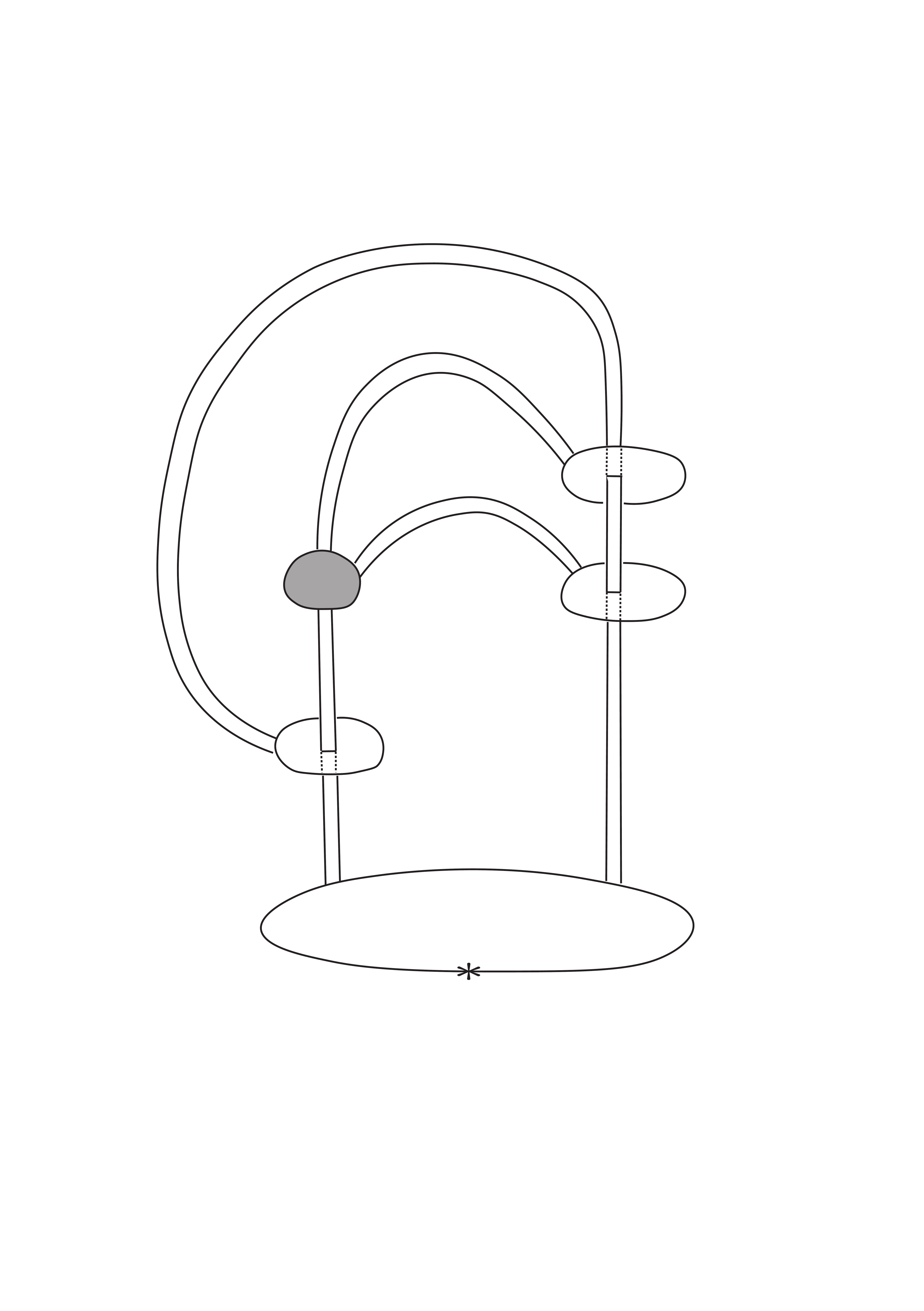}
\caption{Ribbon presentation $P_2$ corresponding to $C_2$}
\label{Ribbon presentations P2}
\centering
\end{figure}

Observe that our chord diagrams $C_1$ and $C_2$ have a subdiagram of the form
\begin{tikzpicture}[x=1pt,y=1pt,yscale=0.2,xscale=0.3,baseline=20pt, line width = 1pt]
\draw [color={rgb, 255:red, 0; green, ; blue, 255 }, line width=1pt]  [-Stealth] (0,-100)--(0, 400);
\draw [color={rgb, 255:red, 0; green, 0; blue, 255 }, line width=1pt]  [-Stealth] (200,-100)--(200, 400);
\draw [color={rgb, 255:red, 0; green, 0; blue,0 }, line width=1pt]   [-Stealth] (0,0)--(200,300) ;
\draw [color={rgb, 255:red, 0; green, 0; blue,0 }, line width=1pt]   [-Stealth] (0,150)--(200,150) ;

\draw  [fill={rgb, 255:red, 0; green, 0; blue, 0 }  ,fill opacity=1 ] (0,0) circle (5) ;
\node at (-20,0) {$O$};
\draw  [fill={rgb, 255:red, 0; green, 0; blue, 0 }  ,fill opacity=1 ] (0,150) circle (5) ;
\draw  [fill={rgb, 255:red, 0; green, 0; blue, 0 }  ,fill opacity=1 ] (200,150) circle (5) ;
\draw [fill={rgb, 255:red, 0; green, 0; blue, 0 }  ,fill opacity=1 ] (200,300) circle (5);
\end{tikzpicture}.
Recall that we have constructed a crossing for each chord. 
There are two ways to construct a crossing: in which orientation the band intersects with the disk. 
In the part 
\begin{tikzpicture}[x=1pt,y=1pt,yscale=0.2,xscale=0.3,baseline=20pt, line width = 1pt]
\draw [color={rgb, 255:red, 0; green, ; blue, 255 }, line width=1pt]  [-Stealth] (0,-100)--(0, 400);
\draw [color={rgb, 255:red, 0; green, 0; blue, 255 }, line width=1pt]  [-Stealth] (200,-100)--(200, 400);
\draw [color={rgb, 255:red, 0; green, 0; blue,0 }, line width=1pt]   [-Stealth] (0,0)--(200,300) ;
\draw [color={rgb, 255:red, 0; green, 0; blue,0 }, line width=1pt]   [-Stealth] (0,150)--(200,150) ;

\draw  [fill={rgb, 255:red, 0; green, 0; blue, 0 }  ,fill opacity=1 ] (0,0) circle (5) ;
\node at (-20,0) {$O$};
\draw  [fill={rgb, 255:red, 0; green, 0; blue, 0 }  ,fill opacity=1 ] (0,150) circle (5) ;
\draw  [fill={rgb, 255:red, 0; green, 0; blue, 0 }  ,fill opacity=1 ] (200,150) circle (5) ;
\draw [fill={rgb, 255:red, 0; green, 0; blue, 0 }  ,fill opacity=1 ] (200,300) circle (5);
\end{tikzpicture},
we choose the left presentation of Figure \ref{Suitable ribbon presentation (left)}. That is, we assume that two crossings have opposite orientations. We can write this crossing information on chord diagrams by assigning a sign to each chord. For $C_1$ and $C_2$, we choose
\begin{tikzpicture}[x=1pt,y=1pt,yscale=0.2,xscale=0.3,baseline=20pt, line width = 1pt]
\draw [color={rgb, 255:red, 0; green, 0; blue, 255 }, line width=1pt]  [-Stealth] (0,-100)--(0, 400);
\draw [color={rgb, 255:red, 0; green, 0; blue,0 }, line width=1pt]   [Stealth-] (0,300)..controls (100,200)..(0,100) ;
\node at (120,200) {$+$};
\draw [color={rgb, 255:red, 0; green, 0; blue,0 }, line width=1pt]   [-Stealth] (0,0)..controls (100,100)..(0,200) ;
\node at (120,100) {$-$};
\draw  [fill={rgb, 255:red, 0; green, 0; blue, 0 }  ,fill opacity=1 ] (0,0) circle (5) ;
\node at (-20,0) {$O$};
\draw  [fill={rgb, 255:red, 0; green, 0; blue, 0 }  ,fill opacity=1 ] (0,100) circle (5) ;
\node at (-20,100) {};
\draw [fill={rgb, 255:red, 0; green, 0; blue, 0 }  ,fill opacity=1 ] (0,200) circle (5) ;
\node at (-20,200) {};
\draw [fill={rgb, 255:red, 0; green, 0; blue, 0 }  ,fill opacity=1 ] (0,300) circle (5);
\node at (-20,300) {};
\end{tikzpicture}
and \begin{tikzpicture}[x=1pt,y=1pt,yscale=0.2,xscale=0.3,baseline=20pt, line width = 1pt]
\draw [color={rgb, 255:red, 0; green, ; blue, 255 }, line width=1pt]  [-Stealth] (0,-100)--(0, 400);
\draw [color={rgb, 255:red, 0; green, 0; blue, 255 }, line width=1pt]  [-Stealth] (200,-100)--(200, 400);
\draw [color={rgb, 255:red, 0; green, 0; blue,0 }, line width=1pt]   [Stealth-] (0,300)--(200,0);
\node at (160,300) {$+$};
\draw [color={rgb, 255:red, 0; green, 0; blue,0 }, line width=1pt]   [-Stealth] (0,0)--(200,300) ;
\node at (40,300) {$+$};
\draw [color={rgb, 255:red, 0; green, 0; blue,0 }, line width=1pt]   [-Stealth] (0,150)--(200,150) ;
\node at (40,180) {$-$};

\draw  [fill={rgb, 255:red, 0; green, 0; blue, 0 }  ,fill opacity=1 ] (0,0) circle (5) ;
\node at (-20,0) {$O$};
\draw  [fill={rgb, 255:red, 0; green, 0; blue, 0 }  ,fill opacity=1 ] (0,150) circle (5) ;
\draw [fill={rgb, 255:red, 0; green, 0; blue, 0 }  ,fill opacity=1 ] (0,300) circle (5);

\draw  [fill={rgb, 255:red, 0; green, 0; blue, 0 }  ,fill opacity=1 ] (200,0) circle (5) ;
\node at (220,0) {$O$};
\draw  [fill={rgb, 255:red, 0; green, 0; blue, 0 }  ,fill opacity=1 ] (200,150) circle (5) ;
\draw [fill={rgb, 255:red, 0; green, 0; blue, 0 }  ,fill opacity=1 ] (200,300) circle (5);

\end{tikzpicture}.

\begin{figure}[htpb]
\begin{center}
\includegraphics[width = 8cm]{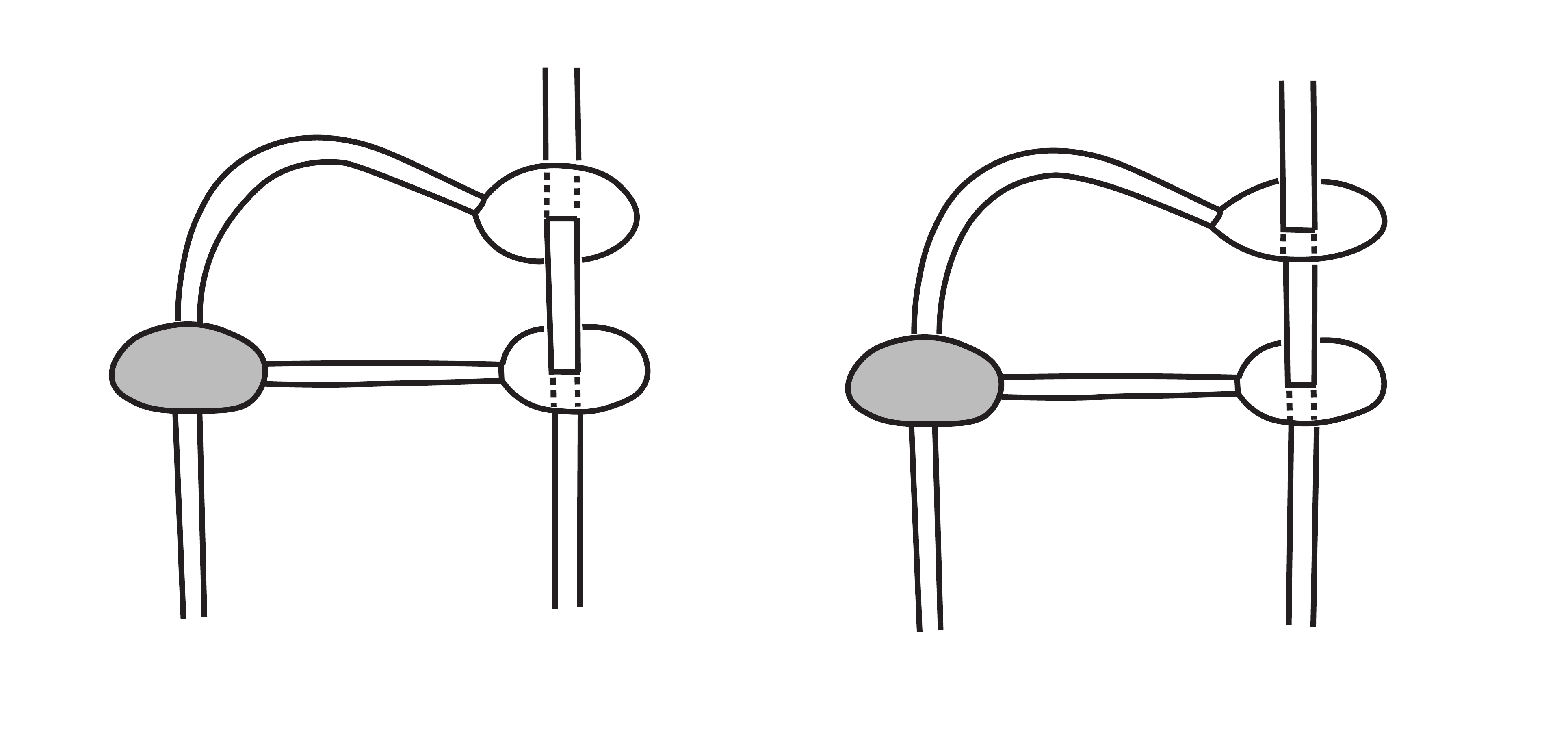}
\caption{Suitable ribbon presentation (left)}
\label{Suitable ribbon presentation (left)}
\end{center}
\end{figure}

\begin{notation}
Let $\phi$ and $\psi$ be the cycles constructed from the ribbon presentations $P_1$ and $P_2$ in subsection \ref {Construction of the cycles sub}. If $n-j=2$, define the cycle
\[
\psi(0,0,0) : S^{j-1} \longrightarrow \overline{\text{Emb}}(\mathbb{R}^j, \mathbb{R}^{j+2})
\]
as the $(0,0, 0)\in S^0\times S^0 \times S^0$ component of $\psi$, that is, the component where no crossing is resolved. Define $\phi(0,0)$ similarly.
\end{notation}

\begin{prop}
\label{unknot component}
The ribbon presentations $P_1$ and $P_2$ are equivalent to the trivial presentation. Thus, when $n-j \geq 3$, both $\phi$ and $\psi$ are in the unknot component. When $n-j = 2$, both $\phi(0,0)$ and $\psi(0,0,0)$ are in the unknot component.
\end{prop}

\begin{proof}
By using moves of ribbon presentations including $S4$, we can resolve the pair of two crossings of opposite orientations as in Figure \ref{Resolving the pair of two crossings}. 
Then, we can show the resulting ribbon presentation is equivalent to the trivial presentation. 
\end{proof}

\begin{figure}[htpb]

\labellist
\small \hair 10pt
\pinlabel $S1$ at 940 1550
\pinlabel $S4$ at 1630 950
\pinlabel $S1$ at 1930 590
\endlabellist

\begin{center}
\includegraphics[width = 11cm]{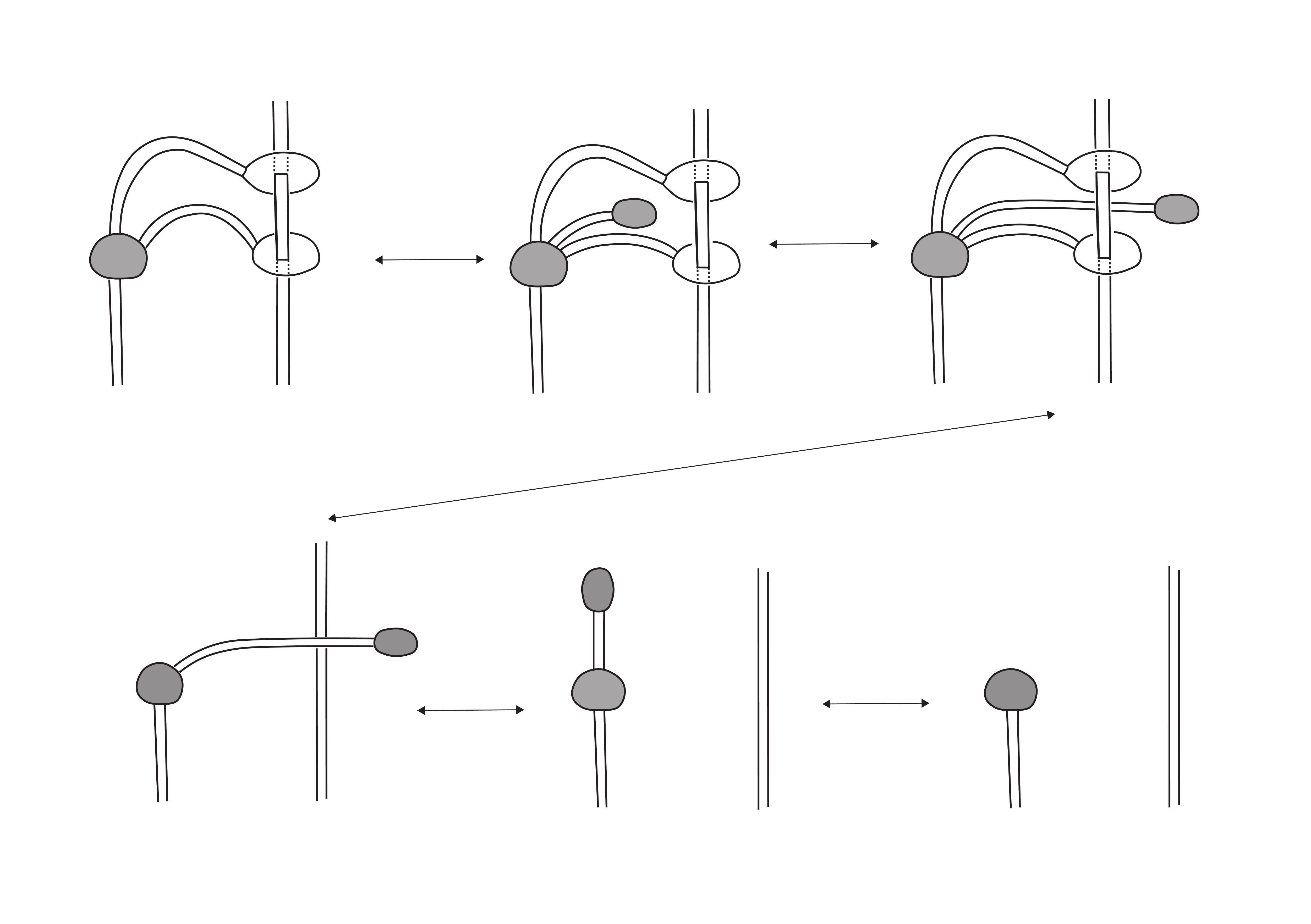}
\caption{Resolving the pair of two crossings}
\label{Resolving the pair of two crossings}
\end{center}
\end{figure}

\subsection{Lifting the cycles to $\overline{\text{Emb}}(\mathbb{R}^j, \mathbb{R}^n)$}
\label{Lifting the cycle}
Finally, we give a parametrization to the cycle $\phi$ and the cycle $\psi$ respectively. We give this parametrization by using a path to the trivially immersed $\mathbb{R}^j$ as in subsection \ref{The path to the trivial immersion}. Moreover, this path gives a lift of $\phi$ and $\psi$ to $\overline{\text{Emb}}(\mathbb{R}^j, \mathbb{R}^n)$.

The path is constructed as follows.  Let $\hat{D}_i$ be the ball of crossings corresponding to $D_{\theta}(\varepsilon)$. In the ribbon presentation, the disk $D_i$ is connected to the node. First, we resolve each crossing, by the move ($m1$) of subsection \ref{The path to the trivial immersion}. Then 
\begin{itemize}
\item ($m3$)  pull back $\hat{D}_i$ to the stem tube. (See Figure \ref{Lift}.)
\end{itemize}
Finally, we deform the immersion to the trivial immersion by the move ($m2$) of subsection \ref{The path to the trivial immersion}. 
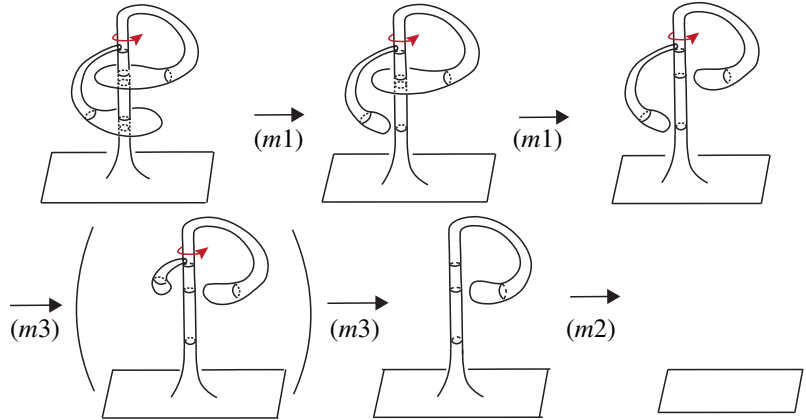
\begin{figure}[htpb]

\labellist
\small \hair 10pt
\pinlabel $(m1)$ at 400 400
\pinlabel $(m1)$ at 740 400

\pinlabel $(m3)$ at 80 150
\pinlabel $(m3)$ at 500 150

\pinlabel $(m2)$ at 810 150

\endlabellist

\centering
\includegraphics[width = 12cm]{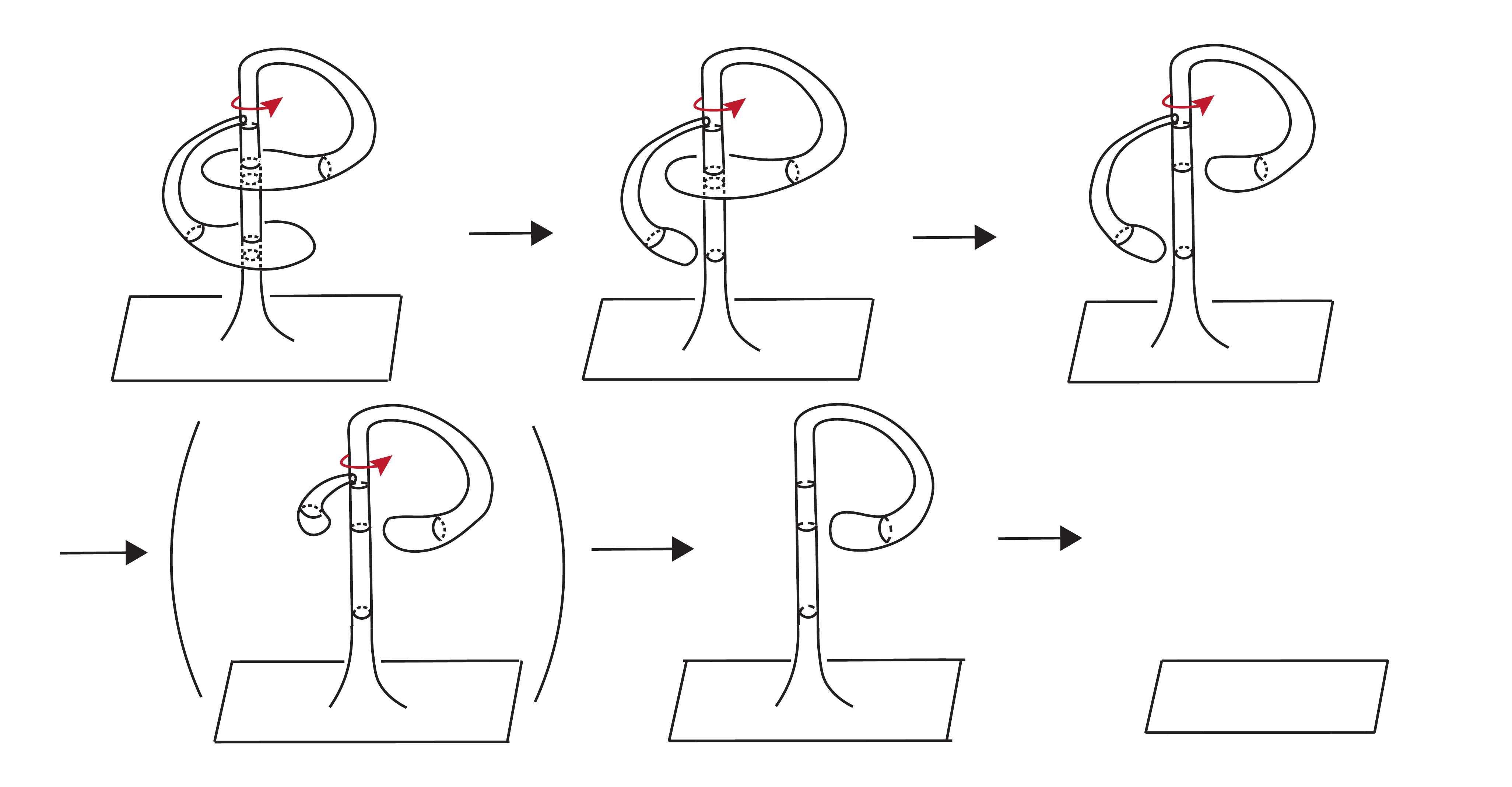}
\caption{Lift of $\phi$ to $\overline{\text{Emb}}(\mathbb{R}^j, \mathbb{R}^n)$}.
\label{Lift}
\centering
\end{figure}

\section{Some lemmas for computing cocycle-cycle pairing}
\label{Some lemmas for computing cocycle-cycle pairing}
In this section, we state preliminary lemmas which are used to calculate pairings between cocycles and cycles of the space of long embeddings, as pairings between graphs and diagrams (see Theorem \ref{Counting formula}).
Readers who are familiar with configuration space integrals may skip until Lemma \ref{Decomposing lemma}. 
Graphs in this section are admissible (see Definition \ref{admissible graphs}) graphs of defect 0.

\subsection{Preliminary lemmas}
\begin{lemma}[Symmetry lemma]
\label{Symmetry lemma}
Let a labeled BCR graph $\Gamma$ have an even symmetry $\iota$ (a change of labels of even sign that produces the same labeled graph). Consider the involution on $C_{s, t}$
\begin{equation*}
\sigma :  (x_1, x_2, \dots, x_s, x_{s+1},\dots, x_{s+t}) \mapsto (x_{\iota^{-1}(1)}, x_{\iota^{-1}(2)}, x_{\iota^{-1}(s)}, x_{\iota^{-1}(s+1)}, \dots, x_{\iota^{-1}(s+t)}).
\end{equation*}
Let the two disjoint submanifold $X$ and $Y$ of  $C_{s, t}$ satisfy $\sigma(X) = Y$.
Then $\int_X I(\Gamma)= \int_Y I(\Gamma)$ holds.
\end{lemma}
\begin{proof}
\begin{align*}
&\int_Y I(\Gamma)\\
=& \int_{\sigma(X)} (\sigma^{-1})^{\ast}I(\Gamma),\\
\end{align*}
since the changes of signs are canceled.
\end{proof}

\begin{lemma}[Localizing lemma {\cite[Lemma 4.7]{SW}}]
\label{Localizing lemma}
\label{Localizing lemma}
Let $\Gamma$ be a BCR graph of order $k$, $\#B(\Gamma)=s$ and $\#W(\Gamma)=t$.
Let $\psi$ be a generalized ribbon cycle with crossings $p_1, \dots, p_k.$
Recall that each crossing consists of the image of a ball $D_{p_i}$ (or $D_{p_i}(\theta^{\prime}_{\tau^{-1}(2i-1)})$) and the image of an annulus $I\times S_{p_i}$. Let $C_0$ be the subspace of $C_{s,t}(\psi)$ which consists of configurations $(x_1, x_2, \dots, x_s, x_{s+1}, \dots, x_{s+t})$ satisfying the following.
\begin{itemize}
\item There exists $i$ such that either $D_{p_i}$ or $I\times S_{p_i}$ has no black vertex.
\end{itemize}
Then $\int_{C_0} I(\Gamma) = 0$ holds.
\end{lemma}
\begin{proof}
Let $p_i$ be such a crossing. Then the contribution of the integral vanishes by dimensional reason (or is canceled) with respect to the parameter $y_i$ of the crossing $p_i$.
\end{proof}

Set 
\begin{equation*}
X_ \alpha =  
\begin{cases}
D_ {p_i }\ \text{or}\ D_ {p_i }(\theta^{\prime}_{\tau^{-1}(2i-1)}) & \text{if}\quad \alpha = 2i-1 \\
I \times S_ {p_i} & \text{if}\quad  \alpha = 2i.\\
\end{cases}
\end{equation*}

By Localizing lemma, we only consider the graph with $s= 2k$ and $t=0$, and configurations that there are some black vertex $x_{\sigma^{-1}(\alpha)}$ in $X_{\alpha}$. We write the set of these configurations as $C_{\sigma}$, $\sigma \in \mathfrak{S}_{2\alpha}$. If $\sigma \neq \tau$, $C_{\sigma}$ and $C_{\tau}$ have no intersection.

\begin{lemma}[Pairing lemma {\cite[Lemma 4.9]{SW}}]
\label{pairing lemma}
If $x_{\sigma^{-1}(2i-1)} \in D_ {p_i }$  and  $x_{\sigma^{-1}(2i)} \in I \times S_ {p_i} $ are not connected by dashed edges, $\int_{C_{\sigma}} I(\Gamma) = 0$ holds. 
\end{lemma}

\begin{proof}
This can be shown by a similar dimensional reason to the proof of Lemma~\ref{Localizing lemma}.
\end{proof}

\begin{lemma}[Same system lemma {\cite[Lemma 4.8]{SW}}]
\label{Same system lemma}
Let $\Gamma$ be a BCR graph of order $k$ (and defect $0$), $\#B(\Gamma) = 2k$ and $\#W(\Gamma) = 0$.
Let $\psi$ be a generalized ribbon cycle with crossings $p_1, \dots, p_k.$
If $X_{\alpha}$ and $X_{\beta}$ are in different planetary systems, and $x_{\sigma^{-1}(\alpha)}$ and $x_{\sigma^{-1}(\beta)}$ are connected by solid edges (see Figure \ref{An edge connecting different systems}),
$\int_{C_{\sigma}} I(\Gamma) = 0$ holds.

\end{lemma}

\begin{figure}
\begin{center}
\tikzset{every picture/.style={line width=1pt, xscale=0.5pt, yscale = 0.5pt}}  

\begin{tikzpicture}

\draw (2,-1) rectangle (24, 7);

\draw (6,3) circle(3) [line width = 2pt] [color= {rgb, 255:red, 0; green, 0; blue, 0 }, fill opacity =1.0 ];

\draw (6,3) circle(2) [line width = 2pt] [color= {rgb, 255:red, 0; green, 0; blue, 0 }, fill opacity =1.0];
\node at (6, 1.3) {\scriptsize $X_{\alpha}$};

\draw (6,3) circle(1) [line width = 2pt] ;

\begin{scope}[xshift = 14cm]

\draw (6,3) circle(3) [line width = 2pt] [color= {rgb, 255:red,0; green, 0; blue, 0 }, fill opacity =1.0 ];

\draw (6,3) circle(2) [line width = 2pt] [color= {rgb, 255:red, 0; green, 0; blue, 0}, fill opacity =1.0];

\draw (6,3) circle(1) [line width = 2pt] ;
\node at (6, 2.3) {\scriptsize $X_{\beta}$};

\end{scope}


\draw (8 ,3) circle (0.1) [fill = {rgb, 255:red, 0; green, 0; blue, 0 }, fill opacity =1.0];
\draw (19 ,3) circle (0.1) [fill = {rgb, 255:red, 0; green, 0; blue, 0 }, fill opacity =1.0];

\draw (8,3) -- (19,3);

\end{tikzpicture}
\caption{An edge connecting different systems}
\label{An edge connecting different systems}
\end{center}

\end{figure}
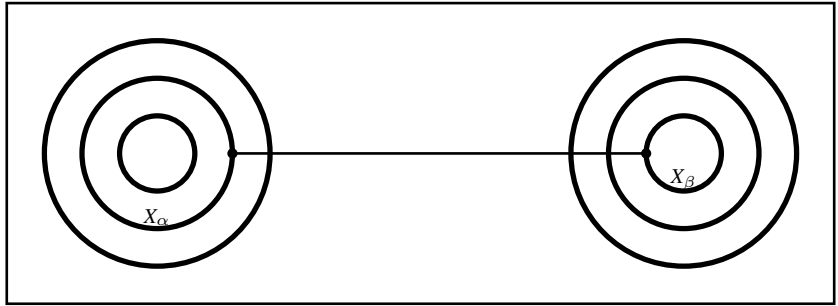

\begin{proof}
This is because the direction $\frac{x_{\sigma^{-1}(\beta)} - x_{\sigma^{-1}(\alpha)}}
{||x_{\sigma^{-1}(\beta)} - x_{\sigma^{-1}(\alpha)}||} \in S^{j-1}$ varies in a small neighborhood of a point in $S^{j-1}$. 
\end{proof}

 \begin{rem}
 More precisely, $\int_{C_{\sigma}} I(\Gamma) = 0$ above means $\int_{C_{\sigma}} I(\Gamma) $ approaches $0$ as the radii of the two planetary systems approach $0$. 
 \end{rem}

\begin{lemma}[Ingoing lemma]
\label{Ingoing lemma}
Let $\Gamma$ be a BCR graph of order $k$, $\#B(\Gamma) = 2k$ and $\#W(\Gamma) = 0$.
Let $\psi$ be a generalized ribbon cycle with crossings (pairings) $p_1, \dots, p_k.$
Let $X_{\alpha}$ be a planetary orbit in a planetary system.
Assume that all solid edges adjacent to $x_{\sigma^{-1}(\alpha)}$ are outgoing. That is,  if $x_{\sigma^{-1}(\alpha)}$ and  $x_{\sigma^{-1}(\beta)}$ are connected by solid edges,  $X_{\beta}$ is outside $X_{\alpha}$ (see Figure \ref{Outgoing edges}).
Then $\int_{C_{\sigma}} I(\Gamma) = 0$. 
\end{lemma}

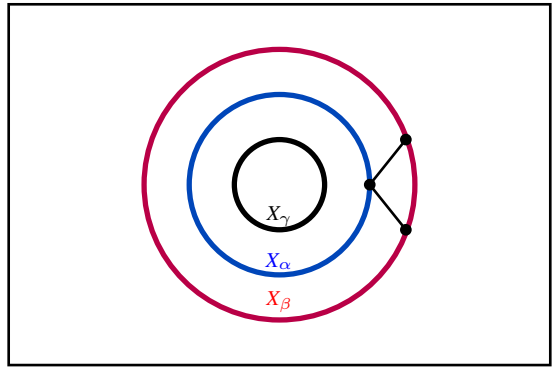
\begin{figure}[htpb]
\begin{center}
\tikzset{every picture/.style={line width=1pt, xscale=0.6pt, yscale = 0.6pt}}  

\begin{tikzpicture}

\draw (0,-1) rectangle (12, 7);

\draw (6,3) circle(3) [line width = 2pt] [color= {rgb, 255:red, 185; green, 0; blue, 70 }, fill opacity =1.0 ];
\node  [red] at (6, 0.4) {\scriptsize $X_{\beta}$};

\draw (6,3) circle(2) [line width = 2pt] [color= {rgb, 255:red, 0; green, 70; blue, 185 }, fill opacity =1.0];
\node [blue] at (6, 1.3) {\scriptsize $X_{\alpha}$};

\draw (6,3) circle(1) [line width = 2pt] ;
\node at (6, 2.3) {\scriptsize $X_{\gamma}$};


\draw (8.8,2) circle (0.1) [fill = {rgb, 255:red, 0; green, 0; blue, 0 }, fill opacity =1.0];
\draw (8.8 ,4) circle (0.1) [fill = {rgb, 255:red, 0; green, 0; blue, 0 }, fill opacity =1.0];
\draw (8 ,3) circle (0.1) [fill = {rgb, 255:red, 0; green, 0; blue, 0 }, fill opacity =1.0];

\draw (8.8, 2) -- (8,3) --(8.8, 4);

\end{tikzpicture}
\caption{Outgoing edges}
\label{Outgoing edges}
\end{center}

\end{figure}

\begin{proof}
Since the integrand form does not depend on the parameter $x_{\sigma^{-1}(\alpha)}$, it factors through $C_\sigma \rightarrow C_\sigma/ X_{\alpha} $.
\end{proof}

\begin{lemma}[Decomposing lemma]
\label{Decomposing lemma}
\begin{itemize}
\item
Let $e = \{\sigma^{-1}(\alpha), \sigma^{-1}(\beta)\} $ be a solid edge. Let $X_{\alpha}$ and $X_{\beta}$ be in the same planetary system and $X_{\beta}$ be outside.
Then if $X_{\beta} = I \times S_p$, $\phi_{e}^{\ast} \omega_{S^{j-1}}$ depends only on the variable $\theta_{\sigma^{-1}(\beta)} \in S_p$ of $x_{\sigma^{-1}(\beta)} = (r_{\sigma^{-1}(\beta)}, \theta_{\sigma^{-1}(\beta)}) \in X_{\beta}$. If $X_{\beta} = D_p(\theta^{\prime})$, it depends only on $\theta^{\prime}$.
\item
Let $e = \{\sigma^{-1}(\alpha), \sigma^{-1}(\beta)\} $ be a dashed edge. Let $X_{\alpha} = D_p$ (or $D_p(\theta^{\prime}$)) and $X_{\beta} = I \times S_p$. Then $\phi_{e}^{\ast} \omega_{S^{n-1}}$ depends on only on the variable $x_{\sigma^{-1}(\alpha)} \in X_{\alpha}$ and on the first factor $r_{\sigma^{-1}(\beta)} \in I $ of $x_{\sigma^{-1}(\beta)} = (r_{\sigma^{-1}(\beta)}, \theta_{\sigma^{-1}(\beta)}) \in X_{\beta}.$
\end{itemize}
\end{lemma}

\begin{proof}
This follows from the way we choose the scaling of planetary orbits.
\end{proof}

\begin{definition}[Graph-chord pairing]
\label{Graph-chord pairing}
Let $C$ be a chord diagram on directed lines of order $k$.  Let $\Gamma$ be a labeled BCR graph of order $k$.
Then the pairing $<\Gamma, C>$ is defined by counting as follows. First, we only count when
\begin{itemize}
\item [(I)] $\#B(\Gamma) =  2k (= \# V(C))$.
\end{itemize}
Let $\#B(\Gamma) =  \#V(C) =2k$.  We count permutations $\sigma : \#B(\Gamma) \rightarrow  \#V(C)$ which satisfies (I) and all the followings.
\begin{itemize}
\item [(II)] $ \sigma$ induces the map  $ \overline{\sigma} : E(\Gamma) \rightarrow E(C)$.
\item [(III)] If two black vertices $v_1$ and $v_2$ are in the same solid component, $\sigma(v_1)$ and $\sigma(v_2)$ are on the same directed line.
\item [(IV)] If a vertex $w$ of $V(C)$ is not on the $x$-axis, $w$ has an ingoing solid edge: there exists a vertex $w^{\prime}$ lower on the same directed line such that $\sigma^{-1}(w)$ and $\sigma^{-1}(w^{\prime})$ are connected by some solid edge of $\Gamma$.
\end{itemize}
The sign of the counting is $+1$ if and only if the orientation of $\Gamma$ coincides with the induced orientation determined by the induced label from $\sigma$. (See Definition~\ref{Induced color on BCR diagrams}.)
\end{definition}

\begin{definition}[Induced label on BCR diagrams]
\label{Induced color on BCR diagrams}
Let $\sigma: B(\Gamma) \rightarrow V(C) $ satisfy all the conditions (I)(II)(III)(IV) of Definition \ref{Graph-chord pairing}.
Then $\sigma$ defines an induced label on the underlying graph $\overline{\Gamma}$ in a natural way:
we order black vertices using $\sigma$ and the ordering of $V(C)$. (See Definition \ref{Induced ordering}.)
We orient solid edges so that the orientations are compatible with directed lines.
We orient dashed edges so that the orientations are compatible with those of chords.
We order edges in the following order.
\begin{itemize}
\item [(1)] Ingoing solid edge (if it exists) from the initial vertex of the $i$th chord ($i = 1, \dots, k$). There are $(g-1)$ solid edges ordered by this first step.
\item [(2)] The $\overline{\sigma}^{-1}(i)$th dashed edge and the ingoing solid edge from the target point of the $i$th chord ($i = 1, \dots, k$).
\end{itemize}

\begin{rem}
We do not use the induced ordering of edges when $(n,j) = $(odd, odd).
\end{rem}

\begin{theorem}(Refer to \cite[Lemma 4.12]{SW})
\label{key lemma}
Let $C$ be a chord diagram of order $k$ associated with a generalized ribbon cycle $\psi$. Assume $r(c)$ chords have the negative sign.
Let $\Gamma$ be a BCR graph of order $k$. Then the value of the integral $I(\Gamma)$ on $\psi$ is equal to the pairing $(-1)^{r(c)} <\Gamma, C>$.
\end{theorem}

\begin{proof}
By (I) Localizing lemma \ref{Localizing lemma} (II) Pairing lemma \ref{pairing lemma}, (III) Same system lemma \ref{Same system lemma}, (IV) Ingoing lemma \ref{Ingoing lemma},
the integral survives only on $C_{\sigma}$ for $\sigma$ which satisfies (I) (II) (III) (IV) of Definition \ref{Graph-chord pairing}.
By Decomposing lemma \ref{Decomposing lemma} the integral on $C_{\sigma}$ can be written as 
\begin{equation*}
\pm (-1)^{r(c)} \left(\int_{\theta^{\prime}_j \in S}\bigwedge_{j=1}^{g-1} \eta_{{\tau(j)}} \right)\left(\bigwedge_i \int_{x_{\sigma^{-1}(2i-1)} \in D_{p_i}}\int_{r_{\sigma^{-1}(2i)} \in I}\omega_{p_i} \int_{\theta_{\sigma^{-1}(2i)} \in S_{p_i}} \eta_{p_i}\right).
\end{equation*}
Here $\omega_{p_i}$ is the form associated with the dashed edge $\{x_{\sigma^{-1}(2i-1)}, x_{\sigma^{-1}(2i)}\}$. $\eta_{p_i}$ is the form associated with the solid edge that goes down from $x_{\sigma^{-1}(2i)}$. $\eta_{{\tau(j)}}$ is the form associated with the solid edge that goes down from $x_{\tau(j)}$. The sign at the head is $+$ if and only if the orientation of the configuration space $C_\sigma$  coincides with the orientation induced by the coordinate $(x_{\sigma^{-1}(1)}, x_{\sigma^{-1}(2)}, \dots,  x_{\sigma^{-1}(2k)})$, and $\omega(\Gamma) $ is exactly $\bigwedge_{j=1}^{g-1} \eta_{\tau(j)} \wedge \bigwedge (\omega_{p_i}\eta_{p_i})$.

Recall that the image of the annulus $I \times S_{p_i}$ is perturbed to the direction $y_i \in S^{n-j-2}$.
Then, in the part  $ \int_{x \in D}  \int_{r  \in I} \theta$, the linking number of $D$ and $\cup_{y_i \in S^{n-j-2}} I(y_i)$ is computed. Since this link is the Hopf link $S^{j-1} \cup \Sigma S^{n-j-2} \hookrightarrow \mathbb{R}^n$, the linking number is $+1$ or $-1$ depending on the sign of the chord.
\end{proof}

\begin{notation}
Let $C = (\{t_i\}_{i = 1, \dots, s}, \{p_i\}_{i=1, \dots, k})$ be a chord diagram on directed lines. 
Let $G(C)$ be the set of BCR graphs (without a label) mapped on $C$ satisfying all the conditions of Definition \ref{Graph-chord pairing}. 
\end{notation}

\begin{lemma}
$G(C)$ consists of $\sum_{i=1}^s  2^{t_i -1}$ graphs. 
\end{lemma}

\begin{proof}
Since dashed edges must be mapped on chords, the remaining choices are solid edges.
The top vertex $v$ on each directed line must be one of the ends of some solid component. The solid component first goes down from $v$ to $O$ through some of the vertices on the directed line. After reaching $O$, the solid component goes up though the remaining vertices. We can choose which of $(t_i -1)$ vertices between $v$ and $O$ the solid component passes when it goes down. 
\end{proof}

\subsection{The counting formula of configuration space integrals}

\begin{theorem}[Counting formula]
\label{Counting formula}
Let $H = \sum w(\Gamma_i) \Gamma_i$ be a graph cocycle that consists of $g$-loop graphs of order $k$. 
Let $C$ be a chord diagram on directed lines of order $k$, with $r(C)$ chords of negative sign. Let $\psi$ be a generalized ribbon cycle constructed from $C$.
Then 
\begin{equation*}
 I(H)(\psi) = (-1)^{r(C)} \sum_{\overline{\Gamma} \in G(C)}  w(\Gamma) \#\text{Aut}(\overline{\Gamma})  s(\Gamma, \overline{\Gamma}).
\end{equation*}
Here $\Gamma$ is a labeled graph in $H$ such that the underlying graph of $\Gamma$ is $\overline{\Gamma}$.
The sign $s(\Gamma, \overline{\Gamma}) \in \{+1, -1\}$ is $1$ if and only if the orientation of $\Gamma$ coincides with that of $\overline{\Gamma}$ with the induced label.
\end{theorem}

\begin{proof}
This follows from Theorem \ref{key lemma} and Symmetry lemma \ref{Symmetry lemma}.
\end{proof}

\begin{example}
\label{Example of counting formula}
For the chord diagram $C_2= \chorddiagramb{}$, $G(C_2)$ consists of the four graphs drawn in red in Figure \ref{graphchordpairing2}.
The first BCR graph (a) is \graphd{}. The other three graphs (b) (c) (d) are \graphe{}. The induced label is written in black. The sign $s(\Gamma, \overline{\Gamma})$ is
(a) $-1$ (b) $+1$ (c) $+1$ (d) $-1$ respectively.

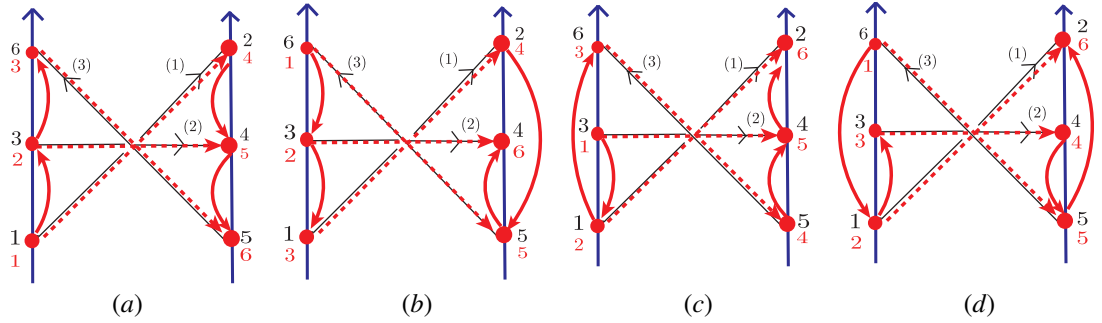
\begin{figure}[htpb]

\labellist
\small \hair 10pt

\pinlabel $(a)$ at 330 230
\pinlabel $(b)$ at 1000 230
\pinlabel $(c)$ at 1650 230
\pinlabel $(d)$ at 2300 230

\endlabellist

\begin{center}
\includegraphics[width = 15cm]{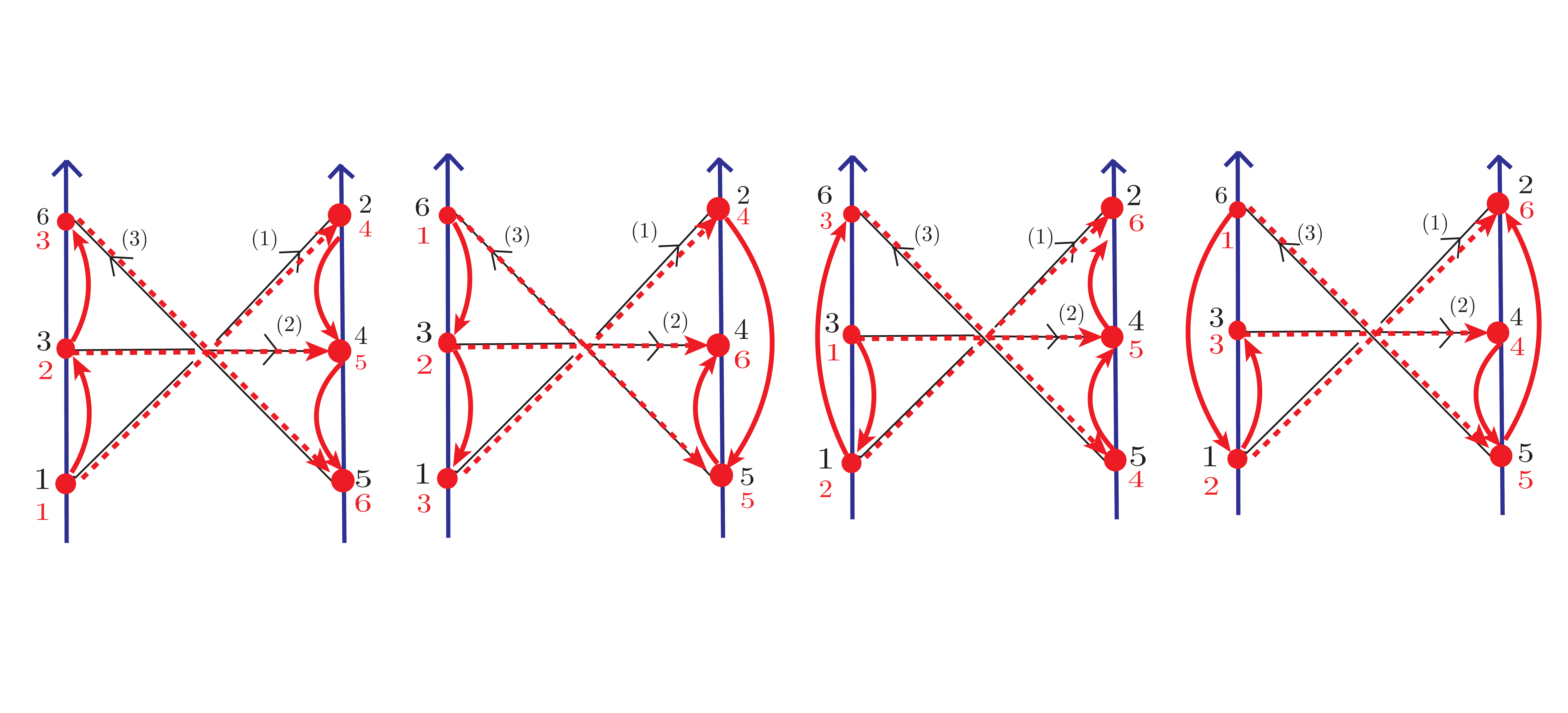}
\caption{BCR graphs of $G(C_2)$}
\label{graphchordpairing2}
\end{center}
\end{figure}
\end{example}
\end{definition}

\subsection{Cancellation of correction terms}

Finally, we see lemmas for canceling the contribution of the collection term $\overline{c}(H)$.

\begin{lemma}\cite [Lemma 4.13] {SW}
\label{psi0}
Let $\overline{\psi}: X \rightarrow \overline{\text{Emb}}(\mathbb{R}^j, \mathbb{R}^n)$ be a generalized ribbon cycle of order $k$, and let $p_1, p_2, \dots, p_k$ be its crossings.
Then there exists a cycle $\overline{\psi}_0: X \rightarrow \overline{\text{Emb}}(\mathbb{R}^j, \mathbb{R}^n)$ satisfying the following.
\begin{itemize}
\item $\overline{\psi}_0$ is obtained from $\overline{\psi}$ by resolving all crossings.
\item The two differentials $D\overline{\psi}_0, D\overline{\psi}$ : $[0,1] \times C_1(\mathbb{R}^j )\times X \rightarrow \text{Inj}(\mathbb{R}^j, \mathbb{R}^n)$ is homotopic through an isotopy of $C_1(\mathbb{R}^j$) that has a support in small neighborhoods of balls of crossings.
\end{itemize}
\end{lemma}

\begin{lemma}
\label{Cancellation of the correction term}
The cycle $\overline{\psi}_0$ satisfies the following.
\begin{itemize}
\item [(1)] $\overline{c}(H)(\overline{\psi}- \overline{\psi}_0)= 0$
\item [(2)] $I(H)(r_{\ast} \overline{\psi}_0) = 0$, where $r: \overline{\text{Emb}}(\mathbb{R}^j, \mathbb{R}^n) \rightarrow \text{Emb}(\mathbb{R}^j, \mathbb{R}^n)$ is the natural projection.
\item [(3)] $\overline{\psi}_0$ is a null-homotopic cycle.
\end{itemize}
\end{lemma}
\begin{proof}
(1) Recall that $\overline{c}(H)$ depends only on the differential $\overline{D}: [0,1] \times C_1(\mathbb{R}^j) \times \overline{\text{Emb}}(\mathbb{R}^j, \mathbb{R}^n) \longrightarrow \text{Inj}(\mathbb{R}^j, \mathbb{R}^n)$. $\overline{c}(H)$ is defined as the fiber integral of the pullback of some closed form $\omega$ on $\text{Inj}(\mathbb{R}^j, \mathbb{R}^n)$.  By the second property of Proposition \ref{psi0}, the two compositions $D \overline{\psi}_0 =  \overline{D} \circ (id_{[0,1]} \times id_{C_1(\mathbb{R}^j)}\times \overline{\psi}_0)$ and $ D \overline{\psi} = \overline{D} \circ (id_{[0,1]} \times id_{C_1(\mathbb{R}^j)} \times \overline{\psi})$ differ only in small neighborhoods of balls of crossings. Moreover, observe that neither $D \overline{\psi}_0$ nor $D\overline{\psi}$ depend on the parameter space $X$, in these small neighborhoods of balls. Hence by a dimensional reason, the integral vanishes on these neighborhoods. So we have (1).
(2) and (3) are due to the first property of Lemma \ref{psi0}.
\end{proof}

\begin{cor}
\label{Cor of cancellation of the correction term}
$(r^{\ast}I(H)+\overline{c}(H))(\overline{\psi}) =  I(H)(\psi)$. 
\end{cor}

\begin{proof}
By  Lemma  \ref{Cancellation of the correction term} (1)(2)(3),  we have
\begin{align*}
&(r^{\ast}I(H)+\overline{c}(H))(\overline{\psi}) \\
\underset{(3)}{=}  &(r^{\ast}I(H)+\overline{c}(H))(\overline{\psi}- \overline{\psi}_0) \\
\underset{(1)}{=}  &(r^{\ast}I(H))(\overline{\psi}- \overline{\psi}_0) \\
 \underset{(2)}{=}& I(H)(\psi).
 \end{align*}
 \end{proof}

\begin{rem}
In particular, we have $\overline{c}(H)(\overline{\psi}) = 0$. We can show this directly without using $\overline{\psi}_0$, if we make the argument of the proof of (1) more precise: the parameter that runs on $S^{n-j-2}$ only affects a small neighborhood of the annulus of exactly one crossing. So if the cycle $\overline{\psi}$ has at least two crossings, the integral $\overline{c}(H)(\overline{\psi})$ vanishes by a dimensional reason.
\end{rem}


\section{Proof of Main Result}
\label{Proof of Main Result}
Finally, we perform the pairing between the cocycle given in Section \ref{Construction of  the cocycles} and the cycle given in Section \ref{Construction of the cycles}, and prove Main Result \ref{main result 2} and Corollary \ref{main corollary}. The following completes the proof of Theorem \ref{main result 2}.

\begin{theorem}
\label{pairing 2}
\begin{align*}
\overline{z}^3_2(\overline{\psi})& =1.\\
\end{align*}
\end{theorem}

\begin{proof}
By  Corollary \ref{Cor of cancellation of the correction term},  we have
\begin{align*}
\overline{z}^3_2(\overline{\psi}) =  I(H)(\psi).
 \end{align*}
Observe that the underlining graph of \graphd{} has 4 automorphisms, and the underlining graph of \graphe{} has 2 automorphisms. 
In the graph cocycle $H$, $w(\graphd{})= -\frac{1}{4}$ and $w(\graphe{})= -1$.
Hence by Counting formula \ref{Counting formula} and Example \ref{Example of counting formula},
\begin{align*}
I(H)(\psi) &= -\sum_{\overline{\Gamma} \in G(C_2)} w(\Gamma) \# \text{Aut}(\overline{\Gamma})  s(\Gamma, \overline{\Gamma})
= -(-\frac{1}{4} \times 4 \times (-1)+ (-1)\times 2 \times (+1+1-1)) = 1.
\end{align*}
\end{proof}

\begin{rem}
We can give a simpler proof of Theorem \ref{pairing 2} using a graph homology \cite{SW}. We give a sketch of the proof here.
Let $\mathcal{A}(k, g)$ be the graph homology generated by $g$-loop BCR graphs of order~$k$.
Let $\overline{w}: \mathcal{A}(k, g) \rightarrow \mathbb{R}$ be a weight system on it. Then
\[
H = \sum_{\substack{\overline{\Gamma} \\ k(\overline{\Gamma}) = k, g(\overline{\Gamma}) = g}} \frac{\overline{w}(\Gamma)}  {\# \text{Aut}(\overline{\Gamma})}\Gamma
\]
is a graph cocycle. Set $w(\Gamma) = \frac{\overline{w}(\Gamma)}  {\# \text{Aut}(\overline{\Gamma)}}$.
Then the counting formula \ref{Counting formula} is
\begin{equation*}
 I(H)(\psi) = (-1)^{r(C)} \sum_{\overline{\Gamma} \in G(C)}  \overline{w}(\Gamma)  s(\Gamma, \overline{\Gamma}).
\end{equation*}
Consider the weight system $\overline{w}: \mathcal{A}(3, 2) \rightarrow \mathbb{R}$ such that $\overline{w}(\graphg) = \pm 1$.
Observe that the linear combination of four graphs $\sum_{\overline{\Gamma} \in G(C_2)}  s(\Gamma, \overline{\Gamma}) \Gamma $
gives the same class as the hairy graph $\pm \graphg \in \mathcal{A}(3, 2)$. (To see this, use STU relations.) Then we have
\[
 I(H)(\psi) = (-1)^{r(C_2)} \ \overline{w} \left(\sum_{\overline{\Gamma} \in G(C_2)} s(\Gamma, \overline{\Gamma}) \Gamma \right) = \pm \overline{w}(\graphg) \neq 0 .
\]
\end{rem}

Finally, we prove Corollary \ref{main corollary}, the case $n-j-2=0$. Let 
\[
\overline{\psi}(0,0,0) : S^{j-1} \longrightarrow \overline{\text{Emb}}(\mathbb{R}^j, \mathbb{R}^{j+2})
\]
be the $(0,0, 0)\in S^0\times S^0 \times S^0$ component of $\overline{\psi}$, that is, the component where no crossing is resolved. By Proposition \ref{unknot component}, $\overline{\psi}(0,0,0)$ is in the unknot component (the path component of the trivial path of the trivial immersion).

\begin{cor}
When $n-j = 2$, $\overline{\psi}$ and $\overline{\psi}(0,0,0)$ give the same homology class. Hence, when $j$ is odd, $\overline{\psi}$ gives an non-trivial element of
\[
\pi_{\ast}(\overline{\text{Emb}}(\mathbb{R}^j, \mathbb{R}^{j+2})_{\iota}) \otimes \mathbb{Q}.
\]
\end{cor}

\begin{proof}
Since $S_0 = \{x_3=0, x_3=2\}$, we write
\[
\overline{\psi} = \sum_{\epsilon_i \in \{0,2\}} \overline{\psi}(\varepsilon_1,\varepsilon_2,\varepsilon_3)
\]
Observe that if one of $\varepsilon_i$ is $2$, $\overline{\psi}(\varepsilon_1,\varepsilon_2,\varepsilon_3)$ is a degenerate cycle. 
\end{proof}


\newpage
\begin{thebibliography}{9}
\bibitem [Arn] {Arn}V. Arnol'd, \textit{The cohomology ring of the colored braid group}. Mathematical Notes of the Academy of Sciences of the USSR 5, (1969), 138--140.
\bibitem [AS] {AS} S. Axelrod, I. M. Singer, \textit{Chern-Simons perturbation theory. II.} J. Differential Geom.  39  (1994),  no. 1, 173--213.
\bibitem [AT1] {AT1} G. Arone, V. Turchin, \textit{On the rational homology of high-dimensional analogues of spaces of long knots}. Geom. Topol.  18  (2014),  no. 3, 1261--1322.
\bibitem [AT2] {AT2} G. Arone, V. Turchin, \textit{Graph-complexes computing the rational homotopy of high dimensional analogues of spaces of long knots}. Ann. Inst. Fourier (Grenoble)  65  (2015),  no. 1, 1--62.
\bibitem [Bar] {Bar} D. Bar-Natan, \textit{On the Vassiliev knot invariants}. Topology  34  (1995),  no. 2, 423--472.
\bibitem [BG] {BG} R. Budney, D.Gabai, \textit{Knotted 3-balls in $S^4$.} arXiv.1912.09029.
\bibitem [Bud] {Bud} R. Budney, \textit{A family of embedding spaces}. Geometry and Topology Monographs 13 (2008) 41-83.
\bibitem [Bot] {Bot} R. Bott, \textit{Configuration spaces and imbedding invariants}. Turkish J. Math.  20  (1996),  no. 1, 1--17.
\bibitem [BT] {BT} R. Bott, C. Taubes, \textit{On the self-linking of knots}. Topology and physics. J. Math. Phys. 35 (1994), no. 10, 5247--5287. 
\bibitem [CCL] {CCL} A. Cattaneo, P. Cotta-Ramusino, R. Longoni, \textit{Configuration spaces and Vassiliev classes in any dimension}. Algebr. Geom. Topol. 2 (2002), 949--1000.
\bibitem[CCTW]{CCTW} J. Conant, J. Costello, V. Turchin, P. Weed, \textit{Two-loop part of the rational homotopy of spaces of long embeddings}. J. Knot Theory Ramifications  23,(2014),  no. 4, 1450018, 23 pp.
\bibitem [CR] {CR}  A. S. Cattaneo, C. A. Rossi, \textit{Wilson surfaces and higher dimensional knot invariants}. Comm. Math. Phys.  256  (2005),  no. 3, 513--537.
\bibitem [FTW] {FTW} B. Fresse, V. Turchin, T.Willwacher, \textit{The rational homotopy of mapping space of $E_n$ operads}. arXiv.1703.06123.
\bibitem [GW] {GW} T. Goodwillie, M. Weiss, \textit{Embeddings from the point of view of immersion theory. II.} Geom. Topol.  3  (1999), 103--118.
\bibitem [GKW] {GKW} T. Goodwillie, J. Klein, M. Weiss,  \textit{Spaces of smooth embeddings, disjunction and surgery.} Surveys on surgery theory, Vol. 2, 221--284, Ann. of Math. Stud., 149, Princeton Univ. Press, Princeton, NJ,  2001.
\bibitem [HKS] {HKS} K. Habiro, T, Kanenobu, A. Shima,  \textit{Finite type invariants of ribbon 2-knots.} Low-dimensional topology (Funchal, 1998), 187--196, Contemp. Math., 233, Amer. Math. Soc., Providence, RI,  1999.
\bibitem [HS] {HS} K. Habiro, A. Shima, \textit{Finite type invariants of ribbon 2-knots. II.} Topology Appl. 111 (2001), no.3, 265--287.
\bibitem [Kon] {Kon} M. Kontsevich, \textit{Feynman diagrams and low-dimensional topology.} First European Congress of Mathematics, Vol. II (Paris, 1992), 97--121, Progr. Math., 120, Birkh\"{a}user, Basel,  1994. 
\bibitem [KT] {KT}  G. Kuperberg, D. Thurston, \textit{Perturbative 3-manifold invariants and cut-and-paste topology.} math.GT.9912167.
\bibitem [Les 1] {Les 1} C. Lescop, \textit{On the Kontsevich-Kuperberg-Thurston construction of a configuration-space invariant for rational homology 3-spheres.} math.GT.0411088.
\bibitem [Les 2] {Les 2} C. Lescop, \textit{Invariants of links and 3-manifolds from graph configurations.} arXiv.2001.09929.
\bibitem [Let] {Let} D. Leturcq, \textit{Generalized Bott-Cattaneo-Rossi invariants of high dimensional long knots.} J.Math. Soc. Japan 73 (2021), no.3, 815--860.
\bibitem [Lon] {Lon} R. Longoni, \textit{Nontrivial classes in {$H^*({\rm Imb}(S^1,\Bbb R^n))$} from nontrivalent graph cocycles}. Int. J. Geom. Methods Mod. Phys. 1 (2004), no.5, 639--650.
\bibitem [MT] {MT} M. Mimura, H. Toda, \textit{Topology of Lie groups. I, II.} Translations of Mathematical Monographs, 91. American Mathematical Society, Providence, RI,  1991. 
\bibitem [PS] {PS} K. E. Pelatt, D. P. Sinha, \textit{A geometric homology representative in the space of knots.} Manifolds and {$K$}-theory, Contemp. Math., 682, Amer. Math. Soc., Providence, RI, 2017, 167--188.
\bibitem [Sak 1] {Sak 1} K. Sakai, \textit{Nontrivalent graph cocycle and cohomology of the long knot space.} Algebr. Geom. Topol.  8  (2008),  no. 3, 1499--1522.
\bibitem [Sak 2] {Sak 2} K. Sakai, \textit{Configuration space integrals for embedding spaces and the Haefliger invariant}. J. Knot Theory Ramifications 19  (2010),  no. 12, 1597--1644.
\bibitem [Sak 3] {Sak 3} K. Sakai, \textit{An integral expression of the first nontrivial one-cocycle of the space of long knots in $\mathbb{R}3$.} Pacific J. Math.  250  (2011),  no. 2, 407--419.
\bibitem [SW] {SW} K. Sakai, T. Watanabe, \textit{$1-$loop graphs and configuration space integral for embedding spaces.} Math. Proc. Cambridge Philos. Soc.  152  (2012),  no. 3, 497--533.
\bibitem [Sin] {Sin} D. P. Sinha, \textit{The homology of the little discs operad}. S\'eminaire et Congr\`es de Soci\'et\'e Math\'ematique de France 26 (2011), p. 255-281.
\bibitem [Sin 2] {Sin 2} D. P. Sinha, \textit{Manifold-theoretic compactifications of configuration spaces.} Selecta Math. (N.S.), 10 (2004), no. 3, 391--428.
\bibitem [Wat 1] {Wat 1} T. Watanabe, \textit{Configuration space integral for long n-knots and the Alexander polynomial.} Algebr. Geom. Topol.  7  (2007), 47--92.
\bibitem [Wat 2]{Wat 2} T. Watanabe, \textit{On Kontsevich's characteristic classes for higher dimensional sphere bundles. I. The simplest class.} Math. Z.  262  (2009),  no. 3, 683--712.
\bibitem [Wat 3] {Wat 3} T. Watanabe, \textit{On Kontsevich's characteristic classes for higher-dimensional sphere bundles. II. Higher classes.}  J. Topol.  2  (2009),  no. 3, 624--660.
\bibitem [Wat 4] {Wat 4} T. Watanabe,  \textit{Higher order generalization of Fukaya's Morse homotopy invariant of 3-manifolds I}. Invariants of homology 3-spheres. Asian J. Math. 22 (2018), no. 1, 111--180.
\bibitem [Wat 5] {Wat 5} T. Watanabe, \textit{Some exotic nontrivial elements of rational homotopy groups of $\text{Diff}(S^4)$.} arxiv.1812.02448.
\bibitem [Wat 6] {Wat 6} T. Watanabe, \textit{Theta-graph and diffeomorphisms of some 4-manifolds.} arXiv.2005.09545
\bibitem [Wei] {Wei} M. Weiss, \textit{Embeddings from the point of view of immersion theory I.} Geom. Topol.  3  (1999), 67--101.
\bibitem [Wit] {Wit} E. Witten, \textit{Quantum field theory and the Jones polynomial}. Comm. Math. Phys. 121(1989), no.3, 351-399.
\end{thebibliography}
\end{document}